\documentclass[11pt]{amsart}
\pdfoutput=1

\usepackage{amscd,latexsym,amsthm,amsfonts,amssymb,amsmath,amsxtra,}
\usepackage[mathscr]{eucal}
\renewcommand\theequation{\thesection.\arabic{equation}}

\newcommand{\BC}{{\mathbb {C}}}

\newcommand{\BQ}{{\mathbb {Q}}}
\newcommand{\BR}{{\mathbb {R}}}

\newcommand{\BZ}{{\mathbb {Z}}}

\newcommand{\CN}{{\mathcal {N}}}

\newcommand{\CW}{{\mathcal {W}}}

\newcommand{\GL}{{\mathrm{GL}}}

\newcommand{\GSp}{{\mathrm{GSp}}}

\newcommand{\Hom}{{\mathrm{Hom}}}

\newcommand{\SL}{{\mathrm{SL}}}

\newcommand{\bs}{\backslash}

\def\diag{{\rm diag}}

\newtheorem{thm}{Theorem}[section]
\newtheorem{cor}[thm]{Corollary}
\newtheorem{lem}[thm]{Lemma}
\newtheorem{prop}[thm]{Proposition}
\newtheorem {conj}[thm]{Conjecture}
\newtheorem {ques/conj}[thm]{Question/Conjecture}

\newtheorem{defn}[thm]{Definition}
\newtheorem{rmk}[thm]{Remark}

\newtheorem{exmp}[thm]{Example}

\makeatletter

\newcommand{\Rmnum}[1]{\expandafter\@slowromancap\romannumeral #1@}
\makeatother

\begin{document}
\renewcommand{\theequation}{\arabic{equation}}
\numberwithin{equation}{section}

\title{Bessel Functions and Kloosterman Integrals on $\GL(n)$}


\address{Max-Planck-Institut {f\"{u}r} Mathematik, Vivatsgasse 7, Bonn, 53111, Germany}
\email{miao@mpim-bonn.mpg.de}

\author{Xinchen Miao}
\address{School of Mathematics\\
University of Minnesota\\
Minneapolis, MN 55455, USA}
\email{miao0011@umn.edu}




\begin{abstract}

This paper will focus on the proof of local integrability of Bessel functions for $\GL(n)$ ($p$-adic case) by using the relations between Bessel functions and local Kloosterman (orbital) integrals proved in several papers of E. M. Baruch \cite{Ba03} \cite{Ba04} \cite{Ba05}, the theory of the (relative) Shalika germs established by H. Jacquet and Y. Ye in \cite{JY96} \cite{JY99} and G. Stevens' approach \cite{Ste87} on estimating certain $\GL(n)$ generalized Kloosterman sums.

\end{abstract}

\maketitle

\tableofcontents

\section{Introduction, Background and History}

The study of (classical) Bessel functions can be traced back to the 19th century. Bessel functions are first defined by D. Bernoulli and then generalized by F. Bessel. The classical Bessel functions are canonical solutions $y=y(x)$ of Bessel's differential equation
$$ x^2 \frac{d^2 y}{dx^2}+ x \frac{dy}{dx}+ (x^2- \alpha^2)y=0$$
for an arbitrary complex number $\alpha$, which is defined as the order of the Bessel function \cite{Wat95}.

Among all complex numbers, the most important cases are when $\alpha$ is an integer or half-integer. Bessel functions for integers $\alpha$ are also known as cylinder functions because they show up in the solution to Laplace's equation in cylindrical coordinates. Bessel functions for half-integers $\alpha$ appear in the solution to Helmholtz equation in spherial coordinates \cite{Wat95}. This shows the natural relations between Bessel functions and the solutions to PDEs.

In the 20th century, more connections and applications of the Bessel functions and their generalizations were found in many other fields of mathematics, in particular, analytic number theory, automorphic forms and the Langlands Program.

For example, classical Bessel functions appears natrually in the Voronoi summation formula as well as the Petersson and Kuznetsov trace formula for $\GL_2(\BR)$. These formulas have become fundamental analytic tools in attacking some deep problems in analytic number theory, most notably the subconvexity problem for automorphic $L$-functions. A version of the Voronoi's summation formula, which is not in its most general form, roughly reads as follows (See \cite[Section 1]{Qi20})
$$ \sum_{n=1}^{\infty}  \sqrt{n} \lambda_F^{+}(n) e \left( \frac{an}{c} \right) v(n)= \frac{1}{c} \cdot \sum_{\pm} \sum_{n=1}^{\infty}  \sqrt{n} \lambda_F^{\pm}(n) e \left( \mp \frac{\bar{a}n}{c} \right) \Gamma \left( \pm \frac{n}{c^2} \right).$$
In the above formula, $a,\bar{a}$ and $c$ are integers such that $(a,c)=1$ and $a \bar{a} \equiv 1 (\mod c)$, $\lambda_F^{\pm}(n)$ are certain normalized Fourier coefficients of a holomorphic or Maass cusp form $F$ for $\SL_2(\BZ)$, $v$ is a smooth weight function compactly supported on $(0,\infty)$ and $\Gamma$ is the Hankel transform of $v$,
$$ \Gamma(x)= \int_{0}^{\infty} v(y) J_F(xy) dy,\; x \neq 0.$$
Here the intgral kernel $J_F$ of Hankel transform $\Gamma$ has an expression in terms of classical Bessel functions. Moreover, under the representation theoretical viewpoint, the Hankel transform associated to classical Bessel functions is closely related to the local functional equation of the $\GL_2 \times \GL_1$ Rankin-Selberg $L$-function in the real place (See \cite[Section 17,18]{Qi20}). Hence, the Bessel functions here should have a close relation to the $\GL_2(\BR)$ Whittaker functions. Good references for $\GL(2)$ Voronoi Summation Formula and Bessel function are \cite{Cog14} and \cite{IT13}. With the representation theory closely involved, a natural question rises up. That is, how to generalize the definitions of Bessel functions to higher rank case, for example, the case of $\GL_n(\BR)$ ($n \geq 3$)? This question is partially answered in \cite{Qi20}.

On the other hand, after the milestone work of John Tate (Tate's thesis) to reformulate the functional equation of Hecke $L$-function in 1950, the local-global principle (Euler products) and the adelic languages became more and more important in Number Theory. Moreover, global Fourier coefficients of automorphic forms are factorizable because of uniquenss of global Whittaker models. As an analogy to the real case, it is also natural to believe that the Bessel functions for $\GL_n(\BQ_p)$ should have a close relation to the $\GL_n(\BQ_p)$ local Whittaker functions. Therefore, it is also necessary to deveolop a similar theory of Bessel functions over non-archimedean local fields, in other words, $p$-adic local fields (For example, $\BQ_p$). Actually, we will see in the next section that using the definition of Whittaker functions (model), we can give the rigorous definition of Bessel functions in the non-archimedean local field case.

In this paper, we will focus on the Bessel functions in the case of $p$-adic local fields. More specifically, we want to understand some important properties (local integrability) of Bessel functions for $\GL_n(\BQ_p)$ ($n \geq 2$) by using the theory of (relative) Shalika germs developed by H. Jacquet and Y. Ye \cite{JY96} \cite{JY99} in 1990s relevant to their study of base change relative trace formulae.

The local integrability property is important to the understanding of general Bessel functions. The local integrability of $\GL(2)$ Bessel functions is well-known by the analytic properties of $\GL(2)$ Whittaker functions (see \cite{So84} and \cite{Ba97}). In \cite{Ba04}, M. Baruch proved that the Bessel functions for $\GL(3)$ over $p$-adic fields are locally integrable. However, the local integrability for Bessel functions for $\GL(n)$ ($n \geq 4$) is still an open problem in the subject. In this paper, we will keep our eyes on this problem and discuss some progress on it by using the methods of (relative) Shalika germs defined in \cite{JY96} and \cite{JY99}. More explicitly, the main result of this paper is that the Bessel functions for $\GL(n)$ over the $p$-adic field $\BQ_p$ are locally integrable (See Theorem \ref{inte}).

The outlines of our paper are as follows: In Section 2, following the paper \cite{Ba05}, we will give the definition of Bessel functions in the case of $\GL_n(\BQ_p)$. In Section 3, applying the results in \cite{Ba05}, \cite{JY96} and \cite{JY99}, we can reduce the proof of local integrability of Bessel functions (Theorem \ref{inte}) to finding a non-trivial upper bound for the generalized Kloosterman sums in the case of $\GL_n(\BQ_p)$. The reduction steps are given in Section 3. In Section 4, we will give a brief review of G. Stevens' result \cite{Ste87} which gives an effective method to yield an upper bound for the generalized Kloosterman sums. The definition of generalized Kloosterman sums will also be given in Section 4. In Section 5, we follow the method of Stevens to estimate Kloosterman sums for $\GL_n(\BQ_p)$ and prove a non-triival bound (See Theorem \ref{thm: w_n}). Finally, we complete the proof of the local integrability of $\GL(n)$ Bessel functions (Theorem \ref{inte}) in Section 6. In Section 7, we will give some applications, propositions and corollaries of Theorem \ref{inte}. In the Appendix, we will give a better bound for the Kloosterman sums estimation on $\GL_4(\BQ_p)$ after a more careful and delicate computation (See Theorem \ref{thm: w_8}).

Recently, from email communication, we knew that V. Blomer and S. H. Man proved a similar result (See \cite[Corollary 1]{BM22}) as our Theorem \ref{thm: w_n}, which also gives a non-trivial bound for the generalized $\GL_n(\BQ_p)$ Kloosterman sums attached to the longest Weyl element $w_{G_n}$. Our methods are related, but different. Moreover, we also have different applications. We mainly focus on the local integrability problem of Bessel functions on $\GL_n(\BQ_p)$, while \cite{BM22} has a nice application on the Sarnak density conjecture \cite{Sar90}.

\section{Bessel functions for $\GL(n)$}

In Section 2, we will give the definition of Bessel functions for $\GL_n(\BQ_p)$ by using the uniqueness of Whittaker functional in the $p$-adic case.

Let $G_n=\GL_n(F)$, where $F$ is a $p$-adic local field, i.e. a finite extension of $\BQ_p$. Throughout this paper, we consider $F=\BQ_p$. Hence, $G_n=\GL_n(F)=\GL_n(\BQ_p)$. Let $B$ be the Borel subgroup of upper triangular matrices, $T$ the subgroup of diagonal matrices and $N$ the subgroup of upper unipotent matrices. Let $K:=\GL_n(\BZ_p)$ be the maximal compact open subgroup of $G_n$. Let $\psi$ be a non-degenerate character of $N$, which is of the form
$$ \psi(n):= \xi \left( \sum_{i=1}^{n-1}  n_{i,i+1} \right)$$
for $n=(n_{i,j}) \in N$, and $\xi$ is the standard nontrivial additive character of $\BQ_p$ as in \cite[Section 1]{Ste87}.
For $G_n$, we define the normalizer of $T$ in $G_n$ to be
$$ N_{G_n}(T):= \{ g \in G_n:\; ghg^{-1} \in T \; for \; all \; h \in T \}.$$
The Weyl group $W_{G_n}$ is defined as $W_{G_n}:= N_{G_n}(T)/T$. The Weyl group is a finite group and is isomorphic to symmetric group on $n$ letters $S_n$.
Let
$w_{G_n}$ be the longest Weyl element in the Weyl group $W_{G_n}$, which can be written as
$$ w_{G_1} =1, \;\; w_{G_n}= \begin{pmatrix} 0 & 1 \\ w_{G_{n-1}} & 0 \end{pmatrix},$$
i.e. the $n \times n$ permutation matrix whose anti-diagonal entries are $1$. 

Throughout this paper, we fix the (normalized) Haar measure with the volumes of $K$ and $N \cap K$ both equal to one.

We recall the definition of Whittaker functionals over a $p$-adic field. Let $(\pi,V)$ be a smooth irreducible representation of $G_n$. A Whittaker functional $L$ is a linear functional on $V$ such that $L( \pi(n)v)= \psi(n)L(v)$ for all  $n \in N$ and $v \in V$. The following well-known theorem of Whittaker functionals is proved by J. Shalika \cite{Sha74}.
\begin{thm} [\it{Uniqueness of Whittaker functionals}] \label{uniWhi}
Let $(\pi,V)$ be an irreducible smooth representation of $G_n$. Then the space of Whittaker functionals $L$ has dimension at most equal to one.
\end{thm}
In other words, the Whittaker functional $L$ is unique up to scalar. If this Whittaker functional $L$ is non-zero, the representation $\pi$ is called generic. For a non-zero Whittaker functional $L$, we define that $W_v(g):= L(\pi(g)v),\;v \in V,\;g \in G_n$ and let $G_n$ act on the space of these functions by right translations. That is, if $g_1 \in G_n$ and $W$ is a function on $G_n$ then we define $(\rho(g_1)W)(g)=W(gg_1)$ for $g \in G_n$. The map $v \rightarrow W_v$ gives a realization of $\pi$ on a space of Whittaker functions which satisfy $W(ng)=\psi(n)W(g)$ for all $n \in N$ and $g \in G_n$. We denote this space by $\CW(\pi,\psi)$ and call it the Whittaker model of $\pi$.


Let $N_1 \subseteq N_2 \subseteq N_3 \subseteq \cdots$ be a filtration of $N$ with compact open subgroups $N_i,i=1,2,\cdots$ such that $N= \cup_{i=1}^{\infty} N_i$. We denote this filtration by $\CN$. Let $f:N \rightarrow \BC$ be a locally constant function.

\begin{defn} \label{sta1}
We define the stable integral
$$ \int_N^{\CN} f(n)dn:= \lim_{m \rightarrow +\infty} \int_{N_m} f(n)dn $$
if this limit exists. If the limit exists, we say that the stable integral is convergent.
\end{defn}

The following theorem is proved in \cite{Ba05}.

\begin{thm} \cite{Ba05}  \label{sta2}
Let $\CN:= \{N_i,i \geq 1 \}$ be a filtration of $N$ for $G_n$ as above. Let $g \in Bw_{G_n} B$ and $W \in \CW(\pi,\psi)$. Then the (stable) integral
$$ \int_N^{\CN} W(gn) \psi^{-1}(n)dn$$
is convergent. Moreover, the value for this integral is independent on the choice of filtration $\CN$.
\end{thm}

Let $g \in Bw_{G_n} B$ and we define the linear functional $L_g: V \rightarrow \BC$ by
$$ L_g(v):= \int_N^{\CN} W_v(gn) \psi^{-1}(n)dn.$$
Since $L_g(\pi(n)v)= \psi(n)L_g(v)$ for all $n \in N$, we see that $L_g$ is also a Whittaker functional, hence it follows from Theorem \ref{uniWhi} that there exists a scalar $j_{\pi,\psi}(g)$ such that
$$ L_g(v)= j_{\pi,\psi}(g)L(v)$$
for all $v \in V$. We call $j_{\pi}=j_{\pi,\psi}$ the Bessel function of $\pi$. The Bessel function $j_{\pi}(g)$ is independent on the choice of the Whittaker functional $L$ which is unique up to scalar multiplication (See Theorem \ref{uniWhi}). We see that the Bessel function $j_{\pi}$ is defined on the open Bruhat cell $Bw_{G_n} B$ of $G_n$. Moreover we know that $j_{\pi}(g)$ satisfies
$$ j_{\pi}(n_1gn_2)= \psi(n_1 n_2)j_{\pi}(g)$$
for all $n_1,n_2 \in N$ and $g \in B w_{G_n} B$. The value of $j_{\pi}$ is determined by its values on the set $w_{G_n} T$ and the Bessel function $j_{\pi}$ is locally constant on the set $B w_{G_n} B$ (See \cite{Ba05}). If $g \in G_n - B w_{G_n} B$, it is defined that $j_{\pi}(g) \equiv 0$.

As in the theory of distribution characters of smooth irreducible admissible representations of $G_n$ (See \cite{HC70} and \cite{HC99}), the Bessel function $j_{\pi}$ is expected to be locally integrable on $G_n$. We prove this result for all $n \geq 2$, which states as follows (the main result of our paper):

\begin{thm} \label{inte}
The Bessel function $j_{\pi}(g)$ is locally integrable on $G_n = \GL_n(\BQ_p)$.
\end{thm}

Harish-Chandra's proof of the local integrability of the distribution characters (See \cite{HC70} and \cite{HC99}) depends on certain relations between the asymptotics of the character and certain orbital integrals. In the Bessel function case, we have a similar result for such kind of relations.

For $\phi \in C_c^{\infty}(G_n)$ and $g \in Bw_{G_n} B$, we can define the orbital integral as follows:
$$ J_{\phi,\psi}(g):= \int_{N \times N} \phi(n_1 g n_2) \psi^{-1}(n_1 n_2) dn_1 dn_2.$$
It follows from the results in \cite{JY96} and \cite{JY99} that the above orbital integral converges absolutely and defines a locally constant function when $g \in B w_{G_n} B$. If $g \in G_n- B w_{G_n} B$, it is defined that $J_{\phi,\psi}(g)=0$.

Note that $w_{G_n} N w_{G_n}= \bar{N}$, where $\bar{N}$ means the opposite of the unipotent radical $N$. We can consider a similar orbital integral.

Let $f \in C_c^{\infty}(G_n)$ and $g \in \bar{N}B$, we define the orbital integral similarly:
$$ I_{f, \psi}(g):= \int_{N \times N} f(n_1^{t}gn_2) \psi^{-1}(n_1n_2)dn_1dn_2.$$
Here $n_1^t$ means the transpose of the unipotent radical $n_1$. From our assumptions on the additive character $\psi$, it is clear that $\psi(n)=\psi(w_{G_n} n^t w_{G_n})$ for all $n \in N$. The above integral converges absolutely and defines a locally constant function on $\bar{N}B$. Moreover, if $g \in G_n- \bar{N}B$,  it is defind that $I_{f,\psi}(g)=0$ which is the same way as before. We note that $I_{f,\psi}(g)=J_{\phi,\psi}(w_{G_n} g)$, where $f(w_{G_n} g)=\phi(g) \in C_c^{\infty}(G_n)$ and $f(g) \in C_c^{\infty}(G_n)$.

We may also consider the following orbital integral.

Let $f \in C_c^{\infty}(G_n)$ and $g \in \bar{N}B$. Let $Z_n$ be the center of $G_n$ and $\omega$ be a quasicharacter of $Z_n$. We define:
$$ I_{f, \omega, \psi}(g):= \int_{N \times Z_n \times N} f(n_1^{t}zgn_2) \omega^{-1}(z) \psi^{-1}(n_1n_2)dn_1dzdn_2.$$
The above integral converges absolutely and defines a locally constant function on $\bar{N}B$. Similarly, we extend the orbital integral $I_{f, \omega, \psi}(g)$ to a function on $G_n$ by setting $I_{f,\omega,\psi}(g)=0$ when $g \in G_n- \bar{N}B$.

Moreover, we define that
$$ J_{f, \omega, \psi}(g):= \int_{N \times Z_n \times N} f(n_1zgn_2) \omega^{-1}(z) \psi^{-1}(n_1n_2)dn_1dzdn_2.$$
This orbital integral converges absolutely and defines a locally constant function when $g \in B w_{G_n} B$. If $g \in G_n- B w_{G_n} B$, it is defined that $J_{f, \omega, \psi}(g)=0$. We note that $I_{f, \omega, \psi}(g)=J_{\phi, \omega, \psi}(w_{G_n} g)$, where $f(g)=\phi(w_{G_n} g) \in C_c^{\infty}(G_n)$ and $\phi(g) \in C_c^{\infty}(G_n)$.

The following theorem is proved in \cite{Ba05}.

\begin{thm}  \label{transfer}
Let $\pi$ be an smooth irreducible representation of $G_n$. Let $x \in G_n$, then there exists a neighbourhood $U_x$ of $x$ in $G_n$ and a function $\phi \in C_c^{\infty}(G_n)$ such that
$$ j_{\pi,\psi}(g)= J_{\phi,\omega_{\pi},\psi}(g) $$
for all $g \in U_x$.  Here $\omega_{\pi}$ is the central character of $\pi$.
\end{thm}
From this theorem, we know that the Bessel function $j_{\pi, \psi}(g)$ is a locally integrable function on $G_n$ if  the orbital integral $J_{\phi, \omega_{\pi}, \psi}(g)$ is locally integrable as a function on $G_n$ for every $\phi \in C_c^{\infty}(G_n)$.

Similarly, we see that  $j_{\pi, \psi}(g)$ is a locally integrable function on $G_n$ if the orbital integral $I_{f,\omega_{\pi}, \psi}(g)$ is locally integrable as a function on $G_n$ for every $f \in C_c^{\infty}(G_n)$.


\begin{rmk}
Theorem \ref{uniWhi}, Definition \ref{sta1}, Theorem \ref{sta2} and Theorem \ref{transfer} in Section 2 hold for general $\GL_n(F)$, where $F$ is a general $p$-adic local field (See \cite{Sha74} and \cite{Ba05}). We also hope that Theorem \ref{inte} can be generalized from $\GL_n(\BQ_p)$ to general $\GL_n(F)$. We may discuss such generalization in our future work.
\end{rmk}


\section{Reduction of Proof}


In Section 3 and 4, we will explain the idea of the proof for Theorem \ref{inte}. By applying the results in \cite{Ba05}, \cite{JY96} and \cite{JY99}, we will reduce the proof of local integrability of Bessel functions to finding a non-trivial upper bound for the generalized Kloosterman sums in the case of $\GL_n(\BQ_p)$.

\subsection{Reduction Step 1: Asymptotic behaviour of the orbital integrals}

If for every $f \in C_c^{\infty}(G_n)$, $I_{f,\omega_{\pi}, \psi}(g)$ is locally integrable as a function on $G_n$, then the Bessel function $j_{\pi,\psi}$ is locally integrable in $G_n$. This is given by Theorem \ref{transfer} in Section 2. In order to prove that $I_{f, \omega_{\pi}, \psi}(g)$ is locally integrable as a function on $G_n$, it is important to study the asymptotic behaviour of the orbital integral when $g$ approaches to the boundary of the domain. From Bruhat decomposition, we can further assume that $g \in T$.

\subsection{Reduction Step 2: An estimation and comparison}

For $g \in G_n:=\GL_n(\BQ_p)$, let $\Delta_r(g)$, $1 \leq r \leq n$, be the determinant of the sub-matrix $g_{r,r}$ of $g \in G_n$ formed with the first $r$ rows and the first $r$ columns of $g$. Hence $\Delta_n(g)=\det g$. We define $\Delta: G_n \rightarrow \BR_{\geq 0}$ by
$$ \Delta(g):= \left \vert \frac{(\Delta_1(g))^2 \cdot (\Delta_2(g))^2 \cdots (\Delta_{n-1}(g))^2}{(\Delta_n(g))^2} \right \rvert. $$
It is known that $g \notin \bar{N}B$ (an open Bruhat cell) if and only if we have $\Delta(g)=0$. Moreover
$$ \Delta(\diag(a_1, a_2, \cdots, a_n))= \left \vert a_1^{2(n-2)}\cdot a_2^{2(n-3)}\cdots a_{n-2}^{2} \cdot a_n^{-2} \right \rvert.$$
Let $\delta$ be the modulus character of $B$, we have
$$ \delta(\diag(a_1, a_2, \cdots, a_n))= \left \vert a_1^{n-1}\cdot a_2^{n-3}\cdots a_{n-1}^{3-n} \cdot a_n^{1-n} \right \rvert.$$
Hence, we have
$$ \Delta(a)=\delta(a) \cdot \left \vert a_1 a_2 \cdots a_n \right \rvert^{n-3}= \delta(a) \cdot \vert \Delta_n(a)\rvert^{n-3}.$$

We first recall the main Theorem 0.3 proved by Dabrowski and Reeder \cite{DR98}.

For $i,j \in \{1,2,\cdots,n \}, i \neq j$, we let $\alpha_{i,j}: T \rightarrow \BQ_p^{\times}$ be the functions defined by
$$ \alpha_{i,j}( \diag(a_1,a_2,\cdots,a_n)):= \frac{a_i}{a_j}.$$
Let $\Phi= \{ \alpha_{i,j} \}, 1 \leq i,j \leq n$ be the root system of $G_n$ and let $\check{\Phi}= \{ \check{\alpha}: \alpha \in \Phi \}$. We have $\Phi= \Phi^{+} \bigsqcup \Phi^{-}$, where $\Phi^{+}= \{ \alpha_{i,j}: 1 \leq i<j \leq n \}$ is the set of positive roots and $\Phi^{-}= \{ \alpha_{i,j}: 1 \leq j<i \leq n \}$ is the set of negative roots. Let $\Delta= \{ \alpha_{i,i+1}: 1 \leq i \leq n-1 \}$ be the set of simple roots. Similarly, we write $\check{\Phi}= \check{\Phi}^{+} \bigsqcup \check{\Phi}^{-}$.

Let $f \in C_c^{\infty}(G_n)$ and let $g \in \bar{N}B$. We define the following orbital integral:
$$ O_f(g):= \int_{N \times N} f(n_1^t g n_2) dn_1 dn_2.$$
The convergence of this integral follows from \cite{JY96} and \cite{JY99}. In 1998, Dabrowski and Reeder \cite{DR98} studied this orbital integral when $f=f_0$ is the characteristic function of the maximal compact subgroup $K$. We recall that the Haar measure is normalized with the volumes of $K$ and $N \cap K$ both equal to one. We let $X=X(T):=\Hom (T,\BQ_p^{\times})$ be the group of $\BQ_p$-rational characters. Let $\check{X}=\check{X}(T):=\Hom (\BQ_p^{\times},T)$ be the set of co-characters. For each $a \in T$, there exists a unique $\lambda_a \in \check{X}$ such that $a= a_K \lambda_a(\omega)$ where $a_K \in T_K:=T \cap K$ and $\omega$ is the uniformizer of $\BQ_p$. The following result is the main theorem which is proved by Dabrowski and Reeder \cite{DR98}. 
\begin{thm} \label{DR1998} \cite[Theorem 0.3]{DR98}
Let $f_0$ be the characteristic function of $K$. Then $O_{f_0}(a)=0$ if $\lambda= \lambda_a \notin \BZ_{\geq 0} \check{\Phi}^{+}$, i.e., $\lambda_a$ is not a nonnegative integral linear combination of positive coroots. If $\lambda_a$ is such a linear combination, we can write
\begin{equation} \label{linear}
\lambda_a= \sum_{\beta \in \check{\Phi}^{+}} m_{\beta} \beta= \sum_{1 \leq i<j \leq n} m_{i,j} \check{\alpha}_{i,j},
\end{equation}
for $m_{i,j} \geq 0$ and we also write $\overline{m}=(m_{\beta})_{\beta \in \check{\Phi}^{+}}=(m_{i,j})_{1 \leq i<j \leq n}$. Then we have
\begin{equation}
O_{f_0}(a)= \Delta^{-\frac{1}{2}}(a) \times \sum_{\overline{m}} \left( 1- \frac{1}{p} \right)^{\kappa(\overline{m})},
\end{equation}
where $\kappa(\overline{m})$ is the number of strictly positive coordinates of $\overline{m}$ and $\overline{m}$ runs over all possible decompositions for $\lambda_a$.
\end{thm}
Above Theorem \ref{DR1998} holds for general $G(F)$, where $F$ is a general $p$-adic local field and $G$ is a connected split reductive group.

Using an idea similar to that in \cite[Section 3]{Ba04}, we can prove the following theorem
\begin{thm}  \label{delta}
$\Delta^{-\frac{1}{2}+\epsilon}$ is locally integrable as a function on $G_n=\GL_n(\BQ_p)$ for every $\epsilon>0$.
\end{thm}

\begin{proof}

If $\epsilon \geq \frac{1}{2}$, then $\Delta^{-\frac{1}{2}+\epsilon}$ is a continuous and smooth function on $G_n$. Therefore, there is nothing to prove.

If $0< \epsilon < \frac{1}{2}$, then $\Delta^{-\frac{1}{2}+\epsilon}$ is not well-defined if $g \notin \overline{N}B= \overline{N}TN$. However, since $\overline{N}TN$ is Zariski open dense in $G_n$, the complement subset of $\overline{N}TN$ in $G_n$ is closed and of measure zero. So we can define $\Delta^{-\frac{1}{2}+ \epsilon}(g):=0$ if $g \notin \overline{N}TN$. This will not affect the local integrability of the function $\Delta^{-\frac{1}{2}+\epsilon}$. Note that if $0<\epsilon<\frac{1}{2}$, then $\Delta^{-\frac{1}{2}+\epsilon}$ is not a continuous and smooth function on $G_n$.

We note that it is sufficient to prove that
$$ \int_{G_n} \Delta^{-\frac{1}{2}+\epsilon}(g)f(g)dg< +\infty $$
for every characteristic function $f$ of $K g_0$, where $K=\GL_n(\BZ_p)$ is the maximal compact open subgroup of $G_n$ and some fixed point $g_0 \in G_n$. From the Iwasawa decomposition, we can write $g_0=k_0b_0$ for some $b_0 \in B$ and $k_0 \in K$. Hence it is enough to show that $ \int_G \Delta^{-\frac{1}{2}+\epsilon}(g)f(g)dg< +\infty$ holds for $f=\rho_r(b_0^{-1})f_0$ where $\rho_r$ is the right translation and $f_0$ is the characteristic function of $K$.

By writing $b_0=a_0n_0$ for $a_0 \in T $ and $n_0 \in N$ we have
\begin{equation*}
\begin{aligned}
\int_{G_n} \Delta^{-\frac{1}{2}+\epsilon}(g)(\rho_r(b_0^{-1})f_0)(g)dg &= \int_{G_n} \Delta^{-\frac{1}{2}+\epsilon}(g)f_0(gb_0^{-1})dg\\
&= \int_{G_n} \Delta^{-\frac{1}{2}+\epsilon}(gb_0)f_0(g)dg\\
&= \Delta^{-\frac{1}{2}+\epsilon}(a_0n_0) \cdot \int_{G_n} \Delta^{-\frac{1}{2}+\epsilon}(g)f_0(g)dg\\
&= \Delta^{-\frac{1}{2}+\epsilon}(a_0) \cdot \int_{G_n} \Delta^{-\frac{1}{2}+\epsilon}(g)f_0(g)dg.
\end{aligned}
\end{equation*}
Hence, it is enough to prove that $\int_{G_n} \Delta^{-\frac{1}{2}+\epsilon}(g)f(g)dg< +\infty$ for $f=f_0$. Now applying the invariance properties of $\Delta$ and writing $dg=\delta(a)dn_1dadn_2=\Delta(a)dn_1 da dn_2$ on the Zariski open dense subset that consists of elements of the form $g= n_1^t a n_2$ where $n_1, n_2 \in N$ and $a \in T$ (Note that since $g= n_1^t a n_2 \in K$, we have $\vert \det(a) \rvert=\vert \Delta_n(a) \rvert=1$, which gives that $\delta(a) = \Delta(a)$), we can get the following
\begin{equation*}
\begin{aligned}
\int_{G_n} \Delta^{-\frac{1}{2}+\epsilon}(g)f_0(g)dg &= \int_{\overline{N} \times T \times N} \Delta^{-\frac{1}{2}+\epsilon}(g)f_0(g)dg\\
&= \int_{N \times T \times N} \Delta^{-\frac{1}{2}+\epsilon}(n_1^t a n_2)f_0(n_1^t a n_2) \Delta(a)dn_1 da dn_2\\
&= \int_T \Delta^{\frac{1}{2}+\epsilon}(a) \cdot O_{f_0}(a) da,
\end{aligned}
\end{equation*}
where
$$ O_{f_0}(a):= \int_{N \times N} f_0(n_1^t a n_2)dn_1dn_2.$$
Here we can change the order of integrations because $\Delta^{-\frac{1}{2}+\epsilon} \geq 0$ and $f_0 \geq 0$. The Haar measures on $G_n$, $N$ and $T$ are all normalized such that the volumes of $K$ and $N \cap K$ equal to one.

We recall that $X=X(T)=\Hom (T,\BQ_p^{\times})$ be the group of $\BQ_p$-rational characters. Let $\check{X}=\check{X}(T)=\Hom (\BQ_p^{\times},T)$ be the set of co-characters. For each $a \in T$, there exists a unique $\lambda_a \in \check{X}$ such that $a= a_K \lambda_a(\omega)$ where $a_K \in T_K:=T \cap K$ and $\omega$ is the uniformizer of $\BQ_p$. We can see that
\begin{equation*}
\begin{aligned}
\int_T \Delta^{\frac{1}{2}+\epsilon}(a) \cdot O_{f_0}(a)da &= \sum_{\lambda \in \check{X}} \int_{(T \cap K) \cdot \lambda(\omega)} \Delta^{\frac{1}{2}+\epsilon}(a) \cdot O_{f_0}(a) da\\
&= \left ( \int_{T \cap K} da \right ) \cdot \left ( \sum_{\lambda \in \check{X}} \Delta^{\frac{1}{2}+\epsilon}(\lambda(\omega)) \cdot O_{f_0}(\lambda(\omega)) \right ).
\end{aligned}
\end{equation*}
It is easy to see that
$$\int_{T \cap K} da = Vol(T \cap K)= \left(1-\frac{1}{p} \right)^n<1.$$
From Theorem \ref{DR1998} (\cite{DR98}), we know that the sum over $\lambda \in \check{X}$ takes place for $\lambda$ of the form
$$ \lambda= \sum_{i=1}^{n-1} m_i \cdot \check{\alpha}_{i,i+1},$$
where $ \check{\alpha}_{i,i+1}$ are positive simple coroots and $m_i$ are nonnegative integers. Moreover, we have the following estimation:
$$ O_{f_0}(\lambda(\omega)) = \Delta^{-\frac{1}{2}}(\lambda(\omega)) \times \sum_{\overline{m}} \left( 1- \frac{1}{p} \right)^{\kappa(\overline{m})} < p^{\sum_{i=1}^{n-1} m_i} \times R(a).$$
Here $\Delta(a)= \Delta(\lambda_a(\omega))= \Delta(\lambda(\omega))=p^{-2 \cdot \sum_{i=1}^{n-1} m_i}$ and $R(a):=R(\lambda_a)$ is the number of possibilities of writing $\lambda_a$ as in Equation \eqref{linear} (See Theorem \ref{DR1998}).  We let $1 \leq i \leq n-1$, the multiplicity of the simple coroot $\check{\alpha}_{i,i+1}$ in the coroot $\check{\alpha}_{s,t}$ ($1 \leq s<t \leq n$) equals to $i(n-i)$.

Now by direct estimation, we have
$$ R(a) \leq \prod_{i=1}^{n-1} (m_i+1)^{i(n-i)}.$$
We compute $\sum_{i=1}^{n-1} i(n-i)< n \times \sum_{i=1}^{n-1} i= \frac{(n-1)n^2}{2}< \frac{n^3}{2}$, which gives that
$$ R(a) \leq \prod_{i=1}^{n-1} (m_i+1)^{i(n-i)} < \left ( \sum_{i=1}^{n-1} m_i +n \right )^{\frac{n^3}{2}}.$$

From above discussion, we see that
$$ O_{f_0}(\lambda(\omega)) < p^{\sum_{i=1}^{n-1} m_i} \times  \left ( \sum_{i=1}^{n-1} m_i+n \right )^{\frac{n^3}{2}}.$$
Note that
$$ \Delta(a)= \Delta(\lambda_a(\omega))= \Delta(\lambda(\omega))= p^{-2 \cdot \sum_{i=1}^{n-1} m_i},$$
therefore we have
\begin{equation*}
\begin{aligned}
& \Delta^{\frac{1}{2}+\epsilon}(\lambda(\omega)) \cdot O_{f_0}(\lambda(\omega)) \\
< & p^{- 2 \epsilon \cdot \sum_{i=1}^{n-1} m_i} \times \left ( \sum_{i=1}^{n-1} m_i +n \right )^{\frac{n^3}{2}}.
\end{aligned}
\end{equation*}
Hence our integral is controlled by
\begin{equation*}
\begin{aligned}
1+ \sum_{m_1 \geq 0} \sum_{m_2 \geq 0} \cdots \sum_{m_{n-1} \geq 0} p^{- 2 \epsilon \cdot \sum_{i=1}^{n-1} m_i} \times \left (\sum_{i=1}^{n-1} m_i+n \right )^{\frac{n^3}{2}}
\leq &\quad 1+ \sum_{k=1}^{+\infty} \frac{(k+n)^n \times (k+n)^{\frac{n^3}{2}}}{p^{2 \epsilon k}} \\
\leq &\quad 1+ \sum_{k=1}^{+\infty} \frac{(k+n)^{n^3}}{p^{2 \epsilon k}}
\end{aligned}
\end{equation*}
which is finite when $\epsilon>0$ (Note that $n$ is a fixed positive integer). This proves that the function $\Delta^{-\frac{1}{2}+\epsilon}$ is locally integrable for every $\epsilon>0$.

\end{proof}

Above Theorem \ref{delta} holds for general $\GL_n(F)$, where $F$ is a general $p$-adic local field.

According to Theorem \ref{delta} and Theorem \ref{transfer}, the local integrability of Bessel function (See Theorem \ref{inte}) can be reduced to prove the following conjecture.

\begin{conj}  \label{conj1}
Fix $f \in C_c^{\infty}(G_n)$. Then $\vert I_{f, \omega_{\pi}, \psi}(g) \Delta^{\frac{1}{2}-\delta}(g) \rvert $ is bounded on compact sets in $G_n$ for some given $\delta>0$.
\end{conj}

\begin{rmk}
In $\GL(3)$ case, the Conjecture \ref{conj1} was proved in \cite{Ba04} for $\delta=\frac{1}{8}>0$. Hence the local integrability for $\GL(3)$ Bessel functions is proved.
\end{rmk}

Similarly, we can establish the following conjecture:
\begin{conj}  \label{conj11}
Fix $f \in C_c^{\infty}(G_n)$. Then $\vert I_{f, \psi}(g) \Delta^{\frac{1}{2}-\delta}(g) \rvert $ is bounded on compact sets in $G_n$ for some given $\delta>0$.
\end{conj}

We have the following quick proposition.
\begin{prop}  \label{prop1}
Suppose that Conjecture \ref{conj11} holds for general $n$, then Conjecture \ref{conj1} also holds in general.
\end{prop}

\begin{proof}
Since $f \in C_c^{\infty}(G_n)$, we let $Q_1$ be the support of $f$. Since $Q_1$ is compact, it follows that $\vert \det(g) \rvert$ have both lower and upper bounds for $g \in Q_1$. Hence the support of the orbital integral $I_{f,\psi}(g)$ is also on a set on which the determinant is bounded.

Now let $Q_2$ be a compact set in $G_n$. We will have to show that $\vert I_{f, \omega_{\pi}, \psi}(g) \Delta^{\frac{1}{2}-\delta}(g) \rvert $ is bounded on $Q_2$ in $G_n$ for the fixed $\delta>0$ in Conjecture \ref{conj11}. If $g \in Q_2$, $z \in Z_n$ and $gz$ is in the support of the orbital integral $I_{f,\psi}(g)$, we have $\det(gz)=\det(g)\times \det(z)$ is in some fixed compact set in $F^{\times}$. Hence we know that $z$ is in a fixed compact set $P$ in $Z_n$ which is independent on the choice of $g \in Q_2$.

We let $C_1:= \max_{z \in P} \left( \vert \omega_{\pi}(z)^{-1} \rvert \right)>0$ and $g \in Q_2$. By Conjecture \ref{conj11}, there exists a constant $C_2>0$ such that
$$ \vert I_{f,\psi}(gz) \rvert \leq C_2 \cdot \Delta^{-\frac{1}{2}+ \delta}(gz) $$
for all $g \in Q_2$ and $z \in P$. Hence, if $g \in Q_2$, then we have
\begin{equation}
\begin{aligned}
\left \vert I_{f,\omega_{\pi},\psi}(g) \right \rvert &= \left \vert \int_{Z_n} I_{f,\psi}(gz) \omega_{\pi}^{-1}(z) dz \right \rvert \\
&= \left \vert \int_{P} I_{f,\psi}(gz) \omega_{\pi}^{-1}(z) dz \right \rvert \\
& \leq C_1 C_2 \times \int_{P} \left \vert \Delta^{-\frac{1}{2}+ \delta}(gz) \right \rvert dz.
\end{aligned}
\end{equation}
Note that $I_{f,\omega_{\pi},\psi}(g)=0$ if $g \notin \bar{N}B$. Now we may write $g=n_1^t a n_2$ for $n_1, n_2 \in N$ and $a \in T$. Hence we have $\Delta^{-\frac{1}{2}+ \delta}(gz)= \Delta^{-\frac{1}{2}+ \delta}(n_1^t a n_2 z)=\Delta^{-\frac{1}{2}+ \delta}(n_1^t az n_2)=\Delta^{-\frac{1}{2}+ \delta}(az)=\Delta^{-\frac{1}{2}+ \delta}(a) \times \Delta^{-\frac{1}{2}+ \delta}(z)=\Delta^{-\frac{1}{2}+ \delta}(g) \times \Delta^{-\frac{1}{2}+ \delta}(z)$. Since $P$ is compact, there exists $C_3>0$ such that
$$ \int_{P} \left \vert \Delta^{-\frac{1}{2}+ \delta}(z) \right \rvert dz < C_3.$$
Therefore we have
$$ \left \vert I_{f,\omega_{\pi},\psi}(g) \right \rvert < C_1 C_2 C_3 \left \vert \Delta^{-\frac{1}{2}+ \delta}(g) \right \rvert.$$
This is equivalent to
$$ \left \vert I_{f, \omega_{\pi}, \psi}(g) \Delta^{\frac{1}{2}-\delta}(g) \right \rvert < C_1 C_2 C_3,$$
which is exactly Conjecture \ref{conj1}.

\end{proof}

\subsection{Reduction Step 3: (Relative) Shalika germs}

In order to study the asymptotic behaviour of the orbital integrals, we need to introduce the (relative) Shalika germs of Jacquet-Ye \cite{JY96} \cite{JY99}.

We define the generalized orbital integrals of a function $f \in C_c^{\infty}(G_n)$, that is, the functionals
$$ I(g,f):= \int f(n_1^t g n_2) \psi^{-1}(n_1 n_2)dn_1 dn_2$$
for $g \in G_n$. This is an extension of the previous orbital integral $I_{f,\psi}(g)$. The integration domain is explicitly given as follows. By the Bruhat decomposition, we can write $g=wa$ where $w \in W_{G_n}$ and $a \in T$. In order for the integral to make sense we must assume that $\psi(n_1 n_2)=1$ if $n_1^t wa n_2= wa$. However, this assumption is not always true for all the Weyl elements. A Weyl element $w$ is said to be relevant if this condition is satisfied. The integral is then taken over the quotient of $N \times N$ by the subgroup of elements $(n_1,n_2)$ of $N \times N$ satisfying $n_1^t wa n_2= wa$. If $w= I_n \in G_n$, we have $I(g,f)=I(wa,f)=I(a,f)=I_{f,\psi}(a)=I_{f,\psi}(g)$, where the orbital integral $I_{f,\psi}$ is defined in Section 2. We let an element $a \in T$ operate on $G_n$ by $g \rightarrow ga$.

Let $W_{G_n}$ be the Weyl group of $G_n$ and $R_{G_n}$ be the subset of relevant elements of $W_{G_n}$. For each positive integer $r$ we denote by $w_{G_r}$ the $r \times r$ permutation matrix whose anti-diagonal entries are 1. If $w \in R_{G_n}$ then $w$ has the form
$$ w= \begin{pmatrix}
w_{G_{n_1}} & 0 & 0 & \cdots & 0 \\ 0 & w_{G_{n_2}} & 0 & \cdots & 0 \\ \cdots & \cdots & \cdots & \cdots & \cdots \\ 0 & 0 & 0 & \cdots & w_{G_{n_r}}
\end{pmatrix}
$$
for a suitable $r$-tuple of positive integers $(n_1,n_2,\cdots,n_r)$ with $\sum_{i=1}^r n_i=n$. This is given in \cite{JY96}. Hence, there are $2^{n-1}$ relevant Weyl elements in $\GL(n)$ case (We may label them as $e=I_n,w_1,w_2, \cdots$ $\cdots,w_{2^n-2},w_{G_n}$).

A parabolic subgroup $P_n$ of $G_n$ is called standard if it contains $TN=B$. Then $P_n=M_nU_n$ where $M_n$ is the unique Levi subgroup of $P_n$ containing $T$ and $U_n$ is the unipotent radical of $P_n$. A Levi subgroup of this type is said to be standard. It is said to be of type $(n_1,n_2,\cdots,n_r)$ where $\sum_{i=1}^r n_i=n$ if it consists of matrices of the following form
$$ m= \begin{pmatrix}
m_1 & 0 & 0 & \cdots & 0 \\ 0 & m_2 & 0 & \cdots & 0 \\ \cdots & \cdots & \cdots & \cdots & \cdots \\ 0 & 0 & 0 & \cdots & m_r
\end{pmatrix}
$$
with $m_i \in \GL(n_i)$. If $M_n$ is the standard Levi subgroup of type $(n_1,n_2,\cdots,n_r)$, we can write $M_n=M_w$ and $w=w_{M_n}$.

\begin{exmp}
In the special case $n=4$, all the $2^{4-1}=8$ relevant Weyl elements can be represented as follows:

$$ e= I_4, \;\;\; w_{G_4},$$
$$ w_1= \begin{pmatrix} 1 & 0 \\ 0 & w_{G_3} \end{pmatrix}, \;\;\; w_2= \begin{pmatrix} w_{G_2} & 0 \\ 0 & w_{G_2} \end{pmatrix}, \;\;\; w_3= \begin{pmatrix} w_{G_3} & 0 \\ 0 & 1 \end{pmatrix},$$
$$ w_4= \begin{pmatrix} 1&& \\ & w_{G_2} & \\ && 1 \end{pmatrix}, \;\;\; w_5= \begin{pmatrix}1&& \\ &1& \\ && w_{G_2} \end{pmatrix}, \;\;\; w_6= \begin{pmatrix} w_{G_2} && \\ & 1 & \\ &&1  \end{pmatrix}.$$
\end{exmp}

Now we continue for the general case. Let $M_n$ be a standard Levi subgroup of $G_n$. We set $W_{M_n}:= W_{G_n} \cap M_n$ and $R_{M_n}= R_{G_n} \cap M_n$. We let $A_{M_n}$ be the center of $M_n$. If $w$ is in $R_{G_n}$ then we set $A_w= A_{M_w}$. Any relevant orbit of $N \times N$ contains a unique representative of the form $wa$ with $w \in R_{M_n}$ and $a \in A_w$.

If $w_1$ and $w_2$ are in $R_{G_n}$, we write $w_1 \rightarrow w_2$ if $M_{w_1} \subseteq M_{w_2}$. This is equivalent to $A_{w_1} \supseteq A_{w_2}$ and also to $w_1 \in R_{M_{w_2}}$.

We recall that $\Delta_r(g)$ is the determinant of principal $r \times r$ minor of $g$ for $1 \leq r \leq n$. Let $\widetilde{\Delta}(G_n)$ be the set of the functions $\Delta_r$ on $G_n$. These functions are invariant under the action of $N \times N$, where the action of $N \times N$ on $G_n$ is given by $g \rightarrow n_1^t g n_2$ for an element $(n_1, n_2) \in N \times N$.

Suppose $w_1 \rightarrow w_2$, that is, $w_1 \in R_{M_{w_2}}$. This implies that if $\widetilde{\Delta} \in \widetilde{\Delta}(G_n)$ and $\widetilde{\Delta}(w_2) \neq 0$ then $\widetilde{\Delta}(w_1 a) \neq 0$ as well for any $a \in A_{w_1}$. We can therefore define the subset $A_{w_1}^{w_2}$ of $a \in A_{w_1}$ such that
$$ \widetilde{\Delta}(w_1 a)=\widetilde{\Delta}(w_2) $$
for each $\widetilde{\Delta} \in \widetilde{\Delta}(G_n)$ such that $\widetilde{\Delta}(w_2) \neq 0$.

Using above notations, we are able to state the follow lemma, which was proved in \cite{JY99}.

\begin{lem}
Suppose $w_1 \rightarrow w_2$ and $a \in A_{w_1}$. Then there are only finitely many pairs $(b,c)$, $b \in A_{w_1}^{w_2}$, $c \in A_{w_2}$ such that $a=bc$.
\end{lem}

From the above lemma, we can state the following definition. If $f_1$ and $f_2$ are any complex valued and locally constant functions on $A_{w_1}^{w_2}$ and $A_{w_2}$ respectively, we can define a new function $f_1 \star f_2$ on $A_{w_1}$ by the formula
$$ f_1 \star f_2(a):= \sum f_1(b)f_2(c), $$
where the sum is over all pairs $(b,c)$, $b \in A_{w_1}^{w_2}$, $c \in A_{w_2}$ such that $a=bc$. Moreover, $f_1 \star f_2(a)=0$ if there is no such pair. We can now state the main result in \cite{JY96} and \cite{JY99}.

\begin{thm}  \label{germ}
For each pair $(w,w') \in R_{G_n} \times R_{G_n}$ with $w \rightarrow w'$, there exist a family of locally constant functions $K_w^{w'}$ defined on $A_w^{w'}$, satisfying the following properties.
\begin{itemize}
\item If $w'=w$ then $K_w^{w'}=\delta_e$ the Dirac delta function on the finite set $A_w^w$.
\item For each function $f \in C_c^{\infty}(G_n)$, there are functions $\omega_w \in C_c^{\infty}(A_w)$ which depend on the choice of function $f$, $w \in R_{G_n}$, such that, for any $w \in R_{G_n}$, the corresponding generalized orbital integral satisfies
$$ I(w \cdot ,f) = \sum_{w' \in R_{G_n},\;w \rightarrow w'} K_w^{w'} \star \omega_{w'}.$$
\end{itemize}
In this expression, locally constant functions $K_w^{w'}$ are independent on the choice of the test function $f$.
\end{thm}
We recall that $I(w \cdot, f)$ is the generalized orbital integral defined from $A_w$ to $\BC$ as follows:
For any $a \in A_w$, we define
\[
I(wa,f):= \int f(n_1^t wa n_2) \psi^{-1}(n_1 n_2) dn_1 dn_2,
\]
where the integral is taken over the quotient of $N(\BQ_p) \times N(\BQ_p)$ by the subgroup of pairs $(n_1,n_2)$ such that $n_1^t wa n_2=wa$.

In above Theorem \ref{germ}, $K_w^{w'}$ are called the (relative) Shalika germs according to Jacquet-Ye's paper \cite{JY96} and \cite{JY99}.

For the special case $G_4=\GL_4(\BQ_p)$, let $K_e^{w_i}$ ($i=1,2,3,4,5,6$) and $K_e^{w_{G_4}}$ be the relative Shalika germs. We can state a special case of above theorem as follows:

\begin{exmp}  \label{germ2}
In the $\GL(4)$ case, the generalized orbital integral $I(w \cdot, f)$ has the following asymptotic behaviour if we fix $w=e=I_4$:
\begin{equation} \label{equ1}
I(e \cdot, f)= I_{f , \psi}= \omega_e+ \sum_{i=1}^{6} K_e^{w_i} \star \omega_{w_i} + K_e^{w_{G_4}}\star \omega_{w_{G_4}},
\end{equation}
where $\omega_e$, $\omega_{w_{G_4}}$ and $\omega_{w_i}$ ($i=1,2,3,4,5,6$) are all locally constant and compact supported functions which are dependent on the choice of $f$.
\end{exmp}

\begin{rmk}  \label{gl4}
The above theorem actually gives the asymptotic behaviour of orbital integrals at the boundary. Assume that $C_1, C_2, C_3, C_4$ are four fixed positive real number and $\epsilon$ is a small enough positive real number, we have the following seven different cases.
\begin{itemize}
\item[(a)] If $\vert \Delta_1(a) \rvert =C_1$, $\vert \Delta_2(a) \rvert < \epsilon$, $\vert \Delta_3(a) \rvert < \epsilon$ and $\vert \Delta_4(a) \rvert= C_4$, then we can write
$$ I_{f,\psi}(a)= K_e^{w_1} \star \omega_{w_1}= \sum_{a=bc} K_e^{w_1}(b)\omega_{w_1}(c).$$
\item[(b)] If $\vert \Delta_1(a) \rvert < \epsilon$, $\vert \Delta_2(a) \rvert =C_2$, $\vert \Delta_3(a) \rvert < \epsilon$ and $\vert \Delta_4(a) \rvert = C_4$, then we can write
$$ I_{f,\psi}(a)= K_e^{w_2} \star \omega_{w_2}= \sum_{a=bc} K_e^{w_2}(b)\omega_{w_2}(c).$$
\item[(c)] If $\vert \Delta_1(a) \rvert < \epsilon$, $\vert \Delta_2(a) \rvert < \epsilon$, $\vert \Delta_3(a) \rvert = C_3$ and $\vert \Delta_4(a) \rvert = C_4$, then we can write
$$ I_{f,\psi}(a)= K_e^{w_3} \star \omega_{w_3}= \sum_{a=bc} K_e^{w_3}(b)\omega_{w_3}(c).$$
\item[(d)] If $\vert \Delta_1(a) \rvert =C_1$, $\vert \Delta_2(a) \rvert < \epsilon$, $\vert \Delta_3(a) \rvert =C_3$ and $\vert \Delta_4(a) \rvert = C_4$, then we can write
$$ I_{f,\psi}(a)= K_e^{w_4} \star \omega_{w_4}= \sum_{a=bc} K_e^{w_4}(b)\omega_{w_4}(c).$$
\item[(e)] If $\vert \Delta_1(a) \rvert =C_1$, $\vert \Delta_2(a) \rvert =C_2$, $\vert \Delta_3(a) \rvert < \epsilon$ and $\vert \Delta_4(a) \rvert = C_4$, then we can write
$$ I_{f,\psi}(a)= K_e^{w_5} \star \omega_{w_5}= \sum_{a=bc} K_e^{w_5}(b)\omega_{w_5}(c).$$
\item[(f)] If $\vert \Delta_1(a) \rvert < \epsilon$, $\vert \Delta_2(a) \rvert = C_2$, $\vert \Delta_3(a) \rvert =C_3$ and $\vert \Delta_4(a) \rvert = C_4$, then we can write
$$ I_{f,\psi}(a)= K_e^{w_6} \star \omega_{w_6}= \sum_{a=bc} K_e^{w_6}(b)\omega_{w_6}(c).$$
\item[(g)] If $\vert \Delta_1(a) \rvert <\epsilon$, $\vert \Delta_2(a) \rvert < \epsilon$, $\vert \Delta_3(a) \rvert < \epsilon$ and $\vert \Delta_4(a) \rvert = C_4$, then we can write
$$ I_{f,\psi}(a)= K_e^{w_{G_4}} \star \omega_{w_{G_4}}= \sum_{a=bc} K_e^{w_{G_4}}(b)\omega_{w_{G_4}}(c).$$
\end{itemize}
\end{rmk}

\begin{rmk}
An analogy of above theorem in $\GL(3)$ case can be found in \cite[Theorem 4.2]{Ba04}.
\end{rmk}

Now we come back to the general case. In order to prove Conjecture \ref{conj1}, from above Theorem \ref{germ} of Jacquet-Ye, it is sufficient to study the (relative) Shalika germs $K_e^{w_i}$ ($i=1,2,\cdots,2^n-3,2^n-2$) and $K_e^{w_{G_n}}$.

In other words, Conjecture \ref{conj1} can be reduced to the proof of following Conjecture on the boundness of relative Shalika germs:

\begin{conj} \label{con}
The absolute value $\left \vert K_e^{w_i} \Delta^{\frac{1}{2}-\delta} \right \rvert$ ($i=1,2,\cdots,2^n-3,2^n-2$) and $\left \vert K_e^{w_{G_n}} \Delta^{\frac{1}{2}-\delta} \right \rvert$ are bounded on $A_e^{w_i}$ (respectively, $A_e^{w_{G_n}}$) for some given $\delta>0$.
\end{conj}

\begin{conj} \label{conj2}
The absolute value $\left \vert K_e^{w_{G_n}} \Delta^{\frac{1}{2}-\delta_n} \right \rvert$ is bounded on $A_{e}^{w_{G_n}}$ for some given $\delta_n>0$.
\end{conj}

Although Conjecture \ref{conj2} is a part of Conjecture \ref{con}, we will see that Conjecture \ref{conj2} is exactly equivalent to Conjecture \ref{con}.

\begin{thm}
Assume that Conjecture \ref{conj2} holds for general $n$, then Conjecture \ref{con} is true in general.
\end{thm}

\begin{proof}
It is sufficient to prove the following result:

If $\left \vert K_e^{w_{G_n}} \Delta^{\frac{1}{2}-\delta_n} \right \rvert$ is bounded on $A_{e}^{w_{G_n}}$ for some given $\delta_n>0$ and every positive integer $n \geq 2$, then $\left \vert K_e^{w_i} \Delta^{\frac{1}{2}-\delta} \right \rvert$ ($i=1,2,\cdots,2^n-3,2^n-2$) is bounded on $A_e^{w_i}$ for some given $\delta>0$.

From the discussion of Section 2 in \cite{JY99}, the system of (relative) Shalika germs is inductive. More precisely,  if $w' \neq w_{G_n}$, then for every $a$ in $A_e^{w'}$, we can write $a=\diag (a_1,a_2,\cdots,a_s)$ with $a_i \in A_e^{w_i'}$ and $s \geq 2$. By the classification of relevant Weyl elements, it is known that each $w_i'$ is the longest Weyl element of some small $\GL(n_i)$ and $\sum_{i=1}^s n_i=n$. By the definition of the subset $A_e^{w_i'}$, for each $a_i \in A_e^{w_i'}$, we have $\vert \det a_i \rvert= \vert \det w_i' \rvert=1$. Therefore, we have $A_e^{w'}= \prod A_{e_i}^{w_i'}$ and the equation of the (relative) Shalika germs is as follows: $K_e^{w'}(a)= \prod_{i=1}^s K_{e_i}^{w_i'} (a_i)$. Now we pick $\delta:= \frac{1}{2} \cdot \min \{\delta_2,\delta_3,\cdots,\delta_n \}$. Hence by the equation $K_e^{w'}(a)= \prod_{i=1}^s K_{e_i}^{w_i'} (a_i)$ and $\Delta^{\frac{1}{2}-\delta}(a)= \prod_{i=1}^s \Delta^{\frac{1}{2}-\delta}(a_i)$ (Note that $\vert \det a_i \rvert= \vert \det w_i' \rvert=1$), we prove the reduction.
\end{proof}

\begin{rmk}
We give a remark for the above theorem in $\GL(4)$ case. Since $w_i \neq w_{G_4}$ ($i=1,2,3,4,5,6$), we can reduce the calculation of the (relative) Shalika germs $K_e^{w_i}$ to the $\GL(3)$ or $\GL(2)$ case. Explicitly, we can reduce the (relative) Shalika germs $K_e^{w_1}$ and $K_e^{w_3}$ to $\GL(3)$ (relative) Shalika germs. The (relative) Shalika germs $K_e^{w_i}$ ($i=4,5,6$) can be reduced to $\GL(2)$ (relative) Shalika germs. Finally, the (relative) Shalika germ $K_e^{w_2}$ can be reduced to the product of two $\GL(2)$ (relative) Shalika germs. In conclusion, we can reduce these $\GL(4)$ (relative) Shalika germs to lower rank cases. Therefore by the results in \cite{Ba04} \cite{JY96} \cite{JY99}, we can prove the first part of above Conjecture \ref{con} by setting $\delta=\frac{1}{8}$.
\end{rmk}

The proof of above Conjecture \ref{conj2} is the key ingredient of the proof of local integrability of Bessel functions. We will prove above Conjecture \ref{conj2} in the following Section 5 and 6 by using G. Stevens' idea to estimate the corresponding Kloosterman integral (sum). In order to estimate the Kloosterman integral, we need to give a fine stratification of Kloosterman sets and use the Weil bound to prove a nontrivial bound for the Kloosterman sum attached to the longest Weyl element $w_{G_n}$.

\subsection{Reduction Step 4: An explicit formula for $K_e^{w_{G_n}}$}

By \cite[Section 2]{JY99}, we give the explicit formula for $K_e^{w_{G_n}}$. The relation between $K_e^{w_{G_n}}$ and local Kloosterman sums for $G_n$ is given in the next section.

For the $p$-adic local field $F=\BQ_p$, let $\BZ_p$ be the ring of integers and $p \BZ_p$ be the (unique) maximal ideal. Let $p$ be the cardinality of the residue field $ \BZ_p / p \BZ_p$. Let $K_{m'}$ be the principal congruence subgroup of maximal compact open subgroup of $K$, i.e. we have
$$ K_{m'}= I_n + M_{n \times n}( p^{m'} \BZ_p ),$$
where $M_{n \times n}$ is the $n \times n$ matrix.
We fix the positive integer $m'$. Here $m'$ should be chosen large enough such that $a=1$ when $a^n=1$ and $a \in 1 + p^{m'} \BZ_p $. Moreover, without loss of generality, we can further assume that $m'>>n$.

For $\alpha= \diag(a_1,a_2,\cdots,a_n)$ where $\prod_{i=1}^n a_i=(-1)^{\frac{(n+1)(n+2)}{2}+1}$ and $\vert a_1 \rvert \leq q^{-k},\vert a_1 a_2 \rvert \leq q^{-k},\cdots, \cdots,$ $ \vert a_1 a_2 \cdots a_i \rvert  \leq q^{-k},\cdots, \cdots, \vert a_1 a_2 \cdots a_{n-1} \rvert \leq q^{-k}$ for some large enough positive integer $k$ ($k>>m'$), we have
$$ K_{e}^{w_{G_n}}(\alpha)= I(\alpha, \Phi)= \int_{N \times N} \Phi(n_1^t\alpha n_2)\psi^{-1}(n_1 n_2)dn_1 dn_2.$$
We recall that the measure here is the standard normalized Haar measure such that the volume of the maximal compact subgroup $\GL_n( \BZ_p)$ and the subgroup $N \cap \GL_n(\BZ_p)= N(\BZ_p)$ both equal to one. Moreover, we have $\Phi(x)=q^{\frac{n(n-1)}{2}m'} \cdot 1_{w_{G_n} K_{m'}}$, where $1_{w_{G_n} K_{m'}}$ is the characteristic function of the compact subset $w_{G_n} K_{m'}$. It is easy to see that $\Phi \in C_c^{\infty}(G_n)$.


\section{Summary of Stevens' results}

\subsection{Definition of $\GL(n)$ Kloosterman sums}

We will mainly use the method in \cite{Ste87} to estimate the Kloosterman integrals for $\GL(n)$. In this section, we will briefly review the paper \cite{Ste87}.

The classical ($\GL(2)$) Kloosterman sum is given by
$$  S(m,n;c) = \sum_{d,\bar{d} \mod c} e\left(\frac{md+n\bar{d}}{c}\right) $$
where $d\bar{d}\equiv1\pmod{c}$ and $e(x)=e^{2\pi i x}$,
which arises when one computes the Fourier expansion of the $\GL(2)$ Poincar\'{e} series.

Now we define the $\GL(n)$ Kloosterman sums. We first recall some notations which are given in above Section 3. Let $F=\BQ_p$ be a non-archimedean local field, where $p$ is a prime number. 
Let $\BZ_p$ be the ring of integers and $p \BZ_p$ be the (unique) maximal ideal. Let $p$ be the cardinality of the residue field $\BZ_p / p \BZ_p$. Let $K_{m'} = I_n + M_{n \times n}( p^{m'} \BZ_p )$ be the principal congruence subgroup of maximal compact open subgroup of $K=\GL_n(\BZ_p)$.

We let $\psi_p$ be a nontrivial additive character from $N(\BQ_p)$ to $\BC^{\times}$ which is trivial on $N(p^m \BZ_p):=N(\BQ_p) \bigcap K_m$, where $K_m$ is the principal congruence subgroup $1+M_{n \times n}(p^m \BZ_p)$. Here $m$ is a fixed large enough positive integer with $m>2m'>>n$. The character $\psi_p$ has the following form:
$$
  \psi_p \left( \begin{pmatrix}
                        1 & x_1 & \cdots & * & * \\
                       0 & 1 & x_2 & \cdots & * \\
                       \cdots & \cdots & \cdots & \cdots & \cdots \\
                        0 & 0 & \cdots & 1 & x_{n-1} \\
                        0 & 0 & \cdots & 0 & 1
                      \end{pmatrix} \right)
  = \xi \left( \sum_{i=1}^{n-1}  x_i \right),
$$
where $\xi$ is the standard nontrivial additive character of $\BQ_p$ as in \cite[Section 1]{Ste87}. 

We set $c= \diag(a_1,a_2,\cdots,a_n) \in T$.
Now following \cite[Section 2]{Ste87}, we define
\begin{equation*}
  C(w_{G_n} c)  := N(\mathbb{Q}_p)w_{G_n} cN(\mathbb{Q}_p)\cap K_m,
\end{equation*}
and
\begin{equation*}
  X(w_{G_n} c)  := N(p^m \mathbb{Z}_p)\backslash C(w_{G_n} c)/ N(p^m \mathbb{Z}_p).
\end{equation*}
By the Bruhat decomposition we have natural maps
\begin{equation*}
  u:  X(w_{G_n} c)\rightarrow N(p^m\mathbb{Z}_p)\bs N(\mathbb{Q}_p),  \quad
  u':  X(w_{G_n} c)\rightarrow N(\mathbb{Q}_p)/ N(p^m \mathbb{Z}_p).
\end{equation*}
defined by the relation $x=u(x)w_{G_n} c u'(x)$ for $x\in X(w_{G_n} c)$.

Now the \emph{local Kloosterman sum for $\GL_n(\BQ_p)$} is defined as follows:
\[
    Kl_p(\psi_p ;c,w_{G_n}) := \sum_{x\in X(w_{G_n} c)}\psi_p(u(x))\cdot \psi_p(u'(x)).
\]

\begin{rmk}
Using the notation in Section 3. If we further assume that $\psi_p=\psi_p'$, in this special case we have the following equation:
$$ Kl_p(\psi_p^{-1};c,w_{G_n}) = p^{n(n-1)m} \cdot J_{1_{K_m},\psi_p}(w_{G_n} c) = p^{n(n-1)m} \cdot I_{1_{w_{G_n} K_m},\psi_p}(c).$$
By Section 3.4, we have the explicit formula as follows:
$$K_{e, \psi_p}^{w_{G_n}}(c)= p^{\frac{n(n-1)}{2}m} \cdot I_{1_{w_{G_n} K_m},\psi_p}(c).$$
Hence we see that
$$ K_{e, \psi_p}^{w_{G_n}}(c)= p^{- \frac{n(n-1)}{2}m} \cdot Kl_p(\psi_p^{-1};c,w_{G_n}).$$
So the local Kloosterman sum is exact a special case of the orbital (Kloosterman) integral when the test function is the multiplication of the characteristic function for the congruence subgroup $K_m$ and a positive constant $p^{n(n-1)m}$. Now in order to prove the non-trivial bound for the relative Shalika germs and orbital integrals, it suffices to give a non-trivial bound for the local Kloosterman sum.
\end{rmk}

\subsection{Decomposition of $\GL(n)$ Kloosterman sums}

We need an orbit decomposition of $X(\tau)$ where $\tau= w_{G_n} c$. Let $t' \in T(1+ p^m \BZ_p)$ and set $s:= \tau^{-1}t' \tau \in T(1+ p^m \BZ_p)$. If $\gamma \in C(\tau)$, we can write $\gamma = u \cdot \tau \cdot u'$ with $u \in N(\BQ_p)$ and $u' \in N(\BQ_p)$. We also have
\[
t' \gamma s^{-1}= (t'ut'^{-1}) \cdot \tau \cdot (su' s^{-1}) \in N(\BQ_p) \tau N(\BQ_p) \cap K_m = C(\tau).
\]
Since conjugation by $t$ and $s$ preserves $N(p^m \BZ_p)$, we know that the map
\[
T(1+ p^m \BZ_p) \times C(\tau) \rightarrow C(\tau),\;\quad (t,\gamma) \rightarrow  t \gamma s^{-1}
\]
descends to an action of $T(1+ p^m \BZ_p)$ on $X(\tau)$ as follows
$$ T(1+p^m \BZ_p) \times X(\tau) \rightarrow X(\tau),\;\;\;t,\;x \rightarrow t * x.$$
For characters $\psi_p$, the decomposition of $X(\tau)$ into $T(1+ p^m \BZ_p)$-orbits leads to the following decomposition of the Kloosterman sums:
$$ Kl_p(\psi_p;c,w_{G_n}):= \sum_{ x \in T(1+ p^m \BZ_p) \backslash X(\tau)} \sum_{y \in T(1+ p^m \BZ_p) * x} \psi_p(u(y)) \cdot \psi_p(u'(y)).$$
Here $T(1+ p^m \BZ_p) \backslash  X(\tau)$ is a set of representations $x \in X(\tau)$ for the $T(1+ p^m \BZ_p)$-orbits and $T(1+ p^m \BZ_p) * x$ is the orbit through $x$.


Now we want to write down the above formula explicitly. We need the following notations.

The roots of the standard torus $T$
in $\GL(n)$ are the characters
$\underline{\lambda}_{ij}: T\rightarrow \GL(1)$ given by
\begin{equation}\label{eqn: roots of T}
  \underline{\lambda}_{ij}\left(\begin{pmatrix}  t_1 &&&& \\  &t_2&&&  \\ && \cdots && \\ &&&t_{n-1}& \\&&&&t_n \end{pmatrix} \right)
  = t_it_j^{-1}, \quad 1\leq i,j \leq n,\ i\ne j.
\end{equation}
For $w\in W_{G_n}$, we recall that $w(j)$, $j\in\{1,2,3,\cdots,n\}$ is given by the formula
\[
  w\cdot e_j =  e_{w(j)},
\]
where $e_1,e_2,e_3,\cdots,e_n$ is the standard basis of column vectors.
Let $\Delta = \{\underline{\lambda}_{i,i+1}: 1\leq i \leq n-1 \}$ be
the root basis associated to the standard unipotent subgroup $U$.
Let
\begin{equation}\label{eqn: Delta_w}
  \Delta_w := \{\underline{\lambda}_{i,i+1}:w(i+1)<w(i)\}. 
\end{equation}
We write
\[
  u(x)= \begin{pmatrix} 1 & u_1 & * &\cdots & *\\  & 1 & u_2 & \cdots & *\\ && \cdots & \cdots & \cdots \\ &&& 1 & u_{n-1} \\  &&&&1 \end{pmatrix},
\quad
  u'(x)= \begin{pmatrix} 1 & u_1' & * & \cdots & *\\  & 1 & u_2' & \cdots & *\\ && \cdots & \cdots & \cdots \\ &&& 1 & u_{n-1}' \\  &&&&1 \end{pmatrix}.
\]
For $1\leq i \leq n-1$ we define
$\kappa_i: X(w_{G_n} c)\rightarrow \mathbb{Q}_p/ p^m \mathbb{Z}_p$, and
if also $w(i+1)<w(i)$ then we define $\kappa_i':X(w_{G_n} c)\rightarrow \mathbb{Q}_p/ p^m \mathbb{Z}_p$ by
\begin{equation}\label{eqn: kappa}
  \kappa_i(x) = u_i, \quad \kappa_i'(x) = u_i'.
\end{equation}
For $t\in T(1+ p^m \mathbb{Z}_p)$ we then have
\begin{equation}\label{eqn: kappa t*x}
  \kappa_i(t*x) = \underline{\lambda}_{i,i+1} \kappa_i(x), \quad
  \kappa_i'(t*x) = \underline{\lambda}_{w(i),w(i+1)} \kappa_i'(x).
\end{equation}
Since $$u'(t*x)=s \cdot u'(x) \cdot s^{-1}$$ for $t \in T(1+ p^m \mathbb{Z}_p)$ and $s= \tau^{-1}t \tau \in T(1+ p^m \mathbb{Z}_p)$, we see that the orbits in $X(w_{G_n}c)=X(\tau)$ correspond to $T(1+ p^m \mathbb{Z}_p)$-conjugacy classes in $N(\mathbb{Q}_p)/ N(p^m \mathbb{Z}_p)$.

The following definition is an analogy of Definition 4.9 in \cite{Ste87}.

\begin{defn}\label{defn: stevens} \cite[Definition 4.9]{Ste87}
  \begin{itemize}
    \item [(a)] For $\ell>0$ and $w\in W_{G_n}$, we let
            \begin{equation*}
            \begin{aligned}
              A_{w}(\ell) & := (\mathbb{Z}/p^{\ell+m} \mathbb{Z})^\Delta \times (\mathbb{Z}/p^{\ell+m}\mathbb{Z})^{\Delta_w} \\
             & = \prod_{i=1}^{n-1}(\mathbb{Z}/p^{\ell+m} \mathbb{Z}) \times \prod_{i=1\atop w(i+1)<w(i)}^{n-1} (\mathbb{Z}/p^{\ell+m} \mathbb{Z}).
            \end{aligned}
            \end{equation*}


            A typical element of $A_{w_{G_n}}(\ell)$ will be denoted by
            $$\underline{\lambda}\times\underline{\lambda}'
            = (\lambda_i)_{i=1,2,3,\cdots,n-1}\times (\lambda_i')_{i=1,2,3,\cdots,n-1 \atop w(i+1)<w(i)}$$.

          If $w=w_{G_n}$ which is the longest Weyl element, we will simply have $A_{w_{G_n}}(\ell)= (\mathbb{Z}/p^{\ell+m} \mathbb{Z})^{2(n-1)}$.

    \item [(b)] Let
           \begin{equation*}
           \begin{aligned}
              V_{w}(\ell) := & \Big\{ \underline{\lambda}\times\underline{\lambda}'\in A_{w}(\ell)
              : \lambda_i\in (\mathbb{Z}/p^{\ell+m} \mathbb{Z})^\times,   \\
          & \lambda_i \equiv 1 \mod p^m \; \textrm{and} \;\lambda_i'\cdot\prod_{w(i+1)\leq j<w(i)}\lambda_j=1  \Big\}.
           \end{aligned}
            \end{equation*}

         If $w=w_{G_n}$ which is the longest Weyl element, we will simply have
          \begin{equation*}
           \begin{aligned}
              V_{w_{G_n}}(\ell) := & \Big\{ \underline{\lambda}\times\underline{\lambda}'\in A_{w_{G_n}}(\ell)
              : \lambda_i\in (\mathbb{Z}/p^{\ell+m} \mathbb{Z})^\times,   \\
          & \lambda_i \equiv 1 \mod p^m \; \textrm{and} \;\lambda_i' \times \lambda_{n-i} = 1  \Big\}.
           \end{aligned}
            \end{equation*}
    \item [(c)] For a character $\theta:A_{w}(\ell)\rightarrow \mathbb{C}^\times$, we define
            \[
              S_w(\theta;\ell) := \sum_{v\in V_w(\ell)} \theta(v).
            \]
  \end{itemize}
\end{defn}

We fix
\begin{equation}\label{eqn: c}
  \widetilde{c}=\begin{pmatrix}
         p^{a_1}u_1 &  &  &  & \\
          & p^{a_2-a_1}u_2 &  &  &  \\
          &&\cdots&&&  \\
          & & & p^{a_{n-1}-a_{n-2}}u_{n-1} &  \\
          &  &  &  & p^{-a_{n-1}}u_n
       \end{pmatrix},
\end{equation}
where $u_i \in \BZ_p^{\times}$ for $i=1,2,3,\cdots,n$ and $u_1u_2u_3\cdots u_n=(-1)^{\frac{(n+1)(n+2)}{2}+1}$. Moreover, $a_i$ ($i=1,2,\cdots,n-1$) are positive integers which are larger than $m$.

The following lemma is an analogy of Theorem 4.10 in \cite{Ste87}.

\begin{lem}\label{lemma: stevens} \cite[Theorem 4.10]{Ste87}

  Let $\ell$ be large enough so that the matrix entries of
  $u(x)$, $u'(x)$ lie in $ p^{-\ell} \mathbb{Z}_p / p^m \mathbb{Z}_p$ for every $x\in X(w_{G_n} \widetilde{c})$ (For example, we can simply pick $\ell:= \max \{ a_1,a_2,\cdots,a_{n-1} \}$.).
  Let $\kappa_i(x),\kappa_i'(x)$ be as in \eqref{eqn: kappa} and
  define the character $\theta_x: A_{w_{G_n}}(\ell)\rightarrow \mathbb{C}^\times$ by
  \[
    \theta_x(\underline{\lambda}\times\underline{\lambda}')
    =\prod_{i=1}^{n-1}\xi(\lambda_i \kappa_i(x))\cdot \prod_{i=1\atop w(i+1)<w(i)}^{n-1}\xi(\lambda_i' \kappa_i'(x)).
  \]
  Let $N(x)$ denote the number of elements in the orbit through an element $x\in X(w_{G_n} \tilde{c})$.
  Then
  \[
    Kl_p(\psi_p;\widetilde{c},w_{G_n}) = \vert V_{w_{G_n}}(\ell) \rvert^{-1} \cdot \sum_{x\in T(1+p^m \mathbb{Z}_p)\bs X(w_{G_n} c)} N(x) S_{w_{G_n}}(\theta_x;\ell).
  \]
Here $\vert V_{w_{G_n}} (\ell) \rvert$ is the cardinality of the set $V_{w_{G_n}}(\ell)$.
\end{lem}

\begin{proof}
  The proof is the same as Theorem 4.10 in \cite{Ste87}. We can rewrite the Kloosterman sum as follows:
\begin{equation*}
\begin{aligned}
Kl_p(\psi_p;\widetilde{c},w) &= \sum_{ x \in T(1+ p^m \BZ_p) \backslash X(\tau)} \sum_{y \in T(1+ p^m \BZ_p) * x} \psi_p(u(y)) \cdot \psi_p(u'(y)) \\
&= \sum_{ x \in T(1+ p^m \BZ_p) \backslash X(\tau)} \sum_{y \in T(1+ p^m \BZ_p) * x} \prod_{i=1}^{n-1}\xi( \kappa_i(y))\cdot \prod_{i=1\atop w(i+1)<w(i)}^{n-1}\xi( \kappa_i'(y)).
\end{aligned}
\end{equation*}
We note that for every pair $(\underline{\lambda}, \underline{\lambda'}) \in V_w(\ell)$, we can find $t \in T(1+p^m \BZ_p)$ such that $\kappa_i(t*x)= \lambda_i \kappa_i(x)$ and $\kappa_j'(t*x)= \lambda_i' \kappa_j'(x)$ for $x \in X(\tau)$, where $1 \leq i,j \leq n$ and $w(j+1)<w(j)$. Hence we have the following equation (The summation for $y \in T(1+ p^m \BZ_p) * x$ runs over all the elements in the orbit through $x \in  X(w_{G_n} \widetilde{c})=X(\tau)$):
\begin{equation*}
\begin{aligned}
& \sum_{y \in T(1+ p^m \BZ_p) * x} \prod_{i=1}^{n-1}\xi( \kappa_i(y))\cdot \prod_{i=1\atop w(i+1)<w(i)}^{n-1}\xi( \kappa_i'(y)) \\
=& \sum_{y \in T(1+ p^m \BZ_p) * x}  \prod_{i=1}^{n-1}\xi( \kappa_i(t*y))\cdot \prod_{i=1\atop w(i+1)<w(i)}^{n-1}\xi( \kappa_i'(t*y)) \\
=& \sum_{y \in T(1+ p^m \BZ_p) * x}  \prod_{i=1}^{n-1}\xi(\lambda_i \kappa_i(y))\cdot \prod_{i=1\atop w(i+1)<w(i)}^{n-1}\xi(\lambda_i' \kappa_i'(y)).
\end{aligned}
\end{equation*}
Now we take the summation over the finite set $V_w(\ell)$. We recall that $N(x)$ is the number of elements in the orbit through an element $x \in  X(w_{G_n} \widetilde{c})$, which is given in the statement of Lemma \ref{lemma: stevens}. We have
\begin{equation*}
\begin{aligned}
& \vert V_w(\ell) \rvert \times Kl_p(\psi_p;\widetilde{c},w) \\
=& \sum_{ x \in T(1+ p^m \BZ_p) \backslash X(\tau)} \sum_{y \in T(1+ p^m \BZ_p) * x} \sum_{\lambda \times \lambda' \in V_w(\ell)} \prod_{i=1}^{n-1}\xi(\lambda_i \kappa_i(y))\cdot \prod_{i=1\atop w(i+1)<w(i)}^{n-1}\xi(\lambda_i' \kappa_i'(y)) \\
=& \sum_{ x \in T(1+ p^m \BZ_p) \backslash X(\tau)} N(x) \cdot \sum_{\lambda \times \lambda' \in V_w(\ell)} \prod_{i=1}^{n-1}\xi(\lambda_i \kappa_i(x))\cdot \prod_{i=1\atop w(i+1)<w(i)}^{n-1}\xi(\lambda_i' \kappa_i'(x)) \\
=& \sum_{x\in T(1+p^m \BZ_p)\backslash X(\tau)} N(x) \times  S_{w}(\theta_x;\ell).
\end{aligned}
\end{equation*}
The summation for $y \in T(1+ p^m \BZ_p) * x$ runs over all the elements in the orbit through $x \in  X(w_{G_n} \widetilde{c})$. We divide both sides by the cardinality $\vert V_w(\ell) \rvert$ to prove the statement.

In summary, we have the following equation:
\begin{equation*}
\begin{aligned}
& Kl_p(\psi_p;\widetilde{c},w) = \sum_{ x \in T(1+ p^m \BZ_p) \backslash X(\tau)} \sum_{y \in T(1+ p^m \BZ_p) * x} \psi_p(u(y)) \cdot \psi_p(u'(y)) \\
&= \sum_{ x \in T(1+ p^m \BZ_p) \backslash X(\tau)} \sum_{y \in T(1+ p^m \BZ_p) * x} \prod_{i=1}^{n-1}\xi( \kappa_i(y))\cdot \prod_{i=1\atop w(i+1)<w(i)}^{n-1}\xi( \kappa_i'(y)) \\
&= \vert V_{w}(\ell) \rvert^{-1} \cdot  \sum_{ x \in T(1+ p^m \BZ_p) \backslash X(\tau)} \sum_{y \in T(1+ p^m \BZ_p) * x} \sum_{\lambda \times \lambda' \in V_w(\ell)} \prod_{i=1}^{n-1}\xi(\lambda_i\kappa_i(y)) \cdot \prod_{i=1\atop w(i+1)<w(i)}^{n-1}\xi(\lambda_i' \kappa_i'(y)) \\
&= \vert V_{w}(\ell) \rvert^{-1}  \cdot \sum_{ x \in T(1+ p^m \BZ_p) \backslash X(\tau)} N(x) \cdot \sum_{\lambda \times \lambda' \in V_w(\ell)} \prod_{i=1}^{n-1}\xi(\lambda_i \kappa_i(x)) \cdot \prod_{i=1\atop w(i+1)<w(i)}^{n-1}\xi(\lambda_i' \kappa_i'(x)) \\
&= \vert V_{w}(\ell) \rvert^{-1}  \cdot \sum_{ x \in T(1+ p^m \BZ_p) \backslash X(wc)} N(x) \cdot  \sum_{\lambda \times \lambda' \in V_w(\ell)} \prod_{i=1}^{n-1}\xi(\lambda_i \kappa_i(x)) \cdot \prod_{i=1\atop w(i+1)<w(i)}^{n-1}\xi(\lambda_i' \kappa_i'(x)) \\
&= \vert V_{w}(\ell) \rvert^{-1} \times \sum_{x\in T(1+p^m \BZ_p)\backslash X(wc)} N(x) \times  S_{w}(\theta_x;\ell).
\end{aligned}
\end{equation*}
\end{proof}

\begin{rmk} \label{rmk:kloosterman}
In the above Lemma \ref{lemma: stevens}, we have
$$ S_{w_{G_n}}(\theta_x; \ell)= \prod_{i=1}^{n-1} S_2( p^{\ell}  \kappa_i(x), p^{\ell}  \kappa_{n-i}'(x) ; p^{\ell}),$$
where
$$ S_2(\nu, \nu'; p^{\ell}):= \sum_{ \lambda, \lambda' \in (\BZ/ p^{\ell+m} \BZ),\atop \lambda \equiv 1 \mod p^m,\; \lambda  \cdot \lambda' =1} \xi \left( \frac{\nu \lambda+ \nu' \lambda'}{p^{\ell}} \right)= \sum_{ \lambda, \lambda' \in (\BZ/ p^{\ell+m} \BZ),\atop \lambda \equiv 1 \mod p^m,\; \lambda  \cdot \lambda' =1} \xi \left( \frac{ p^m \nu \lambda+ p^m \nu' \lambda'}{p^{\ell+m}} \right)$$
is the (restricted) $\GL(2)$ Kloosterman sum. Here we write $\nu := p^{-m} \widetilde{\nu}$ and $\nu' := p^{-m} \widetilde{\nu}'$, where $\widetilde{\nu}, \widetilde{\nu}' \in \BZ_p-\{ 0\}$.
\end{rmk}

\begin{rmk} \label{rmk:weil}
One may find the $\GL(2)$ Kloosterman sums we get above is a little bit different from the classical $\GL(2)$ Kloosterman sums which we defined at the beginning of the section. The $\GL(2)$ Kloosterman sums here have more restrictions. Note that by the orthogonality of multiplicative characters, such kind of restricted $\GL(2)$ Kloosterman sums are closely related to so called twisted $\GL(2)$ Kloosterman sums. We have the following identity:
\begin{equation}
\begin{aligned}
S_2(\nu, \nu'; p^{\ell}) &:= \sum_{ \lambda, \lambda' \in (\BZ/ p^{\ell+m} \BZ),\atop \lambda \equiv 1 \mod p^m,\; \lambda  \cdot \lambda' =1} \xi \left( \frac{\nu \lambda+ \nu' \lambda'}{p^{\ell}} \right) \\
&=   \sum_{\lambda, \lambda' \in (\BZ/ p^{\ell+m} \BZ), \lambda  \cdot \lambda' =1} \frac{1}{\phi(p^m)} \times \sum_{\chi \mod p^m}  \xi \left( \frac{\nu \lambda+ \nu' \lambda'}{p^{\ell}} \right) \chi(\lambda) \\
&= \frac{1}{\phi(p^m)} \times \sum_{\chi \mod p^m} \sum_{\lambda, \lambda' \in (\BZ/ p^{\ell+m} \BZ), \lambda  \cdot \lambda' =1} \xi \left( \frac{\nu \lambda+ \nu' \lambda'}{p^{\ell}} \right) \chi(\lambda) \\
&= \frac{1}{p^{m-1}(p-1)} \times \sum_{\chi \mod p^m} \sum_{\lambda, \lambda' \in (\BZ/ p^{\ell+m} \BZ), \lambda  \cdot \lambda' =1} \xi \left( \frac{\nu \lambda+ \nu' \lambda'}{p^{\ell}} \right) \chi(\lambda).
\end{aligned}
\end{equation}
Therefore, by \cite[Section 9]{KL13} (Weil bounds for twisted $\GL(2)$ Kloosterman sums), we can still have an analogy of Weil bounds for this kind of restricted Kloosterman sum, which is enough for us to get the non-trivial bound of higher rank Kloosterman sums. More precisely, we have the refined Weil bound as follows:
$$  \vert S_2(\nu, \nu';p^{\ell}) \rvert \leq (\ell+m+1) \cdot C_m \cdot (\gcd(| p^m \nu|_p^{-1},| p^m \nu'|_p^{-1},p^{\ell+m}))^{1/2} p^{(\ell+m)/2},$$
where $p^m \nu, p^m \nu' \in \BZ_p-\{ 0\}$. Here $C_m$ is a positive explicit constant only depends on $p,m$ and is independent on the value of $\ell$. For example, we can pick $C_m=p^{m/2}$ in our case. In the calculation of Section 5 and Appendix, we will see that the constant $C_m$ is not sensitive.
\end{rmk}

\section{Estimating Kloosterman sums for $\GL(n)$}

In this section, we follow Stevens' approach \cite{Ste87} to bound the $\GL(n)$ Kloosterman sums introduced in the previous Section 4. Moreover, we give a non-trivial upper bound for the special $\GL(4)$ case in the Appendix (See Theorem \ref{thm: w_8}). The main ideas and ingredients of the proof on $\GL(4)$ case keep the same as the general $\GL(n)$ case in Section 5. After a more careful and delicate estimation, we can prove a slightly stronger bound than that in Theorem \ref{thm: w_n} for the special $\GL(4)$ case.

For $w\in W_{G_n}$, we define $w(j)$, $j\in\{1,2,3,\cdots, n \}$ by the formula
\[
  w\cdot e_j = \pm e_{w(j)},
\]
where $e_1,e_2,e_3,\cdots,e_n$ is the standard basis of column vectors.
We recall that the non-degenerated additive character $\psi_p$ of $N(\mathbb{Q}_p)$ which is trivial on $N(p^m \mathbb{Z}_p)$ is given by
\begin{equation}\label{eqn: psi1}
      \psi_p\left(\begin{pmatrix} 1&u_1&*&\cdots&*\\ &1&u_2&\cdots&*\\ &&\cdots&\cdots&\cdots \\ &&&1&u_{n-1}\\ &&&&1\end{pmatrix}\right) = \xi( u_1+ u_2+ u_3+\cdots+ u_{n-1}).
\end{equation}
Here $m$ is same as the $m$ that we defined in previous Section 3 and 4. The definition of additive character $\psi_p'$ is given in a similar way.

We fix
\begin{equation}\label{eqn: c}
  c=\diag( p^{a_1}v_1, p^{a_2-a_1}v_2, \cdots,  p^{a_{n-1}-a_{n-2}}v_{n-1}, p^{-a_{n-1}}v_n) \in T.
\end{equation}
In other word, we can write
\begin{equation}
  c=\begin{pmatrix}
         p^{a_1}v_1 &  &  &  & \\
          & p^{a_2-a_1}v_2 &  &  & \\
          & & \cdots & & \\
          &  &  & p^{a_{n-1}-a_{n-2}}v_{n-1} &  \\
          &  &  &  & p^{-a_{n-1}}v_n
       \end{pmatrix},
\end{equation}
where $v_i \in \BZ_p^{\times}$ for $i=1,2,3,\cdots,n$ and $v_1v_2v_3\cdots v_{n}=(-1)^{\frac{(n+1)(n+2)}{2}+1}$. Here $a_1,a_2,a_3,\cdots,a_{n-1}$ are all nonnegative integers. Moreover, we have $a_1,a_2,a_3,\cdots,a_{n-1} \geq m$, where $m$ is defined in previous Section 3 and 4 and is same as the $m$ in previous page.

We will use the same notation as in above Section 4 and \cite[Section 4]{Ste87}. Furthermore, we need Definition \ref{defn: stevens} and Lemma \ref{lemma: stevens} in above Section 4.

\begin{thm}\label{thm: w_n}
  Let $Kl_p(\psi_p ;c,w_{G_n})$ be the local Kloosterman sum attached to the longest element $w_{G_n}$.
  Let $\psi$ be as in \eqref{eqn: psi1}, $\ell=
  \max(a_1,a_2,\cdots,a_{n-1}) \geq m$,
$\varrho=\max(a_{n-1}, a_1)$, $\sigma=\min(a_{n-1}, a_1)$,
and
 \begin{equation*}
\begin{aligned}
 C_n &:=  2^{n^2-1} \cdot p^{2(n+3)(n-1)m} \cdot \left( (p^{2m},p^{\ell+m})^{1/2}\right)^{n-1}
\cdot (\ell+(n-1)m+1)^{(n^2-1)} \cdot ((n-1)\ell+n)^{\frac{n^3}{2}} \\
&= 2^{n^2-1} \cdot p^{2(n+3)(n-1)m} \cdot p^{(n-1)m}
\cdot (\ell+(n-1)m+1)^{(n^2-1)} \cdot ((n-1)\ell+n)^{\frac{n^3}{2}} \\
&= 2^{n^2-1} \cdot p^{(2n+7)(n-1)m} 
\cdot (\ell+(n-1)m+1)^{(n^2-1)} \cdot ((n-1)\ell+n)^{\frac{n^3}{2}}.
\end{aligned}
  \end{equation*}
  Then
  \begin{equation}\label{eqn:Kl-w8}
    \begin{split}
      |Kl_p(\psi_p;c,w_{G_n})|
      & \leq
        C_n \cdot \min(p^{ \sigma +a_2+\cdots+ \varrho/2+\frac{n(n-1)}{2}m},p^{ \ell/2+ 2 \sigma+(n-3)\varrho +a_2+\cdots+a_{n-2}-\ell +\frac{n(n-1)}{2}m}).
    \end{split}
  \end{equation}
  In particular, we have $|Kl_p(\psi_p;c,w_{G_n})|\leq C_n \cdot p^{(1-\frac{1}{4n^2-18n+22})\cdot(a_1+a_2+a_3+\cdots+a_{n-1})+\frac{n(n-1)}{2}m}.$
\end{thm}


We first recall Lemma 5.2 of \cite{Ste87}.

\begin{lem}\label{lemma: stevens2} \cite[Lemma 5.2]{Ste87}
Let $r \geq 1$ and for each $k=1,2,\cdots,r$ let $I_k=\{r-k+1,\cdots,r\}$ be the final $k$-element subset of $\{1,\cdots,r\}$. Let $g,g' \in \GL_r(\BQ_p)$. Then $g' \in N(\BQ_p)g$ if and only if for each $k=1,\cdots,r$ and every $k$-element subset $I \subseteq \{1,\cdots,r\}$, the bottom row of $k \times k$ subdeterminants of $g$ agrees with that for $g'$.
\end{lem}

For a fixed element $g_0:= u_1 \cdot w_{G_n} c \cdot u_2 \in X(w_{G_n}c) \subseteq K_m$, where $u_1 \in N(p^m \BZ_p) \bs N(\BQ_p)$ and $u_2 \in N(\BQ_p)/ N(p^m \BZ_p)$. From the definition of the set $X(w_{G_n}c)$, we can assume that the element
$$ u'(g_0):=u_2= \begin{pmatrix}
    1 & x_{1,2} & x_{1,3} & \cdots & x_{1,n} \\
      & 1 & x_{2,3} & \cdots & x_{2,n} \\
      && \cdots & \cdots & \cdots \\
      & & & 1 & x_{n-1,n} \\
     & & & & 1
    \end{pmatrix}
\in N(\BQ_p)/ N(p^m \BZ_p).$$

We let $x_{i,j}=p^{-b_{i,j}} \cdot c_{i,j}$, for $1 \leq i<j \leq n$,
where $b_{i,j}$ are all integers and they satisfy $b_{i,j} \geq -m$. We also have $c_{i,j} \in \BZ_p^{\times}$ for all $1 \leq i<j \leq n$. We write
$$ u_2= \begin{pmatrix}
    1 & x_{1,2} & x_{1,3} & \cdots & x_{1,n} \\
      & 1 & x_{2,3} & \cdots & x_{2,n} \\
      && \cdots & \cdots & \cdots \\
      & & & 1 & x_{n-1,n} \\
     & & & & 1
    \end{pmatrix}
=
\begin{pmatrix}
                                              1 & p^{-b_{1,2}}c_{1,2} &  p^{-b_{1,3}}c_{1,3} & \cdots &  p^{-b_{1,n}}c_{1,n} \\
                                               & 1 & p^{-b_{2,3}}c_{2,3} & \cdots & p^{-b_{2,n}}c_{2,n} \\
                                               & & \cdots & \cdots & \cdots \\
                                               &  &  & 1 & p^{-b_{n-1,n}}c_{n-1,n} \\
                                               &  &  &  & 1
                                            \end{pmatrix}
\in N(\BQ_p)/ N(p^m \BZ_p).$$

Since $X(w_{G_n}c) \subseteq K_m$, we have to put some conditions on $\{b_{i,j}\}$ and $\{c_{i,j}\}$. Thanks to the Lemma \ref{lemma: stevens2} (See also Lemma 5.2 in \cite{Ste87}), we can have the conditions of $\{b_{i,j}\}$ and $\{c_{i,j}\}$ explicitly. The conditions are restrictions of certain subdeterminants of the unipotent radical subgroups, which are some congruence relations (equations). For example, we have $-m \leq b_{i,j} \leq a_i$ for all $1 \leq i \leq n-1$. In particular, we have $b_{1,n}=a_1$.

More explicitly, we note that $g_0:= u_1 \cdot w_{G_n} c \cdot u_2 \in X(w_{G_n}c)$. Since $g_0 \in K_m$, we write $w_{G_n} c \cdot u_2= u_1^{-1} \cdot g_0$. By direct computation, we have
$$ w_{G_n} c \cdot u_2= \begin{pmatrix}
                                               & &   &  &  p^{-a_{n-1}}v_n \\
                                               &  &  & p^{a_{n-1}-a_{n-2}} v_{n-1} & p^{a_{n-1}-a_{n-2}-b_{n-1,n}} v_{n-1} c_{n-1,n} \\
                                               & & \cdots & \cdots & \cdots \\
                                               &p^{a_2-a_1} v_2  & p^{a_2-a_1-b_{2,3}} v_2 c_{2,3} & \cdots & p^{a_2-a_1-b_{2,n}} v_2 c_{2,n} \\
                                              p^{a_1} v_1 & p^{a_1-b_{1,2}} v_1 c_{1,2} & p^{a_1-b_{1,3}} v_1 c_{1,3} & \cdots & p^{a_1-b_{1,n}} v_1 c_{1,n}
                                            \end{pmatrix}.$$
This is a $n \times n$-matrix. Let $I, J \subseteq \{1,2,\cdots,n \}$ be two $k$-element subsets for $1 \leq k \leq n$. We let $g_{I,J}$ be a $k \times k$ submatrix in terms of the matrix $w_{G_n} c \cdot u_2$ by picking the $k \times k$ rows and columns with the index subset $I$ and $J$. We fix $I=\{n-k+1,n-k+2,\cdots,n\}$. By Lemma \ref{lemma: stevens2}, since $g_0 \in X(w_{G_n}c) \subseteq K_m$, we have $\det(g_{I,J}) \in p^m \BZ_p$ if $J \neq \{n-k+1,n-k+2,\cdots,n\}$ for every $1 \leq k \leq n$. If $J=I= \{n-k+1,n-k+2,\cdots,n\}$, we have $\det(g_{I,J}) \in 1+p^m \BZ_p$ for every $1 \leq k \leq n$. For example, for every $1 \leq k \leq n$, if $\{1,2,\cdots,k-1\} \subseteq J$ (If $k=1$, then $\{1,\cdots, k-1\}= \varnothing $), then we have $-m \leq b_{i,j} \leq a_i$ for all $1 \leq i \leq n-1$. If $I=J=\{n\}$, then we have $p^{a_1-b_{1,n}} v_1 c_{1,n} \in 1+p^m \BZ_p$, which gives that $b_{1,n}=a_1$ and $v_1 c_{1,n} \in 1+p^m \BZ_p$.

In summary, if $g_0:= u_1 \cdot w_{G_n} c \cdot u_2 \in X(w_{G_n}c) \subseteq K_m$, and we write 
$$ u'(g_0):=u_2= \begin{pmatrix}
                                              1 & p^{-b_{1,2}}c_{1,2} &  p^{-b_{1,3}}c_{1,3} & \cdots &  p^{-a_1}c_{1,n} \\
                                               & 1 & p^{-b_{2,3}}c_{2,3} & \cdots & p^{-b_{2,n}}c_{2,n} \\
                                               & & \cdots & \cdots & \cdots \\
                                               &  &  & 1 & p^{-b_{n-1,n}}c_{n-1,n} \\
                                               &  &  &  & 1
                                            \end{pmatrix} \pmod{N(p^m \mathbb{Z}_p)},
$$
then $\det(g_{I,J}) \in p^m \BZ_p$ if $J \neq \{n-k+1,n-k+2,\cdots,n\}$ for every $1 \leq k \leq n$. If $J=I= \{n-k+1,n-k+2,\cdots,n\}$, we have $\det(g_{I,J}) \in 1+p^m \BZ_p$ for every $1 \leq k \leq n$. Here $g$ is the $n \times n$-matrix $w_{G_n}c \cdot u_2$. These are the conditions and properties of parameters $\{b_{i,j}\}$ and $\{c_{i,j}\}$ given by Lemma \ref{lemma: stevens2}.

We consider the following $n-1$ submatrices of $u_2 \pmod{N(p^m \mathbb{Z}_p)}$, where $g_0:= u_1 \cdot w_{G_n} c \cdot u_2 \in X(w_{G_n}c) \subseteq K_m$ and $u'(g_0)=u_2$:

$$ M_1:= (p^{-b_{1,n}}c_{1,n} ),\, \text{which is the top right $1 \times 1$ block matrix in $u_2$}\,;$$
$$ M_2:= \begin{pmatrix} p^{-b_{1,n-1}}c_{1,n-1} & p^{-b_{1,n}}c_{1,n} \\ p^{-b_{2,n-1}}c_{2,n-1} & p^{-b_{2,n}} c_{2,n} \end{pmatrix},\, \text{which is the top right $2 \times 2$ block matrix in $u_2$}\,;$$
$$ M_3:= \, \text{the top right $3 \times 3$ block matrix in $u_2$}\,;$$
$$ \cdots \cdots; $$
$$ M_{n-1}:= \, \text{the top right $(n-1) \times (n-1)$ block matrix in $u_2$}.\,$$

Applying Lemma \ref{lemma: stevens2} (See also \cite[Lemma 5.2]{Ste87}) to the submatrix $M_1, M_2, \cdots, M_{n-1}$, we see that
\[
p^{a_1} v_1 \vert M_1 \rvert \in 1+p^m \BZ_p, -p^{a_2} v_1 v_2 \vert M_2 \rvert \in 1+p^m \BZ_p, \cdots, (-1)^{\frac{(n+1)(n+2)}{2}+1} \cdot p^{a_{n-1}}  \left( \prod_{i=1}^{n-1} v_i \right) \cdot \vert M_{n-1} \rvert \in 1+p^m \BZ_p.
\]
Here we let $\vert \cdot \rvert$ be the determinants of matrices.

For every $k$ ($2 \leq k \leq n-1$), we further consider the submatrix:
$$ M_k= \, \text{the top right $k \times k$ block matrix in $u_2$}.$$
For the matrix $M_k$, we consider its $k-1$ submatrices:

$$ M_{1,k}:= (p^{-b_{k,n}}c_{k,n} ),\, \text{which is the bottom right $1 \times 1$ block matrix in $M_k$}\,;$$
$$ M_{2,k}:= \begin{pmatrix} p^{-b_{k-1,n-1}}c_{k-1,n-1} & p^{-b_{k-1,n}}c_{k-1,n} \\ p^{-b_{k,n-1}}c_{k,n-1} & p^{-b_{k,n}} c_{k,n} \end{pmatrix},\, \text{which is the bottom right $2 \times 2$ block matrix in $M_k$}\,;$$
$$ M_{3,k}:= \, \text{the bottom right $3 \times 3$ block matrix in $M_k$}\,;$$
$$ \cdots \cdots; $$
$$ M_{k-1,k}:= \, \text{the bottom right $(k-1) \times (k-1)$ block matrix in $M_k$}\,;$$
$$ M_{k,k}:=M_k.$$

Applying Lemma \ref{lemma: stevens2} to the submatrices $M_{1,k}, M_{2,k}, \cdots, M_{k-1,k}, M_k$, we see that
\[
(-1)^{\frac{k(k-1)}{2}}  \cdot p^{a_{k}} \left( \prod_{i=1}^{k} v_i \right) \cdot \vert M_{k} \rvert \in 1+p^m \BZ_p,\ p^{a_{k}} \vert M_{1,k} \rvert \in p^m \BZ_p,\ p^{a_{k}} \vert M_{2,k} \rvert \in p^m \BZ_p,\ \cdots,\ p^{a_{k}} \vert M_{k-1,k} \rvert \in p^m \BZ_p.
\]
Hence we know that there exist $t_1,t_2,\cdots, t_{k-1} \in \BZ$, where $t_i \leq a_k\leq \ell$ for all $1 \leq i \leq k-1$ such that
\[
p^{t_1} \vert M_{1,k} \rvert \in \BZ_p^{\times},\ p^{t_2} \vert M_{2,k} \rvert \in \BZ_p^{\times},\ \cdots,\ p^{t_{k-1}} \vert M_{1,k-1} \rvert \in \BZ_p^{\times}.
 \]
Furthermore, without loss of generality, we can assume that $t_1= b_{k,n} \geq -m$ and $t_i \leq t_{i+1}+m$ for all $1 \leq i \leq k-1$, where $t_k:=a_k$. Therefore, we have $t_i \geq -i \times m$ for all $1 \leq i \leq k-1 \leq n-1$. Otherwise, we may substitute the element $x_{\{b_{i,j}\}}^{\{c_{i,j}\}} \in C(w_{G_n}c)$ by $x_{\{b_{i,j}\}}^{\{c_{i,j}\}} u' \in C(w_{G_n}c)$, where $u'$ is an element in $N(p^m \BZ_p) \subseteq N(\BZ_p) \subseteq N(\BQ_p)$. For example, if $t_{i+1}<t_i-m$ for certain $1 \leq i \leq k-1$, we may focus on the submatrix $M_{i+1,k}$ and substitute the element $p^{-b_{k-i,n-i}}c_{k-i,n-i}$ by $p^{-b_{k-i,n-i}}c_{k-i,n-i}+p^m$. This is given by the right multiplication of $u'=(u_{i,j})$ with $u_{t,t}=1$ ($1 \leq t \leq n$), $u_{k-i,n-i}=p^m$ and all the other entrices equal to zero. We note that $x_{\{b_{i,j}\}}^{\{c_{i,j}\}}$ and $x_{\{b_{i,j}\}}^{\{c_{i,j}\}} u'$ represent the same element in $X(w_{G_n}c)$. After the right multiplication, we will finally have $t_{i+1} \geq t_i-m$.  

Moreover, we consider extra $n$ submatrices of $M_{n-1}$ as follows:
$$ U_0:= ( 1 );$$
$$ U_1:= ( p^{-b_{n-1,n}}c_{n-1,n} ),\,\text{which is the bottom right $1 \times 1$ block matrix in $M_{n-1}$}\,;$$
$$ U_2:= \begin{pmatrix} p^{-b_{n-2,n-1}}c_{n-2,n-1} & p^{-b_{n-2,n}}c_{n-2,n} \\ 1 & p^{-b_{n-1,n}}c_{n-1,n} \end{pmatrix},\,\text{which is the bottom right $2 \times 2$ block matrix in $M_{n-1}$}\,;$$
$$ U_3:= \,\text{the bottom right $3 \times 3$ block matrix in $M_{n-1}$}\,;$$
$$ \cdots \cdots;$$
$$ U_{n-2}:= \,\text{the bottom right $(n-2) \times (n-2)$ block matrix in $M_{n-1}$}\,;$$
$$ U_{n-1}:= M_{n-1}.$$
Applying Lemma \ref{lemma: stevens2} to the submatrices $U_1, U_2, \cdots, U_{n-2}$, we see that
\[
p^{a_{n-1}} \vert U_1 \rvert \in p^m \BZ_p,\ p^{a_{n-1}} \vert U_2 \rvert \in p^m \BZ_p,\ \cdots,\ p^{a_{n-1}} \vert U_{n-2} \rvert \in  p^m \BZ_p.
\]
Without loss of generality, we can assume that the norm of the determinant of $U_j$ satisfies $\vert \det(U_j) \rvert \geq p^{-(n-1)m}$ for every $0 \leq j \leq n-2$. Otherwise, we may again substitute the element $x_{\{b_{i,j}\}}^{\{c_{i,j}\}} \in C(w_{G_n}c)$ by $x_{\{b_{i,j}\}}^{\{c_{i,j}\}} u' \in C(w_{G_n}c)$, where $u'$ is an element in $N(p^m \BZ_p) \subseteq N(\BZ_p) \subseteq N(\BQ_p)$. For example, if the norm of the determinant of $U_i$ (certain $1 \leq i \leq n-2$) satisfies $\vert \det(U_i) \rvert < p^{-(n-1)m}$, we may focus on the submatrix $U_i$ and substitute the element $p^{-b_{n-i,n-i+1}}c_{n-i,n-i+1}$ by $p^{-b_{n-i,n-i+1}}c_{n-i,n-i+1}+p^m$. This is given by the right multiplication of $u'=(u_{i,j})$ with $u_{t,t}=1$ ($1 \leq t \leq n$), $u_{n-i,n-i+1}=p^m$ and all the other entrices equal to zero. We note that $x_{\{b_{i,j}\}}^{\{c_{i,j}\}}$ and $x_{\{b_{i,j}\}}^{\{c_{i,j}\}} u'$ represent the same element in $X(w_{G_n}c)$. After the right multiplication, we will have the norm of the determinant of $U_i$ satisfies $\vert \det(U_i) \rvert \geq p^{-m} \times \vert \det(U_{i-1}) \rvert$. By induction, we have the determinant $\vert \det(U_i) \rvert \geq p^{-im} \times \vert \det(U_0) \rvert \geq p^{-im} \times 1 = p^{-im} \geq p^{-(n-1)m}$. These right multiplications of $u'$ will give at most $2^{n-2}$ ($\{0,p^m\}^{n-2}$) different cases.

Conversely, if we are given integers $b_{i,j}$ ($1 \leq i<j \leq n$) with $b_{i,j} \geq -m$ and $c_{i,j} \in \BZ_p^{\times}$ ($1 \leq i<j \leq n$) satisfying all the conditions and properties in Lemma \ref{lemma: stevens2} (See also \cite[Lemma 5.2]{Ste87}), then there exists an elememt $x_{\{b_{i,j}\}}^{\{c_{i,j}\}} \in X(w_{G_n}c)$ for which
\begin{equation}\label{eqn: u'(x)}
  u'(x_{\{b_{i,j}\}}^{\{c_{i,j}\}}) = \begin{pmatrix}
                                              1 & p^{-b_{1,2}}c_{1,2} &  p^{-b_{1,3}}c_{1,3} & \cdots &  p^{-a_1}c_{1,n} \\
                                               & 1 & p^{-b_{2,3}}c_{2,3} & \cdots & p^{-b_{2,n}}c_{2,n} \\
                                               & & \cdots & \cdots & \cdots \\
                                               &  &  & 1 & p^{-b_{n-1,n}}c_{n-1,n} \\
                                               &  &  &  & 1
                                            \end{pmatrix} \pmod{N(p^m \mathbb{Z}_p)}.
\end{equation}
We also let
\begin{equation}  \label{eqn: u(x)}
 u(x_{\{b_{i,j}\}}^{\{c_{i,j}\}}) = \begin{pmatrix} 1&u_1&*&\cdots&*\\ &1&u_2&\cdots&*\\ &&\cdots&\cdots&\cdots \\ &&&1&u_{n-1}\\ &&&&1\end{pmatrix}  \in  N(p^m \BZ_p) \bs N(\BQ_p).
\end{equation}

Now we let $\psi_p$ be the nontrivial additive character of $N(\mathbb{Q}_p)$ which is trivial on $N(p^m \mathbb{Z}_p)$, i.e. the nontrivial additive character of $N(\BQ_p)/ N(p^m \BZ_p)$ which is defined in Section 4 and Section 5.
For certain $\{b_{i,j}\}$, and $\{c_{i,j}\}$ ($1 \leq i<j \leq n$) satisfying the conditions and properties given in Lemma \ref{lemma: stevens2}, we define
$$
  X_{\{b_{i,j}\}}^{\{c_{i,j}\}}(w_{G_n}c)
              := T(1+p^m \mathbb{Z}_p)*x_{\{b_{i,j}\}}^{\{c_{i,j}\}}
$$
be the orbit through $x_{\{b_{i,j}\}}^{\{c_{i,j}\}}$, and we also define
$$
  S_{\{b_{i,j}\}}^{\{c_{i,j}\}}(\psi_p;c,w_{G_n})
  := \sum_{x\in X_{\{b_{i,j}\}}^{\{c_{i,j}\}}(w_{G_n}c)}\psi_p(u(x))\psi_p(u'(x))
$$
be the Kloosterman sum restricted to the given orbit.
For all $\{b_{ij}\}, \{c_{ij}\}$ ($1 \leq i<j \leq n$) which satisfy the previous relations and properties given in Lemma \ref{lemma: stevens2}, we fix $\{b_{i,i+1}\}$, $\{c_{i,i+1}\}$ and let
$$
  X_{\{b_{i,i+1}\}, \{c_{i,i+1}\}}(w_{G_n}c) := \bigcup\limits_{\{b_{i,j}\},\{c_{i,j}\},j-i \geq 2} X_{\{b_{i,j}\}}^{\{c_{i,j}\}}(w_{G_n}c),
$$
where $\{b_{ij}\}$ run over all integers bigger than $-m$, $\{c_{ij}\}$ run over all the elements of $\mathbb{Z}_p^\times$, and $\{b_{ij}\}$, $\{c_{ij}\}$
satisfy previous conditions and properties (See Lemma \ref{lemma: stevens2}) for all $j-i \geq 2$. Let
$$
S_{\{b_{i,i+1}\},\{c_{i,i+1}\}}(\psi_p;c,w_{G_n}) := \sum_{x\in X_{\{b_{i,i+1}\},\{c_{i,i+1}\}}(w_{G_n}c)}\psi_p(u(x))\psi_p(u'(x)).
$$


\begin{lem}\label{lemma: X(cw_G)}
  We have $X(w_{G_n}c) = \coprod_{\{b_{i,i+1}\}, \{c_{i,i+1}\}} X_{\{b_{i,i+1}\},\{c_{i,i+1}\}}(w_{G_n}c)$, where $b_{i,i+1}$ run over all integers larger than $-m$ and smaller than $\ell$, and $c_{1,2} \in \mathbb{Z}_p^{\times}/ (1+p^{m_{b_{1,2}}} \BZ_p)$, $c_{2,3} \in \mathbb{Z}_p^{\times}/ (1+p^{m_{b_{2,3}}} \BZ_p)$, $\cdots$, $c_{i,i+1} \in \mathbb{Z}_p^{\times}/ (1+p^{m_{b_{i,i+1}}} \BZ_p)$, $\cdots$, $c_{n-1,n} \in \mathbb{Z}_p^{\times}/ (1+p^{m_{b_{n-1,n}}} \BZ_p)$
satisfying properties in Lemma \ref{lemma: stevens2} (See also \cite[Lemma 5.2]{Ste87}). Here $m_{b_{1,2}}=\min(m, m+b_{1,2})$, $m_{b_{2,3}}=\min(m, m+b_{2,3})$, $\cdots$, $m_{b_{i,i+1}}=\min(m, m+b_{i,i+1})$, $\cdots$, and $m_{b_{n-1,n}}=\min(m, m+b_{n-1,n})$.
\end{lem}

\begin{proof}
The proof is the same as Lemma 5.2 and 5.7 in \cite{Ste87}. The union is clearly disjoint and is contained in $X(w_{G_4} \widetilde{c})$ by definition. 
Actually, from the uniqueness of the Bruhat decomposition, the map $u': X(\tau) \rightarrow N(\BQ_p)/ N(p^m \BZ_p)$ is injective. Hence the matrix $u(x)$ is uniquely determined by the matrix $u'(x)$. For every $g_0 \in X(w_{G_n} c)$, we write $g_0= u_1 \cdot w_{G_n} c \cdot u_2$ for some $u_1 \in N(p^m \BZ_p) \bs N(\BQ_p)$ and $u_2 \in N(\BQ_p)/ N(p^m \BZ_p)$. Hnece, we can write $u'(g_0)=u_2 \pmod{N(p^m \mathbb{Z}_p)}$. We note that $u_1^{-1} \cdot g_0= w_{G_n} c \cdot u_2= \tau \cdot u_2$. By Lemma \ref{lemma: stevens2} (See also \cite[Lemma 5.2]{Ste87}), it is known that parameters $\{b_{i,j}\}$ and $\{c_{i,j}\}$ satisfy certain congruence properties and relations if we write $u_2= \begin{pmatrix}
                                              1 & p^{-b_{1,2}}c_{1,2} &  p^{-b_{1,3}}c_{1,3} & \cdots &  p^{-a_1}c_{1,n} \\
                                               & 1 & p^{-b_{2,3}}c_{2,3} & \cdots & p^{-b_{2,n}}c_{2,n} \\
                                               & & \cdots & \cdots & \cdots \\
                                               &  &  & 1 & p^{-b_{n-1,n}}c_{n-1,n} \\
                                               &  &  &  & 1
                                            \end{pmatrix} \pmod{N(p^m \mathbb{Z}_p)}.$
Since
\begin{equation} 
  u'(x_{\{b_{i,j}\}}^{\{c_{i,j}\}}) = \begin{pmatrix}
                                              1 & p^{-b_{1,2}}c_{1,2} &  p^{-b_{1,3}}c_{1,3} & \cdots &  p^{-a_1}c_{1,n} \\
                                               & 1 & p^{-b_{2,3}}c_{2,3} & \cdots & p^{-b_{2,n}}c_{2,n} \\
                                               & & \cdots & \cdots & \cdots \\
                                               &  &  & 1 & p^{-b_{n-1,n}}c_{n-1,n} \\
                                               &  &  &  & 1
                                            \end{pmatrix} \pmod{N(p^m \mathbb{Z}_p)},
\end{equation}
we know that $g_0= x_{\{b_{i,j}\}}^{\{c_{i,j}\}}$ by the uniqueness of the Bruhat decomposition and $g_0$ $ \in X_{\{b_{i,j}\}}^{\{c_{i,j}\}}(w_{G_n}c) \subseteq X_{\{b_{i,i+1}\},\{c_{i,i+1}\}}(w_{G_n}c)$.
\end{proof}

\begin{rmk}
In Section 4, we know that $u'(t*x)=s \cdot u'(x) \cdot s^{-1}$ for $t \in T(1+p^m \BZ_p)$ and $s:=\tau^{-1} t \tau \in T(1+p^m \BZ_p)$. Since the map $u': X(\tau) \rightarrow N(\BQ_p)/ N(p^m \BZ_p)$ is injective, we see that the orbits in $X(\tau)$ correspond to $T(1+p^m \BZ_p)$-conjugacy classes in the coset $N(\BQ_p)/ N(p^m \BZ_p)$. Moreover, from the above injective map $u'$, the counting of the size of the Kloosterman set $X(\tau)=X(w_{G_n}c)$ transfers to the counting of corresponding elements in the coset $N(\BQ_p)/ N(p^m \BZ_p)$.
\end{rmk}

\begin{lem}\label{lemma:S<<wn}
  Let $\ell = \max(a_1,a_2,\cdots,a_{n-1}) \geq m $, 
  and $b_{i,i+1}+m \leq a_i$ (for all $1 \leq i \leq n-1$) be integers which are bigger than $-m$.
  Then
  \[
    \begin{split}
      & |S_{\{b_{i,i+1}\},\{c_{i,i+1}\}}(\psi_p ;c,w_{G_n})|  \\
     \leq & \hskip 10pt
         2^{n-1} \cdot p^{2(n-1)m} \cdot (p^{2m}, p^{\ell+m})^{1/2} (p^{2m} ,p^{\ell+m})^{1/2} \cdots ( p^{2m} ,p^{\ell+m})^{1/2} \\
      & \hskip 10pt \cdot (\ell+m+1)^{n-1} \cdot
          p^{-\frac{b_{1,2}+b_{2,3}+\cdots+b_{n-1,n}}{2}}  \cdot  \#(X_{\{b_{i,i+1}\},\{c_{i,i+1}\}}(w_{G_n}c)) \\
     = & \hskip 10pt
         2^{n-1} \cdot p^{2(n-1)m} \cdot p^{(n-1)m} \\
      & \hskip 10pt \cdot (\ell+m+1)^{n-1} \cdot
          p^{-\frac{b_{1,2}+b_{2,3}+\cdots+b_{n-1,n}}{2}}  \cdot  \#(X_{\{b_{i,i+1}\},\{c_{i,i+1}\}}(w_{G_n}c)) \\
     = & \hskip 10pt 2^{n-1} \cdot p^{3(n-1)m}  \cdot (\ell+m+1)^{n-1} \cdot
          p^{-\frac{b_{1,2}+b_{2,3}+\cdots+b_{n-1,n}}{2}}  \cdot  \#(X_{\{b_{i,i+1}\},\{c_{i,i+1}\}}(w_{G_n}c)). 
       \end{split}
  \]
\end{lem}

\begin{proof}
  The involution map $\iota: g \rightarrow g^{\iota}:= w_{G_n} \cdot (g^t)^{-1} \cdot w_{G_n}$ sends $X_{\{b_{i,j}\}}^{\{c_{i,j}\}}(w_{G_n}c)$ to $X_{\{b_{i,j}\}}^{\{c_{i,j}\}}((w_{G_n}c)^\iota)$.
  Composing $\psi_p$, $\psi_p'$ with $\iota$ has the effect of replacing $\psi_p$ by $\overline{\psi_p}$ and $\psi_p'$ by $\overline{\psi_p'}$.
  For $g_0 =u_1 w_{G_n}c u_2 \in X(w_{G_n}c)$, we have $g_0^{\iota} \in X((w_{G_n}c)^{\iota}) = X( w_{G_n} (w_{G_n} c^{-1} w_{G_n}))  \subseteq K_m$ by definition.
  Hence, applying $\iota$ to the element $w_{G_n}c$ reverses the roles of $a_{n-1}$ and $a_1$.
  Thus we may assume that $a_{n-1} \geq a_1$ without loss of generality.

  We recall that $\ell=\max(a_1,a_2,\cdots,a_{n-1})=\max(a_2,\cdots,a_{n-1})$. Then Properties in Lemma 5.2 of \cite{Ste87} imply that the matrix entries of $u(x)$ and $u'(x)$
  lie in $p^{-\ell}\mathbb{Z}_p/ p^m \mathbb{Z}_p$ for every $x\in X(w_{G_n}c)$.
  Indeed, by Lemma \ref{lemma: X(cw_G)}, it is enough to verify this for $x=x_{\{b_{i,j}\}}^{\{c_{i,j}\}}(w_{G_n}c)$. By definition, it is easy to verify that the matrix entries of $u'(x)$ lies in $p^{-\ell}\mathbb{Z}_p/ p^m \mathbb{Z}_p$. 
By the uniqueness of the Bruhat decomposition, Proposition 3.1 and Lemma 3.4 in \cite{KN22}, we see that the matrix entrices of $u(x)$ also lies in $p^{-\ell}\mathbb{Z}_p/ p^m \mathbb{Z}_p$.

  Now let $\mathcal{S}$ be a finite subset of $\mathbb{Z}_{\geq -m}^{\frac{(n-1)(n-2)}{2}-1}\times (\mathbb{Z}_p^\times)^{\frac{(n-1)(n-2)}{2}}$
  such that $X_{\{b_{i,i+1}\}, \{c_{i,i+1}\}}(w_{G_n}c)$ is the disjoint union of the
  $X_{\{b_{i,j}\}}^{\{c_{i,j}\}}(w_{G_n}c)$ with $(\{b_{i,j}\}, \{c_{i,j}\},j-i \geq 2)\in\mathcal{S}$.
  Then as in Lemma \ref{lemma: stevens}
  we have
  \begin{equation}\label{eqn: S decomp1}
\begin{aligned}
    S_{\{b_{i,i+1}\}, \{c_{i,i+1}\}}(\psi_p ;c,w_{G_n})
    & < p^{-(n-1)\ell}(1-p^{-1})^{-(n-1)} \cdot \\
& \cdot \sum_{(\{b_{i,j}\}, \{c_{i,j}\},j-i \geq 2)\in\mathcal{S}}
    \#(X_{\{b_{i,j}\}}^{\{c_{i,j}\}}( w_{G_n}c))
    S_{w_{G_n}}(\theta_{\{b_{i,j}\}}^{\{c_{i,j}\}};\ell),
\end{aligned}
  \end{equation}
  where $S_{w_{G_n}}$ is defined in Definition \ref{defn: stevens},
  and $\theta_{\{b_{i,j}\}}^{\{c_{i,j}\}}: A_{w_{G_n}}(\ell)\rightarrow\mathbb{C}^\times$
  is the character defined in Definition \ref{defn: stevens} by
  \[
    \begin{split}
      \theta_{\{b_{i,j}\}}^{\{c_{i,j}\}}(\underline{\lambda}\times\underline{\lambda}')
      &  = e( u_1\lambda_1+ u_2\lambda_2+\cdots+ u_{n-1}\lambda_{n-1}\\
      & + p^{-b_{12}}c_{12} \lambda_1'+ p^{-b_{23}}c_{23} \lambda_2'+\cdots+ p^{-b_{n-1,n}}c_{n-1,n}\lambda_{n-1}').
    \end{split}
  \]

  By Remark \ref{rmk:kloosterman} (See also Example 4.12 in Stevens \cite{Ste87}), we have
  \begin{equation}\label{eqn: S decomp to S_2}
    \begin{split}
       S_{w_{G_n}}(\theta_{\{b_{i,j}\}}^{\{c_{i,j}\}};\ell) & = S_2( u_1p^{\ell}, c_{n-1,n}p^{\ell-b_{n-1,n}};p^{\ell}) \\
      & \hskip 10pt  \cdot S_2( u_2p^{\ell}, c_{n-2,n-1}p^{\ell-b_{n-2,n-1}};p^{\ell}) \cdots \\
      & \hskip 10pt \cdot S_2( u_{n-1}p^{\ell}, c_{12}p^{\ell-b_{12}};p^{\ell}),
    \end{split}
  \end{equation}
  where $S_2$ is the restricted $\GL(2)$-Kloosterman sum defined in Remark \ref{rmk:kloosterman}.
  By refined Weil's bound established in Remark \ref{rmk:weil}, we have the inequality
  \begin{equation}\label{eqn: S_2}
    |S_2(\nu,\nu';p^{\ell})| \leq (\ell+m+1) p^{m/2} (\gcd(|p^m \nu|_p^{-1},|p^m \nu'|_p^{-1},p^{\ell+m}))^{1/2} p^{(\ell+m)/2},
  \end{equation}
  for $\nu, \nu' \in p^{-m} \mathbb{Z}_p- \{0 \} $ (See also Section 9 in \cite{KL13}).

 In order to apply the refined Weil's bound, we note that
  \[
    \begin{split}
       \gcd(| u_{n-1}p^{\ell+m}|_p^{-1},| p^{\ell+m-b_{12}}|_p^{-1},p^{\ell+m}) & \leq \gcd( p^{2m},p^{\ell+m}) \cdot p^{\ell-b_{12}}, \\
       \gcd(| u_{n-2}p^{\ell+m}|_p^{-1},| p^{\ell+m-b_{23}}|_p^{-1},p^{\ell+m}) & \leq \gcd( p^{2m} ,p^{\ell+m})p^{\ell-b_{23}}, \\
       \cdots \cdots \\
       \gcd(| u_1 p^{\ell+m}|_p^{-1},| p^{\ell+m-b_{n-1,n}}|_p^{-1},p^{\ell+m}) & \leq \gcd( p^{2m} ,p^{\ell+m})p^{\ell-b_{n-1,n}}.
    \end{split}
  \]
since $\gcd(a,b) \leq \min(a,b)$.   Hence we have
  \begin{equation*}\label{eqn: S_w_8 bound}
  \begin{aligned}
   |S_{w_{G_n}}(\theta_{\{b_{i,j}\}}^{\{c_{i,j}\}};\ell)|
    \leq (\ell+m+1)^{n-1} p^{2(n-1)m} \cdot\left((p^{2m} ,p^{\ell+m})^{1/2}\right)^{n-1}
\cdot  p^{(n-1) \ell-\frac{b_{12}+b_{23}+\cdots+b_{n-1,n}}{2}}.
  \end{aligned}
  \end{equation*}
  This inequality, together with \eqref{eqn: S decomp1}, gives
  \begin{equation}\label{eqn: Sa,b,c}
    \begin{split}
      |S_{\{b_{i,i+1}\},\{c_{i,i+1}\}}(\psi_p ; c,w_{G_n})|& \leq (\ell+m+1)^{n-1}
      \cdot\left(( p^{2m} ,p^{\ell+m})^{1/2}\right)^{n-1}
      \cdot p^{2(n-1)m-\frac{1}{2}\sum_{j=1}^{n-1}b_{j,j+1}} \\
&\quad \cdot (1-p^{-1})^{-(n-1)} \cdot \sum_{(\{b_{i,j}\},\{c_{i,j}\},j-i \geq 2)\in\mathcal{S}}  \#(X_{\{b_{i,j}\}}^{\{c_{i,j}\}}(w_{G_n}c)).
    \end{split}
  \end{equation}
  The sum appearing on the right hand side is equal to $\#(X_{\{b_{i,i+1}\},\{c_{i,i+1}\}}(w_{G_n}c))$.
  Since $p\geq 2$ we have $(1-p^{-1})^{-(n-1)}\leq 2^{n-1}$, by \eqref{eqn: Sa,b,c}. This completes the proof of the lemma.
\end{proof}

\begin{proof}[Proof of Theorem \ref{thm: w_n}]

  By the involution map $\iota$, we can assume that $a_{n-1}\geq a_1$ without loss of generality.
  Let
\begin{equation*}
\begin{aligned}
     C :=2^{n^2-1} \cdot p^{(2n+5)(n-1)m} \cdot  \left((p^{2m} ,p^{\ell+m})^{1/2}\right)^{n-1}
\cdot (\ell+(n-1)m+1)^{(n^2-n)} \cdot ((n-1)\ell+n)^{\frac{n^3}{2}}.
\end{aligned}
\end{equation*}

  At first, we deal with the case $a_{n-1}=\max(a_1,a_2,\cdots,a_{n-1})$.

If $b_{1,2}+b_{2,3}+\cdots+ b_{n-1,n}> a_{n-1}$, we have $p^{-\frac{b_{1,2}+b_{2,3}+\cdots+b_{n-1,n}}{2}}< p^{-\frac{a_{n-1}}{2}}$. From Theorem \ref{DR1998} and the proof of Theorem \ref{delta}, we know that
\begin{equation*}
\begin{aligned}
 X_{\{b_{i,i+1}\},\{c_{i,i+1}\}}(w_{G_n}c) \leq & X(w_{G_n} c) \\
=& \# N(p^m \BZ_p) \bs C(w_{G_n} c) / N(p^m \BZ_p) \\
= & \# N(p^m \BZ_p) \bs \left(N(\BQ_p) w_{G_n} c N(\BQ_p) \cap K_m \right)/ N(p^m \BZ_p) \\
 \leq & \# N(p^m \BZ_p) \bs \left(N(\BQ_p) w_{G_n} c N(\BQ_p \right) \cap K)/ N(p^m \BZ_p) \\
 = & p^{n(n-1)m} \times \# N( \BZ_p) \bs \left(N(\BQ_p) w_{G_n} c N(\BQ_p \right) \cap K)/ N(\BZ_p).\\
 = & p^{n(n-1)m} \times O_{f_0}(c),
\end{aligned}
\end{equation*}
where $O_{f_0}(c)$ is the local orbital integral defined in Theorem \ref{DR1998} and the proof of Theorem \ref{delta}. Hence, we have
\begin{equation*}
\begin{aligned}
X_{\{b_{i,i+1}\},\{c_{i,i+1}\}}(w_{G_n}c) & \leq X(w_{G_n} c) \leq p^{n(n-1)m} \times O_{f_0}(c) \\
& \leq p^{n(n-1)m} \times p^{a_1+a_2+\cdots+ a_{n-1}} \times R(c)  \\
& \leq p^{n(n-1)m} \times p^{a_1+a_2+\cdots+ a_{n-1}} \times \left( a_1+a_2+\cdots+a_{n-1}+n \right)^{\frac{n^3}{2}} \\
& \leq p^{n(n-1)m} \times p^{a_1+a_2+\cdots+ a_{n-1}} \times \left( (n-1) \ell +n \right)^{\frac{n^3}{2}}.
\end{aligned}
\end{equation*}
Hence by Lemma \ref{lemma:S<<wn}, we have
$$
  \vert S_{\{b_{i,i+1}\},\{c_{i,i+1}\}}(\psi_p ; c,w_{G_n}) \rvert \leq C \times p^{a_1+a_2+\cdots+a_{n-2}+\frac{a_{n-1}}{2}+\frac{n(n-1)}{2}m}.
$$
This proves the case for $b_{1,2}+b_{2,3}+\cdots+ b_{n-1,n}> a_{n-1}$.

Now we assume that $b_{1,2}+b_{2,3}+\cdots+ b_{n-1,n} \leq  a_{n-1}$.

If
\begin{equation}
\begin{aligned}
& b_{1,3}+b_{2,4}+\cdots+b_{n-2,n} \leq a_{n-2};\\
& b_{1,4}+b_{2,5}+\cdots+b_{n-3,n} \leq a_{n-3};\\
& \cdots \cdots \cdots \\
& b_{1,k}+b_{2,k+1}+\cdots+b_{n-k+1,n} \leq a_{n-k+1};\\
& \cdots \cdots \cdots \\
& b_{1,n-1}+b_{2,n} \leq a_2;
\end{aligned}
\end{equation}

then we have $\#(\{b_{i,j}\}) \leq (\ell+m+1)^{\frac{(n-2)(n-1)}{2}}$ and $\#(\{c_{i,j}\}) \leq p^{\sum_{1 \leq i<j \leq n} b_{i,j}+ \frac{(n-2)(n-1)}{2}m}$ for $j-i \geq 2$. Therefore, we obtain that
\begin{equation*}
\begin{aligned}
\#( X_{\{b_{i,i+1}\},\{c_{i,i+1}\}}(w_{G_n}c)) & \leq (\ell+m+1)^{\frac{(n-2)(n-1)}{2}} \cdot p^{\sum_{1 \leq i<j \leq n} b_{i,j}+ \frac{n(n-1)}{2}m} \\
& \leq (\ell+m+1)^{\frac{(n-2)(n-1)}{2}} \cdot p^{a_1+a_2+\cdots+a_{n-2}+(b_{1,2}+b_{2,3}+\cdots+b_{n-1,n})+ \frac{n(n-1)}{2}m}.
\end{aligned}
\end{equation*}
Hence by Lemma \ref{lemma:S<<wn}, we have
$$
  \vert S_{\{b_{i,i+1}\},\{c_{i,i+1}\}}(\psi_p ; c,w_{G_n}) \rvert \leq C \cdot p^{a_1+a_2+\cdots+a_{n-2}+\frac{a_{n-1}}{2}+\frac{n(n-1)}{2}m}.
$$
Applying the above Lemma \ref{lemma: X(cw_G)}, we have
$$ \vert Kl_p(\psi_p ;c,w_{G_n}) \rvert \leq C_n \cdot p^{a_1+a_2+\cdots+a_{n-2}+\frac{a_{n-1}}{2}+\frac{n(n-1)}{2}m}.$$

If there exists $k$ ($3 \leq k \leq n-1$) such that $b_{1,k}+b_{2,k+1}+\cdots+b_{n-k+1,n} > a_{n-k+1}$, we will show that 
$$ \#(c_{1,k}, c_{2,k+1},\cdots, c_{n-k,n-1}, c_{n-k+1,n}) \leq 2^{n-k} \cdot (\ell+(n-1)m+1)^{n-k+1} \cdot p^{a_{n-k+1}+(n-k+1)m}. $$

We will consider the following $n-1$ submatrices of the uniponent subgroup $$u_2= \begin{pmatrix}
                                              1 & p^{-b_{1,2}}c_{1,2} &  p^{-b_{1,3}}c_{1,3} & \cdots &  p^{-b_{1,n}}c_{1,n} \\
                                               & 1 & p^{-b_{2,3}}c_{2,3} & \cdots & p^{-b_{2,n}}c_{2,n} \\
                                               & & \cdots & \cdots & \cdots \\
                                               &  &  & 1 & p^{-b_{n-1,n}}c_{n-1,n} \\
                                               &  &  &  & 1
                                            \end{pmatrix}:
$$

$$ M_1:= (p^{-b_{1,n}}c_{1,n} ),\,\text{which is the top right $1 \times 1$ block matrix in $u_2$}\,;$$
$$ M_2:= \begin{pmatrix} p^{-b_{1,n-1}}c_{1,n-1} & p^{-b_{1,n}}c_{1,n} \\ p^{-b_{2,n-1}}c_{2,n-1} & p^{-b_{2,n}} c_{2,n} \end{pmatrix},\,\text{which is the top right $2 \times 2$ block matrix in $u_2$}\,;$$
$$ M_3:= \text{the top right $3 \times 3$ block matrix in $u_2$}\,;$$
$$ \cdots \cdots; $$
$$ M_{n-1}:= \text{the top right $(n-1) \times (n-1)$ block matrix in $u_2$}.\,$$

Applying Lemma \ref{lemma: stevens2} (See also \cite[Lemma 5.2]{Ste87}) to the submatrix $M_1, M_2, \cdots, M_{n-1}$, we see that
\[
p^{a_1} v_1 \vert M_1 \rvert \in 1+p^m \BZ_p, -p^{a_2} v_1 v_2 \vert M_2 \rvert \in 1+p^m \BZ_p, \cdots, (-1)^{\frac{(n+1)(n+2)}{2}+1} \cdot p^{a_{n-1}}  \left( \prod_{i=1}^{n-1} v_i \right) \cdot \vert M_{n-1} \rvert \in 1+p^m \BZ_p;
\]
where $\vert \cdot \rvert$ is the determinants of matrices.

For every $k$ ($3 \leq k \leq n-1$), we further consider the submatrix:
$$ M_{k}=\,\text{the top right $k \times k$ block matrix in $u_2$}.$$

For the matrix $M_k$, we consider its $k-1$ submatrices:

$$ M_{1,k}:= (p^{-b_{k,n}}c_{k,n} ),\,\text{which is the bottom right $1 \times 1$ block matrix in $M_k$}\, ;$$
$$ M_{2,k}:= \begin{pmatrix} p^{-b_{k-1,n-1}}c_{k-1,n-1} & p^{-b_{k-1,n}}c_{k-1,n} \\ p^{-b_{k,n-1}}c_{k,n-1} & p^{-b_{k,n}} c_{k,n} \end{pmatrix},\,\text{which is the bottom right $2 \times 2$ block matrix in $M_k$}\, ;$$
$$ M_{3,k}:= \text{the bottom right $3 \times 3$ block matrix in $M_k$}\, ;$$
$$ \cdots \cdots; $$
$$ M_{k-1,k}:= \text{the bottom right $(k-1) \times (k-1)$ block matrix in $M_k$}\, ;$$
$$ M_{k,k}:=M_k.$$

Applying Lemma \ref{lemma: stevens2} to the submatrices $M_{1,k}, M_{2,k}, \cdots, M_{k-1,k}, M_k$, we see that
\[
(-1)^{\frac{k(k-1)}{2}}  \cdot p^{a_{k}} \left( \prod_{i=1}^{k} v_i \right) \cdot \vert M_{k} \rvert \in 1+p^m \BZ_p,\ p^{a_{k}} \vert M_{1,k} \rvert \in p^m \BZ_p,\ p^{a_{k}} \vert M_{2,k} \rvert \in p^m \BZ_p,\ \cdots,\ p^{a_{k}} \vert M_{k-1,k} \rvert \in p^m \BZ_p.
\]
Hence we know that there exist $t_1,t_2,\cdots, t_{k-1} \in \BZ$, where $t_i \leq a_k\leq \ell$ for all $1 \leq i \leq k-1$ such that
\[
p^{t_1} \vert M_{1,k} \rvert \in \BZ_p^{\times},\ p^{t_2} \vert M_{2,k} \rvert \in \BZ_p^{\times},\ \cdots,\ p^{t_{k-1}} \vert M_{1,k-1} \rvert \in \BZ_p^{\times}.
 \]
Furthermore, without loss of generality, we can assume that $t_1= b_{k,n} \geq -m$ and $t_i \leq t_{i+1}+m$ for all $1 \leq i \leq k-1$, where $t_k:=a_k$. Therefore, we have $t_i \geq -i \times m$ for all $1 \leq i \leq k-1 \leq n-1$. Otherwise, we may substitute the element $x_{\{b_{i,j}\}}^{\{c_{i,j}\}} \in C(w_{G_n}c)$ by $x_{\{b_{i,j}\}}^{\{c_{i,j}\}} u' \in C(w_{G_n}c)$, where $u'$ is an element in $N(p^m \BZ_p) \subseteq N(\BZ_p) \subseteq N(\BQ_p)$. For example, if $t_{i+1}<t_i-m$ for certain $1 \leq i \leq k-1$, we may focus on the submatrix $M_{i+1,k}$ and substitute the element $p^{-b_{k-i,n-i}}c_{k-i,n-i}$ by $p^{-b_{k-i,n-i}}c_{k-i,n-i}+p^m$. This is given by the right multiplication of $u'=(u_{i,j})$ with $u_{t,t}=1$ ($1 \leq t \leq n$), $u_{k-i,n-i}=p^m$ and all the other entrices equal to zero. We note that $x_{\{b_{i,j}\}}^{\{c_{i,j}\}}$ and $x_{\{b_{i,j}\}}^{\{c_{i,j}\}} u'$ represent the same element in $X(w_{G_n}c)$. After the right multiplication, we will finally have $t_{i+1} \geq t_i-m$. These right multiplications of $u'$ will give at most $2^{k-1}$ ($\{0, p^m \}^{k-1}$) different cases. Therefore by induction, we see that
\begin{align*}
\#(c_{k,n}) &\leq  p^{t_1+m},\\
\#(c_{k-1,n-1}, c_{k,n}) &\leq  2 \cdot p^{t_1+m+(t_2-t_1)+m} = 2 \cdot p^{t_2+2m},\\
&\cdots,  \\
\#(c_{i,n-k+i}, c_{i+1,n-k+i+1}, \cdots, c_{k-1,n-1}, c_{k,n}) &\leq 2 \cdot 2^{k-i-1} \cdot p^{t_{k-i}+(k-i)m+ (t_{k-i+1}-t_{k-i})+m}  \;\; (1 \leq i \leq k) \\
&= 2^{k-i} \cdot p^{t_{k-i+1}+(k-i+1)m}\;\; (1 \leq i \leq k), \\
& \cdots, \\
\#(c_{2,n-k+2}, c_{3,n-k+3}, \cdots, c_{k-1,n-1}, c_{k,n}) &\leq 2 \cdot 2^{k-3} \cdot p^{t_{k-2}+(k-2)m+(t_{k-1}-t_{k-2})+m} \\
&= 2^{k-2} \cdot p^{t_{k-1}+(k-1)m}.\\
\#(c_{1,n-k+1}, c_{2,n-k+2}, c_{3,n-k+3}, \cdots, c_{k-1,n-1}, c_{k,n}) &\leq 2 \cdot 2^{k-2} \cdot p^{t_{k-1}+(k-1)m+(t_k-t_{k-1})+m}  \\
&=  2^{k-1} \cdot p^{t_k+km} \\
&=  2^{k-1} \cdot p^{a_k+km}.
\end{align*}
We also note that $$ \#(t_1,t_2,\cdots, t_k) \leq (\ell+(n-1)m+1)^k,$$ which will give that 
$$ \#(c_{1,k}, c_{2,k+1},\cdots, c_{n-k,n-1}, c_{n-k+1,n}) \leq 2^{n-k} \cdot (\ell+(n-1)m+1)^{n-k+1} \cdot p^{a_{n-k+1}+(n-k+1)m}$$
for every $3 \leq k \leq n-1$.

For the general case, using the same argument as above discussions, we write
\begin{equation}
\begin{aligned}
& b_{1,3}+b_{2,4}+\cdots+b_{n-2,n} = a_{n-2}+d_{n-2},\\
& b_{1,4}+b_{2,5}+\cdots+b_{n-3,n} = a_{n-3}+d_{n-3},\\
& \cdots \cdots,\\
& b_{1,k}+b_{2,k+1}+\cdots+ b_{n-k+1,n}= a_{n-k+1}+d_{n-k+1},\\
& \cdots \cdots, \\
& b_{1,n-2}+b_{2,n-1}+b_{3,n}=a_3+d_3,\\
& b_{1,n-1}+b_{2,n}=a_2+d_2.
\end{aligned}
\end{equation}
We can obtain an upper bound as follows:
$$\#(\{c_{i,j}\}) \leq 2^{\frac{(n-2)(n-1)}{2}} \cdot (\ell+(n-1)m+1)^{\frac{(n-2)(n-1)}{2}} \cdot p^{\sum_{1 \leq i<j \leq n} b_{i,j}+ \frac{(n-2)(n-1)}{2}m- \sum_{j=2}^{n-1} \max(0, d_j) },$$
where $j-i \geq 2$ and by Lemma \ref{lemma:S<<wn}, we have
$$
  \vert S_{\{b_{i,i+1}\},\{c_{i,i+1}\}}(\psi_p ; c,w_{G_n}) \rvert \leq C \cdot p^{a_1+a_2+\cdots+a_{n-2}+\frac{a_{n-1}}{2}+\frac{n(n-1)}{2}m}.
$$
Applying the above Lemma \ref{lemma: X(cw_G)}, we have
$$ \vert Kl_p(\psi_p ;c,w_{G_n}) \rvert \leq C_n \cdot p^{a_1+a_2+\cdots+a_{n-2}+\frac{a_{n-1}}{2}+\frac{n(n-1)}{2}m}.$$

  Note that in this case, since $a_{n-1}=\max(a_1,a_2,\cdots,a_{n-1})$, we always have
  \[
  a_1+a_2+\cdots+a_{n-2}+\frac{a_{n-1}}{2} \leq \ell/2 + 2a_1+(n-3)a_{n-1}+a_2+a_3+\cdots+a_{n-2}-\ell.
  \]
  Thus Theorem \ref{thm: w_n} follows from the equality
  $$Kl_p(\psi_p ;c,w_{G_n})=\sum\limits_{\{b_{i,i+1}\}, \{c_{i,i+1}\}} S_{\{b_{i,i+1}\},\{c_{i,i+1}\}}(\psi_p ; c,w_{G_n}).$$

  If we have $a_{n-1} \neq \max(a_1,a_2,\cdots,a_{n-1})$,
  by a similar argument as above, we obtain
  \[
    |S_{\{b_{i,i+1}\},\{c_{i,i+1}\}}(\psi_p ; c,w_{G_n})| \leq C  \cdot p^{a_1+a_2+\cdots+a_{n-2}+\frac{a_{n-1}}{2}+\frac{n(n-1)}{2}m}.
  \]
Applying the above Lemma \ref{lemma: X(cw_G)}, we will have
$$ \vert Kl_p(\psi_p ;c,w_{G_n}) \rvert \leq C_n \cdot p^{a_1+a_2+\cdots+a_{n-2}+\frac{a_{n-1}}{2}+\frac{n(n-1)}{2}m}.$$
Note that if $a_{n-1} < < \max(a_1,a_2,\cdots,a_{n-1})$, in other word $a_{n-1}$ is small, this bound is not good enough.
  So we have to bound this Kloosterman sum in other way.

Now we assume that $a_k:= \max(a_1,a_2,\cdots,a_{n-1})$. Since we assume that $a_{n-1} \geq a_1$, we have $2 \leq k \leq n-2$. If there exist multiple maximum, we may pick $a_k$ with the smallest index $k$. 

If $b_{1,2}+b_{2,3}+\cdots+ b_{n-1,n}> a_{k}>a_{n-1}$, then using the similar argument as above (when $a_{n-1}=\max(a_1,a_2,\cdots,a_{n-1})$), we have $p^{-\frac{b_{1,2}+b_{2,3}+\cdots+b_{n-1,n}}{2}}< p^{-\frac{a_{k}}{2}}$. From Theorem \ref{DR1998} and the proof of Theorem \ref{delta}, we know that
\begin{equation*}
\begin{aligned}
 X_{\{b_{i,i+1}\},\{c_{i,i+1}\}}(w_{G_n}c) \leq & X(w_{G_n} c) \\
=& \# N(p^m \BZ_p) \bs C(w_{G_n} c) / N(p^m \BZ_p) \\
= & \# N(p^m \BZ_p) \bs \left(N(\BQ_p) w_{G_n} c N(\BQ_p) \cap K_m \right)/ N(p^m \BZ_p) \\
 \leq & \# N(p^m \BZ_p) \bs \left(N(\BQ_p) w_{G_n} c N(\BQ_p \right) \cap K)/ N(p^m \BZ_p) \\
 = & p^{n(n-1)m} \times \# N( \BZ_p) \bs \left(N(\BQ_p) w_{G_n} c N(\BQ_p \right) \cap K)/ N(\BZ_p).\\
 = & p^{n(n-1)m} \times O_{f_0}(c),
\end{aligned}
\end{equation*}
where $O_{f_0}(c)$ is the local orbital integral defined in Theorem \ref{DR1998} and the proof of Theorem \ref{delta}. Hence, we have
\begin{equation*}
\begin{aligned}
 X_{\{b_{i,i+1}\},\{c_{i,i+1}\}}(w_{G_n}c) & \leq X( w_{G_n} c)  \leq p^{n(n-1)m} \times O_{f_0}(c) \\
& \leq p^{n(n-1)m} \times p^{a_1+a_2+\cdots+ a_{n-1}} \times R(c) \\
& \leq p^{n(n-1)m} \times p^{a_1+a_2+\cdots+ a_{n-1}} \times \left( a_1+a_2+\cdots+a_{n-1}+n \right)^{\frac{n^3}{2}} \\
& \leq p^{n(n-1)m} \times p^{a_1+a_2+\cdots+ a_{n-1}} \times \left( (n-1) \ell +n \right)^{\frac{n^3}{2}}.
\end{aligned}
\end{equation*}
Hence by Lemma \ref{lemma:S<<wn}, we have
\begin{equation}
\begin{aligned}
  \vert S_{\{b_{i,i+1}\},\{c_{i,i+1}\}}(\psi_p ; c,w_{G_n}) \rvert & \leq C \times p^{a_1+a_2+\cdots+a_{n-2}+a_{n-1}-\frac{a_{k}}{2}+\frac{n(n-1)}{2}m} \\
&= C \times p^{a_1+a_2+\cdots+a_{k-1}+\frac{a_k}{2}+a_{k+1}+\cdots+ a_{n-2}+a_{n-1} +\frac{n(n-1)}{2}m}.
\end{aligned}
\end{equation}
Hence, this proves the case for $b_{1,2}+b_{2,3}+\cdots+ b_{n-1,n}> a_{k}$.

If $b_{1,2}+b_{2,3}+\cdots+ b_{n-1,n} \leq a_{k}$, for this case, we consider extra $n-1$ submatrices of $M_{n-1}$ as follows:
$$ U_0:= ( 1 );$$
$$ U_1:= ( p^{-b_{n-1,n}}c_{n-1,n} ),\,\text{which is the bottom right $1 \times 1$ block matrix in $M_{n-1}$}\, ;$$
$$ U_2:= \begin{pmatrix} p^{-b_{n-2,n-1}}c_{n-2,n-1} & p^{-b_{n-2,n}}c_{n-2,n} \\ 1 & p^{-b_{n-1,n}}c_{n-1,n} \end{pmatrix},\,\text{which is the bottom right $2 \times 2$ block matrix in $M_{n-1}$}\, ;$$
$$ U_3:= \,\text{the bottom right $3 \times 3$ block matrix in $M_{n-1}$}\, ;$$
$$ \cdots \cdots;$$
$$ U_{n-2}:= \,\text{the bottom right $(n-2) \times (n-2)$ block matrix in $M_{n-1}$}.$$
Applying Lemma \ref{lemma: stevens2} to the submatrices $U_1, U_2, \cdots, U_{n-2}$, we see that
\[
p^{a_{n-1}} \vert U_1 \rvert \in p^m \BZ_p,\ p^{a_{n-1}} \vert U_2 \rvert \in p^m \BZ_p,\ \cdots,\ p^{a_{n-1}} \vert U_{n-2} \rvert \in  p^m \BZ_p.
\]
By our assumption, we have $a_k= \max(a_1,a_2,\cdots,a_{n-1})$. We will focus on the following $k-1$ elements:
\[
p^{-b_{k,n}}c_{k,n},\ p^{-b_{k-1,n-1}}c_{k-1,n-1},\ \cdots,\ p^{-b_{2, n-k+2}}c_{2,n-k+2}.
\]
We have the following cases:

If $b_{k,n} \leq a_{n-1}$, then we have $\#(c_{k,n}) \leq p^{a_{n-1}+m}$.

If $b_{k,n} > a_{n-1}$, since $p^{a_{n-1}} \vert U_{n-k} \rvert \in p^m \BZ_p$ and the coefficient of $p^{-b_{k,n}}c_{k,n}$ is given by $(-1)^{n-k+1} \cdot U_0=(-1)^{n-k+1} \cdot (1)= (-1)^{n-k+1}$, then we have
\[
\#(c_{k,n}) \leq p^{b_{k,n}-(b_{k,n}-a_{n-1})+m} \leq p^{a_{n-1}+m}.
\]

Following above idea, we continue our induction steps. If $b_{k-j,n-j} \leq a_{n-1}$, then we have $\#(c_{k-j,n-j}) \leq p^{a_{n-1}+m}$ for $0 \leq j \leq k-2$.

If $b_{k-j,n-j} > a_{n-1}$, since $p^{a_{n-1}} \vert U_{n-k+j} \rvert \in p^m \BZ_p$ and the coefficient of $p^{-b_{k-j,n-j}}c_{k-j,n-j}$ is given by above matrix $(-1)^{n-k+1} \cdot \vert U_j \rvert$. Moreover, without loss of generality, we can assume that the norm of the determinant of $U_j$ satisfies $\vert \det(U_j) \rvert \geq p^{-(n-1)m}$ for every $0 \leq j \leq n-2$. Otherwise, we may again substitute the element $x_{\{b_{i,j}\}}^{\{c_{i,j}\}} \in C(w_{G_n}c)$ by $x_{\{b_{i,j}\}}^{\{c_{i,j}\}} u' \in C(w_{G_n}c)$, where $u'$ is an element in $N(p^m \BZ_p) \subseteq N(\BZ_p) \subseteq N(\BQ_p)$. For example, if the norm of the determinant of $U_i$ (certain $1 \leq i \leq n-2$) satisfies $\vert \det(U_i) \rvert < p^{-(n-1)m}$, we may focus on the submatrix $U_i$ and substitute the element $p^{-b_{n-i,n-i+1}}c_{n-i,n-i+1}$ by $p^{-b_{n-i,n-i+1}}c_{n-i,n-i+1}+p^m$. This is given by the right multiplication of $u'=(u_{i,j})$ with $u_{t,t}=1$ ($1 \leq t \leq n$), $u_{n-i,n-i+1}=p^m$ and all the other entrices equal to zero. We note that $x_{\{b_{i,j}\}}^{\{c_{i,j}\}}$ and $x_{\{b_{i,j}\}}^{\{c_{i,j}\}} u'$ represent the same element in $X(w_{G_n}c)$. After the right multiplication, we will have the norm of the determinant of $U_i$ satisfies $\vert \det(U_i) \rvert \geq p^{-m} \times \vert \det(U_{i-1}) \rvert$. By induction, we have the determinant $\vert \det(U_i) \rvert \geq p^{-im} \times \vert \det(U_0) \rvert \geq p^{-im} \times 1 = p^{-im} \geq p^{-(n-1)m}$. These right multiplications of $u'$ will give at most $2^{n-2}$ ($\{0,p^m\}^{n-2}$) different cases. Therefore we have
\[
\#(c_{k-j,n-j}) \leq  p^{b_{k-j,n-j}-(b_{k-j,n-j}-a_{n-1})+(n-1)m} \leq  p^{a_{n-1}+(n-1)m},
\]
for every $0 \leq j \leq k-2$. Now we vary $j$ from $0 \leq j \leq k-2$, we will have
\[
\#(c_{2,n-k+2}, c_{3,n-k+3},\cdots, c_{k-1,n-1}, c_{k,n}) \leq 2^{n-2} \times p^{(k-1)a_{n-1}+(n-1)(k-1)m}.
\]

It remains to bound the remaining $\#(c_{i,j})$ for $j-i \geq 2$.

Since $b_{1,2}+b_{2,3}+\cdots+b_{n-1,n} \leq  a_k$, then using the similar argument as above (We recall that $ \#(c_{1,j}, c_{2,j+1},\cdots, c_{n-j,n-1}, c_{n-j+1,n}) \leq 2^{n-j} \cdot (\ell+(n-1)m+1)^{n-j+1} \cdot p^{a_{n-j+1}+(n-j+1)m}$ for every $3 \leq j \neq n+1-k \leq n-1$ and $\#(c_{2,n-k+2}, c_{3,n-k+3},\cdots, c_{k-1,n-1}, c_{k,n}) \leq 2^{n-2} \cdot p^{(k-1)a_{n-1}+(n-1)(k-1)m}$), we will obtain the bound as follows: 
\begin{equation}
\begin{aligned}
\#(\{c_{i,j}\}) \leq & 2^{\frac{n(n-1)}{2}} \times (\ell+(n-1)m+1)^{\frac{(n-2)(n-1)}{2}} \\
& \times p^{ 2a_1+a_2+\cdots+a_{k-1}+a_{k+1}+\cdots+a_{n-2}+(k-1)a_{n-1}+(k-1)(n-1)m+\frac{(n-2)(n-1)}{2}m},
\end{aligned}
\end{equation}
for $j-i \geq 2$. Since $2 \leq k \leq n-2$, we will have
\begin{equation}
\begin{aligned}
\#(\{c_{i,j}\}) \leq & 2^{\frac{n(n-1)}{2}} \times (\ell+(n-1)m+1)^{\frac{(n-2)(n-1)}{2}} \\
& \times p^{ 2a_1+a_2+\cdots+a_{k-1}+a_{k+1}+\cdots+a_{n-2}+(n-3)a_{n-1}+(n-3)(n-1)m+\frac{(n-2)(n-1)}{2}m},
\end{aligned}
\end{equation}
where $j-i \geq 2$. Hence by Lemma \ref{lemma:S<<wn}, we have
$$
  \vert S_{\{b_{i,i+1}\},\{c_{i,i+1}\}}(\psi_p ; c,w_{G_n}) \rvert \leq C \times p^{ a_k /2 + 2a_1+a_2+\cdots+a_{k-1}+a_{k+1}+\cdots+a_{n-2}+(n-3)a_{n-1}+\frac{n(n-1)}{2}m}.
$$
Applying the above Lemma \ref{lemma: X(cw_G)}, we have
$$ \vert Kl_p(\psi_p;c,w_{G_n}) \rvert \leq C_n \times p^{ a_k /2 + 2a_1+a_2+\cdots+a_{k-1}+a_{k+1}+\cdots+a_{n-2}+(n-3)a_{n-1}+\frac{n(n-1)}{2}m}.$$

This proves the main inequality in our Theorem \ref{thm: w_n}.

For the second claim, we recall that $a_k=\max(a_1,a_2,\cdots,a_{n-1}) \geq a_{n-1} \geq a_1$. We note that if $a_k \geq (2n-5)a_{n-1}$ (Note that $n \geq 3$), the following inequality holds:
$$a_1+a_2+\cdots+a_{n-2}+\frac{a_{n-1}}{2} \geq a_k /2 + 2a_1+(n-3)a_{n-1}+a_2+a_3+\cdots+a_{k-1}+a_{k+1}+\cdots+ a_{n-2}.$$
Therefore, if $a_k \leq (2n-5)a_{n-1}$, we will have
\begin{equation*}
\begin{aligned}
a_1+a_2+\cdots+a_{n-2}+\frac{a_{n-1}}{2}  & \leq \left( 1- \frac{0.5}{(2n-5)(n-3)+2} \right) \cdot (a_1+a_2+\cdots+a_{n-2}+a_{n-1}) \\
& \leq \left( 1- \frac{0.5}{(2n-5)(n-2)+1} \right) \cdot (a_1+a_2+\cdots+a_{n-2}+a_{n-1}) \\
& = \left( 1- \frac{1}{4n^2-18n+22} \right) \cdot (a_1+a_2+\cdots+a_{n-2}+a_{n-1}).
\end{aligned}
\end{equation*}
If $a_k \geq (2n-5)a_{n-1}$, we will also have (Note that $n \geq 3$)
\begin{equation*}
\begin{aligned}
\ell/2 + 2a_1+(n-3)a_{n-1}+\sum_{j=2,j\neq k}^{n-2}a_j
\leq & \left( 1- \frac{0.5}{(2n-5)(n-3)+2} \right) \cdot \sum_{j=1}^{n-1}a_j\\
\leq & \left( 1- \frac{0.5}{(2n-5)(n-2)+1} \right) \cdot \sum_{j=1}^{n-1}a_j\\
= & \left( 1- \frac{1}{4n^2-18n+22} \right) \cdot \sum_{j=1}^{n-1}a_j.
\end{aligned}
\end{equation*}
This proves the second claim and therefore Theorem \ref{thm: w_n}.

\end{proof}

\begin{rmk}
Note that the trivial bound for the local Kloosterman sum (integral) (see \cite{DR98}) is the following:
$$ \vert Kl_p(\psi_p;c,w_{G_n}) \rvert \leq A_{\epsilon} \cdot p^{(1+\epsilon)(a_1+a_2+a_3+\cdots+a_{n-1})},\; \epsilon>0$$
where $A_{\epsilon}$ is a positive constant independent on the choice of $c$. Since $1-\frac{1}{4n^2-18n+22} <1$, so applying Steven's method, we get a nontrivial bound for the local Kloosterman sum (integral).
\end{rmk}

\begin{rmk}
  The result is not optimal when $n \geq 4$. When $n=2$ and $n=3$, the constant $\frac{1}{2}$ and $\frac{1}{4}$ are sharp (See \cite{Ba97} and \cite{DF97}).
  To improve the bound in other cases, one may use the stationary phase formulas as Dabrowski and Fisher did for $\GL(3)$ (\cite{DF97}). It is also a very interesting question to ask whether we can improve the exponent $1-\frac{1}{4n^2-18n+22}$ to some $1- \delta$, where $0< \delta <1$ and $\delta$ is independent on the choice of the positive integer $n$ (See \cite{Ste87} for the conjecture: $\delta=\frac{1}{4}$).
\end{rmk}

\begin{rmk}
Theorem \ref{thm: w_n}, Theorem \ref{thm: w_8} (Non-trivial upper bound for general Kloosterman sums), Theorem \ref{inte} (Local integrability of Bessel functions) and all the results in Section 4 (See also \cite[Section 4]{Ste87}) are expected to be true for $\GL_n(F)$, where $F$ is a general $p$-adic local field. We hope to come back to these generalizations in the near future.
\end{rmk}

\begin{rmk}
In Theorem \ref{thm: w_n} and Section 4, we only consider the special non-degenerated additive character $\psi(\sum_{i=1}^{n-1} u_{i,i+1})$, since all non-degenerated additive characters are in the same orbit under the action of diagonal matrices $T$. If we consider the general additive characters, we write the non-degenerated additive characters $\psi_p$ and $\psi_p'$ of $N(\mathbb{Q}_p)$ which are trivial on $N(p^m \mathbb{Z}_p)$ as follows:
\begin{equation} 
      \psi_p\left(\begin{pmatrix} 1&u_1&*&\cdots&*\\ &1&u_2&\cdots&*\\ &&\cdots&\cdots&\cdots \\ &&&1&u_{n-1}\\ &&&&1\end{pmatrix}\right) = \xi( \nu_1 u_1+ \nu_2 u_2+ \nu_3 u_3+\cdots+ \nu_{n-1} u_{n-1}),
\end{equation}
and
\begin{equation} 
      \psi_p' \left(\begin{pmatrix} 1&u_1&*&\cdots&*\\ &1&u_2&\cdots&*\\ &&\cdots&\cdots&\cdots \\ &&&1&u_{n-1}\\ &&&&1\end{pmatrix}\right) = \xi( \nu_1' u_1+ \nu_2' u_2+ \nu_3' u_3+\cdots+ \nu_{n-1}' u_{n-1}),
\end{equation}
where $\nu_1,\nu_2,\nu_3, \cdots, \nu_{n-1}$, $\nu_1', \nu_2',\mu_3',\cdots, \nu_{n-1}'$ $ \in p^{-m}\mathbb{Z}_p- \{0 \}$. Moreover, we further assume that $p^{-m} \leq \vert \nu_i \rvert \leq p^{m}$ and $p^{-m} \leq \vert \nu_i' \rvert \leq p^{m}$ for all $1 \leq i, i'  \leq n-1$. Here $m$ is same as the $m$ that we defined in previous Section 3 and 4. We have the following non-trivial upper bound for $\GL(n)$ generalized Kloosterman sums on non-degenerated additive characters:
Let $\ell=
  \max(a_1,a_2,\cdots,a_{n-1}) \geq m$,
$\varrho=\max(a_{n-1}, a_1)$, $\sigma=\min(a_{n-1}, a_1)$,
and
  \begin{equation}
\begin{aligned}
 D_n :=  2^{n^2-1} \cdot p^{2(n+3)(n-1)m} \cdot \left(\prod_{j=1}^{n-1} (|\nu_j\nu_{n-j}' p^{2m}|_p^{-1},p^{\ell+m})^{1/2}\right)
\cdot (\ell+(n-1)m+1)^{(n^2-1)} \cdot ((n-1)\ell+n)^{\frac{n^3}{2}}.
\end{aligned}
  \end{equation}
 Then
  \begin{equation}
    \begin{split}
      |Kl_p(\psi_p,\psi_p';c,w_{G_n})|
      & \leq
        D_n \cdot \min(p^{ \sigma +a_2+\cdots+ \varrho/2+\frac{n(n-1)}{2}m},p^{ \ell/2+ 2 \sigma+(n-3)\varrho +a_2+\cdots+a_{n-2}-\ell +\frac{n(n-1)}{2}m}).
    \end{split}
  \end{equation}
  In particular, we have $|Kl_p(\psi_p,\psi_p';c,w_{G_n})|\leq D_n \cdot p^{(1-\frac{1}{4n^2-18n+22})\cdot(a_1+a_2+a_3+\cdots+a_{n-1})+\frac{n(n-1)}{2}m}.$
\end{rmk}

\section{Proof of Local Integrability of Bessel functions for $\GL(n)$}

In this section, we use the results in the previous Section 5 to prove the local integrability of Bessel functions for $\GL(n)$ ($n \geq 4$). The local integrability of Bessel functions has only been proved for $\GL(2)$ \cite{So84} \cite{Ba97} and $\GL(3)$ \cite{Ba04}. The main ingredients of their proofs are some basic analytic properties of the Bessel functions and orbital integrals. Moreover, an explicit expression for the $\GL(3)$ relative Shalika germ $K_e^{w_{G_3}}$ (see \cite{JY99}) and the $p$-adic stationary phase methods are also needed in \cite{Ba04}. However, an explicit expression for the (relative) Shalika germs is invalid for $\GL(n)$ ($n \geq 4$) and the $p$-adic stationary phase calculations become very complicated. In order to overcome these obstacles, we apply Stevens' method in \cite{Ste87} to bound the (relative) Shalika germs $K_e^{w_{G_n}}$. The key ingredient of our proof is a nontrivial bound for the local Kloosterman sum (integral) attached to the longest Weyl element which we have already established in Theorem \ref{thm: w_n}.

First, let's recall the Conjecture \ref{conj1} in Section 3.

\begin{conj}  \label{conj6} [Conjecture \ref{conj1}]
Fix $f \in C_c^{\infty}(G_n)=C_c^{\infty}(\GL_n(\BQ_p))$. Then $\vert I_{f, \psi}(g) \Delta^{\frac{1}{2}-\delta_n}(g) \rvert $ is bounded on compact sets in $G_n$ for some given $\delta_n>0$.
\end{conj}

According to the discussion in Section 3, we know that Conjecture \ref{conj6} is equivalent to the local integrability of Bessel functions for $\GL(n)$. Applying Jacquet-Ye's theory of the relative Shalkia germs (Theorem \ref{germ} and \cite{JY96},\cite{JY99}), we can reduce Conjecture \ref{conj1} to Conjecture \ref{conj2} from the discussion in Section 3. Now it suffices to prove Conjecture \ref{conj2}.

We will give the proof of Conjecture \ref{conj2}. Therefore, we prove the local integrability of Bessel functions for $\GL_n(\BQ_p)$. We state the following theorem.

\begin{thm}
The absolute value $\left \vert K_e^{w_{G_n}} \Delta^{\frac{1}{2}-\delta} \right \rvert$ is uniformly bounded on the set $A_e^{w_{G_n}}$ for any $\delta$ satisfying $0<\delta< \frac{1}{8n^2-36n+44}$.
\end{thm}

\begin{proof}
According to the definition of (relative) Shalika germ $K_e^{w_{G_n}}$ in \cite{JY96} and \cite{JY99}, for any element $g$ in the set $A_e^{w_{G_n}}$, without loss of generality, we can assume that $\vert \Delta_i(g) \rvert < \epsilon$, where $1 \leq i \leq n-1$ and $\epsilon$ is any arbitrary small positive real number (By the explicit formula in Section 3.4 and \cite[Proposition 2.6]{JY99}, for each $k \geq m'$, there is an inductive system of Shalika germs such that $K_e^{w_{G_n}}$ is only supported on the set $\vert \Delta_i(g) \rvert < p^{-k}$, where $1 \leq i \leq n-1$. By picking $k$ large enough, we can make the above assumption). Here $\Delta_i$ is the determinant of principal $i \times i$ minor of $g$ which we have already defined in Section 3. Moreover, as in Section 3 and 4, we know that $\Delta_n(g)=\det w_{G_n}=(-1)^{\frac{(n+1)(n+2)}{2}+1}$. Hence, we have $\vert \Delta_n(g) \rvert=1$.

Note that $\vert \Delta_i(g) \rvert < \epsilon$ for all $1 \leq i \leq n-1$, we can write $g=c \in T$. From above assumption, we let
$$
c:=\begin{pmatrix}
         p^{a_1}v_1 &  &  &  & \\
          & p^{a_2-a_1}v_2 &  &  & \\
          & & \cdots & & \\
          &  &  & p^{a_{n-1}-a_{n-2}}v_{n-1} &  \\
          &  &  &  & p^{-a_{n-1}}v_n
       \end{pmatrix},
$$
where $v_i \in \BZ_p^{\times}$ for $i=1,2,3,\cdots,n$. Here $a_1,a_2,\cdots,a_{n-1}$ are all large enough positive integers. Moreover, we have $a_1,a_2,\cdots,a_{n-1} >> m>2m'$ if we let positive real number $\epsilon$ to be small enough, where $m$ is a fixed positive integer in Section 4.

By direct calculation, we have
$$
\Delta(c):= \left \vert \frac{(\Delta_1(c))^2 \cdot (\Delta_2(c))^2 \cdots (\Delta_{n-1}(c))^2}{(\Delta_n(c))^2} \right \rvert=p^{-2(a_1+a_2+a_3+\cdots+a_{n-1})}.
$$
So we have $\Delta^{\frac{1}{2}-\delta}(g) = \Delta^{\frac{1}{2}-\delta}(c)=p^{-(1-2\delta)(a_1+a_2+a_3+\cdots+a_{n-1})}$.

Now applying Theorem \ref{thm: w_n}, we further assume that $a_1,a_2,\cdots,a_{n-1}>>m>2m'$ (For example, $a_j>100000(n-1)m$ for $j=1,2,\cdots,n-1$).
Therefore, we have
\begin{equation}
\begin{aligned}
\vert K_{e,\psi_p}^{w_{G_n}}(c) \rvert < \vert Kl_p(\psi_p^{-1} ;c,w_{G_n}) \rvert & \leq  2^{n^2-1} \cdot (\ell+(n-1)m+1)^{(n^2-1)} \cdot p^{(1-\frac{1}{4n^2-18n+22})\cdot(a_1+a_2+\cdots+a_{n-1})} \\
& \qquad\qquad\times p^{(n-1)m+2(n+3)(n-1)m} \times ((n-1) \ell+n)^{\frac{n^3}{2}} \times p^{\frac{n(n-1)}{2}m} \\
& <  2^{n^2-1} \cdot (\ell+(n-1)m+1)^{(n^2-1)} \cdot p^{(1-\frac{1}{4n^2-18n+22})\cdot(a_1+a_2+\cdots+a_{n-1})} \\
& \qquad\qquad\times p^{(n-1)m+3(n+3)(n-1)m} \times ((n-1) \ell+n)^{\frac{n^3}{2}} \\
& < 2^{n^2-1} \cdot (\ell+(n-1)m+1)^{(n^2-1)} \cdot p^{(1-\frac{1}{4n^2-18n+22})\cdot(a_1+a_2+\cdots+a_{n-1})} \\
& \qquad\qquad\times p^{3(n+4)(n-1)m} \times ((n-1) \ell+n)^{n^3},
\end{aligned}
\end{equation}
where $\ell=\max(a_1,a_2,a_3,\cdots,a_{n-1})$.

In conclusion, for any $g \in A_e^{w_{G_n}}$, we have
\begin{equation}
\begin{aligned}
\vert K_e^{w_{G_n}}(g)\Delta^{\frac{1}{2}-\delta}(g) \rvert& = \vert K_e^{w_{G_n}}(c)\Delta^{\frac{1}{2}-\delta}(c) \rvert \\
&\leq 2^{n^2-1} \cdot p^{3(n+4)(n-1)m} \cdot (\ell+(n-1)m+1)^{(n^2-1)} \cdot ((n-1) \ell+n)^{n^3} \\
& \qquad\times p^{-(\frac{1}{4n^2-18n+22}-2\delta)(a_1+a_2+a_3+\cdots+a_{n-1})}.
\end{aligned}
\end{equation}
Now since $0<\delta< \frac{1}{8n^2-36n+44}$, we have $0<\frac{1}{4n^2-18n+22}-2\delta<\frac{1}{4n^2-18n+22}$. This means that the positive real number
\[
2^{n^2-1} \cdot p^{3(n+4)(n-1)m} \cdot (\ell+(n-1)m+1)^{(n^2-1)} \cdot ((n-1) \ell+n)^{n^3} \cdot p^{-(\frac{1}{4n^2-18n+22}-2\delta)(a_1+a_2+a_3+\cdots+a_{n-1})}
 \]
 approaches to $0$ when $a_1,a_2,\cdots,a_{n-1}$ tends to infinity. Hence, it is bounded uniformly in terms of $a_1,a_2,\cdots,a_{n-1}$. This finishes the proof.

\end{proof}

\begin{rmk}
In our proof, we further assume that $\vert \Delta_i(g) \rvert < \epsilon$, where $1 \leq i \leq n-1$ and $\epsilon$ is any arbitrary small positive real number. Actually, in order to apply Theorem \ref{thm: w_n}, it is sufficient to pick $\epsilon=p^{-m}<p^{-m'}$, where $m$ is a positive integer defined in Section 4. Hence, $a_1,a_2,\cdots,a_{n-1}$ are all large enough positive integers. Moreover, we have $a_1,a_2,\cdots,a_{n-1} \geq m \geq 2m'$. The remaining part of the proof keeps the same. Intuitively, by Theorem \ref{germ}, the relative Shalika germ $K_e^{w_{G_n}}$ contributes to the asympototic behaviour of certain local orbital integral. By the definition of $K_e^{w_{G_n}}$ and the set $A_e^{w_{G_n}}$, we can assume that $\vert \Delta_i(g) \rvert < \epsilon$, where $1 \leq i \leq n-1$ and $\epsilon$ is any arbitrary small positive real number (See Remark \ref{gl4}(g) for an example on $\GL(4)$).
\end{rmk}

\section{Applications of Theorem \ref{inte}}

In this section, we will give some applications of our main Theorem \ref{inte}. We give the definition of Bessel distributions and will follow the notations in Section 2 and Section 3.

Let $\pi$ be a smooth irreducible generic representation of $G_n=\GL_n(\BQ_p)$ with contragredient $\widetilde{\pi}$. We use $\pi^{*}$ and $\widetilde{\pi}^{*}$ to denote the linear dual of $\pi$ and $\widetilde{\pi}$ respectively. Let $f \in C_c^{\infty}(G_n)$ be a locally constant function with compact support on $G_n$. We take $l \in \pi^{*}$ and $l' \in \widetilde{\pi}^{*}$ to be fixed nonzero Whittaker functionals with repect to the non-degenerate additive character $\psi$ and $\psi^{-1}$, respectively. We define $\widetilde{\pi}(f)l'$ as
$$ \widetilde{\pi}(f)l':= \int_{G_n} f(g) \widetilde{\pi}(g)l' dg$$
or equivalently, for any $\widetilde{v} \in \widetilde{\pi}$, we have
\begin{equation}
\begin{aligned}
\langle \widetilde{\pi}(f)l', \widetilde{v} \rangle &= \int_{G_n} f(g) \langle \widetilde{\pi}(g)l', \widetilde{v} \rangle dg \\
&= \int_{G_n} f(g) \langle l', \widetilde{\pi}(g^{-1})(\widetilde{v}) \rangle dg.
\end{aligned}
\end{equation}
Then $\widetilde{\pi}(f)l'$ is a smooth linear functional on $\widetilde{\pi}$, hence can be identified with a vector $v_f \in \pi$.

\begin{defn}
We define Bessel distribution $B(f)$ as
$$ B(f):= l(v_f).$$
\end{defn}
E. M. Baruch obtained the first regularity result about the Bessel distribution $B(f)$. By Theorem 2.3 in \cite{Ba01}, when restricted to the open Bruhat cell $\Omega=B w_{G_n} B$, this Bessel distribution $B(f)$ is given by integration against a locally constant kernel function $j_{0,\pi}(g)$ on $\Omega$, which is called the relative Bessel function. In other words, for any $f \in C_c^{\infty}(\Omega)$, we have $$ B(f)= \int_{\Omega} j_{0,\pi}(g) f(g) dg.$$ We can extend this function $j_{0,\pi}(g)$ to the whole group $G_n$ by letting zero when $g \in G_n - \Omega=G_n- B w_{G_n} B$. By Theorem 1.1 in \cite{Chai19a} and Theorem 7.2 in \cite{Chai17}, we know that $j_{0,\pi}(g)=j_{\pi}(g)$ for all $g \in \Omega$ after certain normalizations on $l$ or $l'$ (See Theorem 1.1, 3.2 in \cite{Chai19a} and Lemma 3.2 in \cite{Chai17}).

We have the following regularity theorem, which naturally connects the Bessel function $j_{\pi}(g)$ to the above Bessel distribution $B(f)$.

\begin{prop}  \label{kernel}
The Bessel distribution $B(f)$ is given by integration against the Bessel function $j_{\pi}(g)$ on $G_n$ (See Section 2), that is, for any $f \in C_c^{\infty}(G_n)$, we have
$$ B(f)= \int_{G_n}  j_{\pi}(g) f(g) dg.$$
\end{prop}

\begin{proof}
For $f \in C_c^{\infty}(G_n)$, we define the following distribution
$$ B_1(f):= \int_{G_n} j_{\pi}(g) f(g) dg.$$
Note that the Bessel function $j_{\pi}(g)$ is local integrable on $G_n$ (Main Theorem \ref{inte}), this distribution is well defined. Now we consider the distribution $\widetilde{B}(f):=B(f)-B_1(f)$. By Theorem 2.3 in \cite{Ba01}, Theorem 1.1 in \cite{Chai19a} and Theorem 7.2 in \cite{Chai17}, the restriction of this distribution $\widetilde{B}(f)$ to the open Bruhat cell $\Omega= B w_{G_n} B$ is zero. Therefore, this distribution is supported on $G_n- \Omega=G_n- B w_{G_n} B$. Now by Theorem A in \cite{AGS15}, the wave-front set of the distribution $\widetilde{B}= B-B_1$ is contained in $(G_n- \Omega) \times \CN$, where $\CN$ is the set of nilpotent elements (cone) in the dual of Lie algebra of $G_n$. Hence, by Corollary B and C in \cite{AGK15}, we have $\widetilde{B} \equiv B-B_1 \equiv 0$, which gives that $B \equiv B_1$. This finishes the proof.
\end{proof}

\begin{rmk}
The regularity theorem (Proposition \ref{kernel}) may have applications in the study of relative trace formulae (especially Kuznetsov trace formula). Moreover, the regularity theorem is expected to hold for general $G(F)$, where $F$ is a general $p$-adic local field and $G$ is a connected split reductive group.
\end{rmk}

We will give an application of above Proposition \ref{kernel}.

\begin{cor}
Let $\pi_1$ and $\pi_2$ be two generic smooth irreducible representations of $G_n$, with the corresponding Bessel functions $j_{\pi_1}(g)$ and $j_{\pi_2}(g)$. If there exists a non-zero constant $c$, such that for any $g \in G_n$, we have $j_{\pi_1}(g)=c j_{\pi_2}(g)$. Then $\pi_1 \cong \pi_2$.
\end{cor}

\begin{proof}
By Lemma 2.2 part (2) in \cite{FLO12}, the Bessel distributions are linear independent for two inequivalent smooth irreducible generic representations. This Corollary now follows directly from above Proposition \ref{kernel}.
\end{proof}
We recall that $N= N_n$ is the standard unipotent radical subgroup of $G_n$ and $w_{G_n}$ is the longest Weyl element in $G_n$. We embed $G_{n-1}$ into $G_n$ on the upper and left corner by the map $g_{n-1} \rightarrow \begin{pmatrix} g_{n-1} & \\ &1 \end{pmatrix}$. We have the following kernel formula.

\begin{cor}  \label{kerformula}
Let $\pi$ be a generic smooth irreducible representation of $G_n$ with the corresponding Bessel function $j_{\pi}(g)$. If $W_v(g)$ is a Whittaker function in the Whittaker model $\CW(\pi,\psi)$, then for any diagonal matrix $b= \diag(b_1,\cdots, b_n) \in A_n \subseteq G_n$ and $W_v \begin{pmatrix} h & \\ & 1 \end{pmatrix} \in C_c^{\infty}(N_{n-1} \bs G_{n-1}, \psi)$, we have
$$  W_v(b w_{G_n})= \int_{N_{n-1} \bs G_{n-1}} j_{\pi} \left( b w_{G_n} \begin{pmatrix} h^{-1} & \\ & 1 \end{pmatrix} \right) W_v \begin{pmatrix} h & \\ & 1 \end{pmatrix} dh.$$
\end{cor}
\begin{proof}
By applying the local integrability of the Bessel function $j_{\pi}(g)$ and using the same method in Lemma 5.3 of \cite{Ba04}, we can show that the right hand side of above integral is absolutely convergent. This corollary now follows directly from Theorem 4.2 (the weak kernel formula) in \cite{Chai17}.
\end{proof}

\begin{rmk}
The kernel formula in the above Corollary \ref{kerformula} actually gives the action of the longest Weyl element $w_{G_n}$ on the Kirillov model of $\pi$.
\end{rmk}

\section{Acknowledgement}
I would like to thank my advisor Professor D. Jiang for suggesting me thinking about this interesting problem, providing fruitful comments and suggestions that lead to the solution of the problem, and carefully reviewing the first draft of this paper. I would like to thank Professor A. Diaconu for very helpful discussions and answering my several questions on the (twisted) $\GL(2)$ Kloosterman sums in Section 4. I would like to thank Professor V. Blomer, Professor J. Buttcane and Professor H. Jacquet for many useful comments and corrections. I also acknowledges my doctoral school University of Minnesota and Max-Planck-Institut {f\"{u}r} Mathematik where most of the work has been done. Moreover, I would like to thank the anonymous referee for careful and detailed reading and comments to improve the presentation and fix several inaccurate statements of my paper. The work is supported in part through the NSF Grant: DMS-1901802 of Professor D. Jiang.

\appendix

\section{$\GL(4)$ Kloosterman Sums}

In this Appendix, we follow \cite{Ste87} to bound the $\GL(4)$ Kloosterman sums introduced in the previous section. The main ideas and ingredients of the proof keep the same as that of Section 5. By a more careful and delicate estimation, the bound we get in Theorem \ref{thm: w_8} for $\GL(4)$ is slightly stronger than that in Theorem \ref{thm: w_n}. We apply the results in Section 4 for $n=4$.

For $w\in W_{G_4}$, we recall that $w(j)$, $j\in\{1,2,3,4\}$ is given by the formula
\[
  w\cdot e_j =  e_{w(j)},
\]
where $e_1,e_2,e_3,e_4$ is the standard basis of column vectors.
We recall the definition of non-degenerated additive character $\psi_p$ on $N(\mathbb{Q}_p)$ which is trivial on $N(p^m \mathbb{Z}_p)$ as follows:
\begin{equation}\label{eqn: psi}
      \psi_p\left(\begin{pmatrix} 1&u_1&*&*\\ &1&u_2&*\\ &&1&u_3\\ &&&1\end{pmatrix}\right) = \xi( u_1+ u_2+ u_3).
\end{equation}
Here $m$ is same as the $m$ that we defined in previous Section 3 and 4. The definition of additive character $\psi_p'$ is given in a similar way.

Fix
\begin{equation}\label{eqn: c1}
  \widetilde{c}=\diag( p^{s}v_4, p^{r-s}v_3, p^{t-r}v_2, p^{-t}v_1) \in T,
\end{equation}
where $v_i \in \BZ_p^{\times}$ for $i=1,2,3,4$ and $v_1v_2v_3v_4=1$. Moreover, we further assume that $t,r,s \geq m$, where $m$ is a fixed positive integer.

\begin{thm}\label{thm: w_8}
  Let $Kl_p(\psi_p ;c,w_{G_4})$ be the local Kloosterman sum attached to the longest Weyl element $w_{G_4}$.
  Let $\psi_p$ be as in \eqref{eqn: psi}, $\ell=
  \max(r,s,t) \geq m$, $\varrho=\max(t,s)$, $\sigma=\min(t,s)$, and
\begin{equation*}
\begin{aligned}
 C_8 &:= 8  p^{9m} \cdot ( p^{2m},p^{\ell+m})^{1/2} ( p^{2m},p^{\ell+m})^{1/2} ( p^{2m},p^{\ell+m})^{1/2} \\
& \cdot (\ell+m+1)^3 (\varrho+m+1)(r+m+1)^2(\sigma+m+1)^2 \\
&= 8  p^{9m} \cdot p^m \cdot p^m \cdot p^m \\
& \cdot (\ell+m+1)^3 (\varrho+m+1)(r+m+1)^2(\sigma+m+1)^2 \\
&= 8  p^{12m} \cdot (\ell+m+1)^3 (\varrho+m+1)(r+m+1)^2(\sigma+m+1)^2.
\end{aligned}
  \end{equation*}
  Then
  \begin{equation}\label{eqn:Kl-w8}
    \begin{split}
      |Kl_p(\psi_p ;\tilde{c},w_{G_4})|
      & \leq
        C_8 \cdot \min(p^{r+\sigma+\varrho/2+3m},p^{\varrho+3\sigma/2+r/2+3m}).
    \end{split}
  \end{equation}
  In particular, we have $|Kl_p(\psi_p ;\tilde{c},w_{G_4})|\leq C_8 \cdot p^{7(t+r+s)/8+3m}.$
\end{thm}

We start with the following observation of the matrix identity:

\begin{equation}  \label{equ:identity}
\begin{aligned}
g_0:= & \begin{pmatrix}
  \frac{-a_1}{xyz-xv-uz+w} & & & \\
  \frac{a_2(yz-v)}{uv-wy} & \frac{a_2(xyz-xv-uz+w)}{uv-wy} & & \\
  \frac{-a_3 v}{w} & \frac{a_3 (w-xv)}{w} & \frac{a_3 (wy-uv)}{w} & \\
  a_4 & a_4 x & a_4 u & a_4 w
 \end{pmatrix} \\
= & \begin{pmatrix}
    1 & u_1 & u_4 & u_6 \\
      & 1 & u_2 & u_5 \\
      & & 1 & u_3 \\
     & & & 1
    \end{pmatrix}
 \cdot
 \begin{pmatrix}
   & & & 1 \\
   & & 1 & \\
   & 1 & & \\
   1 & & &
   \end{pmatrix}
\cdot
    \begin{pmatrix}
    a_4 & & &  \\
    & a_3 & &  \\
    & & a_2 &  \\
    & & & a_1
    \end{pmatrix}
\cdot
   \begin{pmatrix}
    1 & x & u & w \\
      & 1 & y & v \\
      & & 1 & z \\
     & & & 1
    \end{pmatrix},
\end{aligned}
\end{equation}
where
$$
u_1= \frac{a_1(u-xy)}{a_2(xyz-xv-uz+w)},\;\;\; u_2= \frac{a_2(w-uz)}{a_3(uv-wy)},
$$
\[
u_3= - \frac{a_3 v}{a_4 w},\;\;\;u_4= \frac{a_1 x}{a_3(xyz-xv-uz+w)},
\]
and
$$ u_5=\frac{a_2(yz-v)}{a_4(uv-wy)}, \;\;\; u_6= - \frac{a_1}{a_4(xyz-xv-uz+w)}.$$

Now, using the notation above, we let
$$  \begin{pmatrix}
    a_4 & & &  \\
    & a_3 & &  \\
    & & a_2 &  \\
    & & & a_1
    \end{pmatrix}
= \widetilde{c}=
    \begin{pmatrix}
         p^{s}v_4 &  &  &  \\
          & p^{r-s}v_3 &  &  \\
          &  & p^{t-r}v_2 &  \\
          &  &  & p^{-t}v_1
   \end{pmatrix},
$$
i.e. $a_1 =p^{-t}v_1$, $a_2 =p^{t-r}v_2$, $a_3 =p^{r-s}v_3$, $a_4v=p^sv_4$. Note that $v_1v_2v_3v_4=1$ and $v_i \in \BZ_p^{\times}$ for $1 \leq i \leq 4$.

From the definition of the set $X(w_{G_4} \widetilde{c})$, we can assume that the element
$$ \begin{pmatrix}
    1 & x & u & w \\
      & 1 & y & v \\
      & & 1 & z \\
     & & & 1
    \end{pmatrix}
\in N(\BQ_p)/ N(p^m \BZ_p).$$

So we can put $x=p^{-a}x'$, $y=p^{-b}y'$, $z=p^{-c}z'$, $u=p^{-d}u'$, $v=p^{-f}v'$ and $w=p^{-e}w'$, where $x',y',z',u',v',w' \in \BZ_p^{\times}$. Moreover, $a,b,c,d,e,f$ are all integers and they satisfy $a,b,c,d,e,f \geq -m$.

From the above matrix identities in \eqref{equ:identity}, assuming that the element $g_0 \in X(w_{G_4} \widetilde{c}) \subseteq K_m$, we can deduce that

\begin{itemize}

\item [(1)]
$-\frac{a_1}{xyz-xv-uz+w} \in 1+p^m\BZ_p$, i.e.
$$\mu:=p^{t}v_1^{-1}(p^{-a-f}x'v'+p^{-c-d}u'z'-p^{-a-b-c}x'y'z'-p^{-s}w') \in 1+p^m\BZ_p.$$

\item [(2)]
$\frac{a_2(xyz-xv-uz+w)}{uv-wy} \in 1+p^m \BZ_p$, i.e.
$$\lambda:=p^{r}(v_1v_2)^{-1}(p^{-b-s}w'y'-p^{-d-f}u'v') \in 1+p^m \BZ_p.$$

\item [(3)]
$a_4 w \in 1+p^m \BZ_p$, i.e. $e=s$; $w=p^{-s}w'$ and $v_4 w' \in 1+p^m \BZ_p$.

\item [(4)]
$a_4 x, a_4 u \in p^m \BZ_p$, i.e. $\vert a_4 x \rvert \leq p^{-m}$ and $\vert a_4 u \rvert \leq p^{-m}$. This means that $a+m \leq s$ and $d+m \leq s$.

\item [(5)]
$-\frac{a_3 v}{w} \in p^m \BZ_p$, i.e.$r-f \geq m$, so $f+m \leq r$.

\item [(6)]
$\frac{a_3 (w-xv)}{w} \in p^m \BZ_p$, i.e. $$\widetilde{m}:=p^{r}(p^{-a-f}x'v'-p^{-s}w') \in p^m\BZ_p.$$ Therefore, we have $a+f \leq \max{(r,s)}$.

\item [(7)]
$\frac{a_2 (yz-v)}{uv-wy} \in p^m \BZ_p$, i.e. $$\widetilde{n}:=p^{t}(p^{-b-c}y'z'-p^{-f}v') \in p^m \BZ_p.$$ Hence we have $b+c \leq \max{(t,f)}$.

\end{itemize}

Applying Lemma \ref{lemma: stevens2} (See also Lemma 5.2 in \cite{Ste87}), we have more properties on the relevant data:


\begin{itemize}
\item [(8)]
$a_3 a_4 y \in p^m \BZ_p$, i.e. $k:=p^rv_3v_4p^{-b}y' \in p^m \BZ_p$. Therefore, we have $b+m \leq r$.

\item  [(9)]
$a_3 a_4 (xy-u) \in p^m \BZ_p$, i.e.
$$p^rv_3v_4(p^{-a-b}x'y'-p^{-d}u') \in p^m \BZ_p,\;\; \widetilde{t}:=p^{r-a-b}x'y'-p^{r-d}u' \in p^m\BZ_p.$$ Therefore, we have $a+b \leq \max{(r,d)}$.

\item  [(10)]
$a_2 a_3 a_4 z \in p^m \BZ_p$, i.e. $p^{t}v_2v_3v_4 p^{-c}z'\in p^m \BZ_p$. Hence we have $c+m \leq t$.

\end{itemize}

\begin{rmk}
The above properties (8)---(10) on $a,b,c,d,f,x',y',z',u',v',w'$ can also be deduced from the following:

By the matrix identities in \eqref{equ:identity}, we have $g_0=uw_{G_4} \tilde{c} u'$. We set 
$g_0^{\iota}:=w_{G_4} \cdot (g^t)^{-1} \cdot w_{G_4}$. Since $g_0 \in X(w_{G_4} \tilde{c})$, we have $g_0^{\iota} \in X((w_{G_4} \tilde{c})^{\iota}) = X( w_{G_4} (w_{G_4} \tilde{c}^{-1} w_{G_4}))  \subseteq K_m$. We can deduce above properties (8)---(10) from the definition of $K_m$. Actually, all the properties (1)---(10) can be achieved from Lemma \ref{lemma: stevens2} (See also Lemma 5.2 in \cite{Ste87}).
\end{rmk}

\begin{rmk}
Since $g_0:= n_1 \cdot w_{G_4} \widetilde{c} \cdot n_2 \in X(w_{G_4} \widetilde{c}) \subseteq K_m$ for some $n_1 \in N(p^m \BZ_p) \bs N(\BQ_p)$ and $n_2 \in N(\BQ_p)/ N(p^m \BZ_p)$, applying the above matrix identity \ref{equ:identity}, we can write $n_1= \begin{pmatrix}
    1 & u_1 & u_4 & u_6 \\
      & 1 & u_2 & u_5 \\
      & & 1 & u_3 \\
     & & & 1
    \end{pmatrix},$
and $n_2=  \begin{pmatrix}
    1 & x & u & w \\
      & 1 & y & v \\
      & & 1 & z \\
     & & & 1
    \end{pmatrix} = \begin{pmatrix}
                                              1 & p^{-a}x' &  p^{-d}u' &  p^{-e}w' \\
                                               & 1 & p^{-b}y' &  p^{-f}v' \\
                                               &  & 1 & p^{-c}z' \\
                                               &  &  & 1
                                            \end{pmatrix} \pmod{N(p^m \mathbb{Z}_p)}.$
We note that $n_1^{-1} \cdot g_0= w_{G_4} \widetilde{c} \cdot n_2$. By direct computation, we have
$ w_{G_4} \widetilde{c} \cdot n_2= \begin{pmatrix} &&& p^{-t}v_1 \\ && p^{t-r}v_2 & p^{t-r-c}v_2 z' \\ & p^{r-s} v_3 & p^{r-s-b} v_3 y' & p^{r-s-f} v_3 v' \\ p^s v_4 & p^{s-a} v_4 x' & p^{s-d} v_4 u' & p^{s-e} v_4 w' \end{pmatrix}.$

This is a $4 \times 4$-matrix. Let $I, J \subseteq \{1,2,3,4 \}$ be two $k$-element subsets for $1 \leq k \leq 4$. We let $g_{I,J}$ be a $k \times k$ submatrix in terms of the matrix $w_{G_4} \widetilde{c} \cdot n_2$ by picking the $k \times k$ rows and columns with the index subset $I$ and $J$. We fix $I=\{5-k,6-k,\cdots,4\}$. By Lemma \ref{lemma: stevens2}, since $g_0 \in X(w_{G_4} \widetilde{c}) \subseteq K_m$, we have $\det(g_{I,J}) \in p^m \BZ_p$ if $J \neq \{5-k,6-k,\cdots,4\}$ for every $1 \leq k \leq 4$. If $J=I= \{5-k,6-k,\cdots,4\}$, we have $\det(g_{I,J}) \in 1+p^m \BZ_p$ for every $1 \leq k \leq 4$. For example, for every $1 \leq k \leq 4$, if $\{1,\cdots, k-1\} \subseteq J$ (If $k=1$, then $\{1,\cdots, k-1\}= \varnothing $), then $s,r,t \geq m$, $-m \leq c \leq t$, $-m \leq b,f \leq r$ and $-m \leq a,d,e \leq s$. If $I=J=\{4\}$, then we have $p^{s-e} v_4 w' \in 1+p^m \BZ_p$, which gives that $s=e$, $w=p^{-s}w'$ and $v_4 w' \in 1+p^m \BZ_p$. Now, from Lemma \ref{lemma: stevens2}, all the properties (1)---(10) are given by the congruence conditions and relations $\det(g_{I,J}) \in p^m \BZ_p$ if $I \neq J$ and $\det(g_{I,J}) \in 1+p^m \BZ_p$ if $I=J$.
\end{rmk}

Conversely, if we are given integers $a,b,c,d,f$ with $a,b,c,d,f \geq -m$ and $x',y',z',u',v',w' \in \BZ_p^{\times}$ satisfying the above properties (1)---(10), there exists an elememt $x_{a,b,c,d,f}^{x',y',z',u',v',w'} \in X(w_{G_4} \tilde{c})$ (Lemma \ref{lemma: stevens2}) for which
\begin{equation}\label{eqn1: u'(x)}
  u'(x_{a,b,c,d,f}^{x',y',z',u',v',w'}) = \begin{pmatrix}
                                              1 & p^{-a}x' &  p^{-d}u' &  p^{-s}w' \\
                                               & 1 & p^{-b}y' &  p^{-f}v' \\
                                               &  & 1 & p^{-c}z' \\
                                               &  &  & 1
                                            \end{pmatrix} \pmod{N(p^m \mathbb{Z}_p)}.
\end{equation}
Moreover, we also note that
\begin{equation} \label{eqn1: u(x)}
u(x_{a,b,c,d,f}^{x',y',z',u',v',w'}) = \begin{pmatrix} 1 & u_1 & u_4 & u_6 \\ & 1 & u_2 & u_5 \\ && 1 & u_3 \\ &&& 1 \end{pmatrix} \in N(p^m \BZ_p) \bs N(\BQ_p).
\end{equation}

Applying the matrix identities in \eqref{equ:identity}. Since all the above properties (1)---(10) are satisfied, we see that the element $g_0 \in X(w_{G_4} \widetilde{c})$ from Lemma \ref{lemma: stevens2}. Hence we can pick $x_{a,b,c,d,f}^{x',y',z',u',v',w'}= g_0$.

Using above notations, we can rewrite $u_i,1 \leq i \leq 6$ as follows:
\begin{align*}
u_1&=\mu^{-1}p^{r-t}v_2^{-1}(p^{-a-b}x'y'-p^{-d}u');\\
u_2&=\lambda^{-1}p^{t-r}(v_1v_3)^{-1}(p^{s-c-d}u'z'-w');\\
u_3&= -v_3v'(v_4w')^{-1}p^{r-s-f};\\
 u_4&=-\mu^{-1}v_3^{-1}x'p^{s-r-a};\\
u_5&=\lambda^{-1}(v_1v_4)^{-1}p^{t-s}(p^{-f}v'-p^{-b-c}y'z');\\
u_6&= \mu^{-1} v_4^{-1} p^{-s}.
\end{align*}

We recall that $\psi_p$ is the nontrivial additive character of $N(\mathbb{Q}_p)$ which is trivial on $N(p^m \mathbb{Z}_p)$.
For certain $a,b,c,d,f$, and $x',y',z',u',v',w'$ satisfying the above Property (1)---(10), we let
$$
  X_{a,b,c,x',y',z'}^{d,f,u',v',w'}(w_{G_4} \widetilde{c})
              := T(1+p^m \mathbb{Z}_p)*x_{a,b,c,x',y',z'}^{d,f,u',v',w'}
$$
be the orbit through $x_{a,b,c,x',y',z'}^{d,f,u',v',w'}$, and let
$$
  S_{a,b,c,x',y',z'}^{d,f,u',v',w'}(\psi_p ; \widetilde{c},w_{G_4})
  := \sum_{x\in X_{a,b,c,x',y',z'}^{d,f,u',v',w'}(w_{G_4} \widetilde{c})}\psi_p(u(x))\psi_p(u'(x))
$$
be the Kloosterman sum restricted to the given orbit.
Now for certain fixed $a,b,c,x',y',z'$ satisfying the previous Property (1)---(10), we let
$$
  X_{a,b,c,x',y',z'}(w_{G_4} \widetilde{c}) := \bigcup\limits_{d,f,u',v',w'} X_{a,b,c,x',y',z'}^{d,f,u',v',w'}(w_{G_4} \tilde{c}),
$$
where $d,f$ run over all integers bigger than $-m$, and $u',v',w'$ run over all the elements of $\mathbb{Z}_p^\times$
satisfying Property (1)---(10). Let
$$
S_{a,b,c,x',y',z'}(\psi_p ; \widetilde{c},w_{G_4}) := \sum_{x\in X_{a,b,c,x',y',z'}(w_{G_4} \tilde{c})}\psi_p(u(x))\psi_p(u'(x)).
$$

\begin{lem}\label{lemma: X(w_{G_n} c)}
  We have $X(w_{G_4} \widetilde{c}) = \coprod_{a,b,c,x',y',z'} X_{a,b,c,x',y',z'}(w_{G_4} \widetilde{c})$, where $a,b,c$ run over all integers larger than $-m$ and smaller than $\ell$, and $x' \in \mathbb{Z}_p^{\times}/ (1+p^{m_a} \BZ_p)$, $y' \in \mathbb{Z}_p^{\times}/ (1+p^{m_b} \BZ_p)$, and $z' \in \mathbb{Z}_p^{\times}/ (1+p^{m_c} \BZ_p)$ 
satisfying Property (1)---(10). Here $m_a:=\min(m, m+a)$, $m_b:=\min(m, m+b)$, $m_c:=\min(m, m+c)$.
\end{lem}

\begin{proof}
The proof is the same as Lemma 5.2 and 5.7 in \cite{Ste87}. The union is clearly disjoint and is contained in $X(w_{G_4} \widetilde{c})$ by definition. 
Actually, from the uniqueness of the Bruhat decomposition, the map $u': X(\tau) \rightarrow N(\BQ_p)/ N(p^m \BZ_p)$ is injective. Hence the matrix $u(x)$ is uniquely determined by the matrix $u'(x)$. For every $g_0 \in X(w_{G_4} \widetilde{c})$, we write $g_0= u_1 w_{G_4} \widetilde{c} u_2$ for some $u_1 \in N(p^m \BZ_p) \bs N(\BQ_p)$ and $u_2 \in N(\BQ_p)/ N(p^m \BZ_p)$. Hnece, we can write $u'(g_0)=u_2 \pmod{N(p^m \mathbb{Z}_p)}$. We note that $u_1^{-1} \cdot g_0= w_{G_4} \widetilde{c} \cdot u_2= \tau \cdot u_2$. By Lemma \ref{lemma: stevens2} (See also \cite[Lemma 5.2]{Ste87}), it is known that $a,b,c,d,f$ and $x',y',z',u',v',w'$ satisfy the above Property (1)---(10) if we write $u_2= \begin{pmatrix}
                                              1 & p^{-a}x' &  p^{-d}u' &  p^{-s}w' \\
                                               & 1 & p^{-b}y' &  p^{-f}v' \\
                                               &  & 1 & p^{-c}z' \\
                                               &  &  & 1
                                            \end{pmatrix} \pmod{N(p^m \mathbb{Z}_p)}.$
Since
\begin{equation}
  u'(x_{a,b,c,d,f}^{x',y',z',u',v',w'}) = \begin{pmatrix}
                                              1 & p^{-a}x' &  p^{-d}u' &  p^{-s}w' \\
                                               & 1 & p^{-b}y' &  p^{-f}v' \\
                                               &  & 1 & p^{-c}z' \\
                                               &  &  & 1
                                            \end{pmatrix} \pmod{N(p^m \mathbb{Z}_p)},
\end{equation}
we know that $g_0= x_{a,b,c,d,f}^{x',y',z',u',v',w'}$ by the uniqueness of the Bruhat decomposition and $g_0$ $ \in X_{a,b,c,x',y',z'}(w_{G_4} \widetilde{c})$.
\end{proof}

\begin{rmk}
In Section 4, we know that $u'(t*x)=s \cdot u'(x) \cdot s^{-1}$ for $t \in T(1+p^m \BZ_p)$ and $s:=\tau^{-1} t \tau \in T(1+p^m \BZ_p)$. Since the map $u': X(\tau) \rightarrow N(\BQ_p)/ N(p^m \BZ_p)$ is injective, we see that the orbits in $X(\tau)$ correspond to $T(1+p^m \BZ_p)$-conjugacy classes in the coset $N(\BQ_p)/ N(p^m \BZ_p)$. Moreover, from the above injective map $u'$, the counting of the size of the Kloosterman set $X(\tau)=X(w_{G_4} \widetilde{c})$ transfers to the counting of corresponding elements in the coset $N(\BQ_p)/ N(p^m \BZ_p)$.
\end{rmk}

\begin{lem}\label{lemma:S<<w8}
  Let $\ell = \max(s,r,t) \geq m$, 
  and $a \leq s-m,\ b \leq r-m,\ c \leq t-m$ be integers which are larger than $-m$.
  Then we have the inequality
 \[
    \begin{split}
      |S_{a,b,c,x',y',z'}(\psi_p ; \tilde{c},w_{G_4})|
      & \leq
        8 \cdot p^{6m} \cdot (p^{2m} ,p^{\ell+m})^{1/2} (p^{2m} ,p^{\ell+m})^{1/2} (p^{2m} ,p^{\ell+m})^{1/2} \\
      &  \cdot (\ell+m+1)^3 \cdot
          p^{-\frac{a+b+c}{2}}  \cdot  \#(X_{a,b,c,x',y',z'}(w_{G_4} \widetilde{c})) \\
     & =
        8 \cdot p^{6m} \cdot p^m \cdot p^m \cdot p^m \\
      &  \cdot (\ell+m+1)^3 \cdot
          p^{-\frac{a+b+c}{2}}  \cdot  \#(X_{a,b,c,x',y',z'}(w_{G_4} \widetilde{c})) \\
 &  = 8 \cdot p^{9m} \cdot (\ell+m+1)^3 \cdot
          p^{-\frac{a+b+c}{2}}  \cdot  \#(X_{a,b,c,x',y',z'}(w_{G_4} \widetilde{c})).
       \end{split}
  \]
\end{lem}

\begin{proof}
  The order two involution map $\iota: g \rightarrow g^{\iota}:= w_{G_4} \cdot (g^t)^{-1} \cdot w_{G_4}$ sends $X_{a,b,c,x',y',z'}(w_{G_4} \tilde{c})$ to $X_{c,b,a,z',y',x'}((w_{G_4} \tilde{c})^\iota)$.
  Composing $\psi_p$, $\psi_p'$ with $\iota$ has the effect of replacing $\psi_p$ by $\overline{\psi_p}$ and $\psi_p'$ by $\overline{\psi_p'}$.
  For $g_0 = u_1 w_{G_4} \tilde{c} u_2 \in X(w_{G_4} \tilde{c})$, we have $g_0^{\iota} \in X((w_{G_4} \tilde{c})^{\iota}) = X( w_{G_4} (w_{G_4} \tilde{c}^{-1} w_{G_4}))  \subseteq K_m$ by definition. 
  Hence, applying $\iota$ to the element $w_{G_4} \tilde{c}$ reverses the roles of $t$ and $s$.
  Therefore we may assume that $t\geq s$ without loss of generality.

  Note that $\ell= \max(r,t)$. Property (1)---(10) imply that the matrix entries of $u(x)$ and $u'(x)$
  lie in $p^{-\ell}\mathbb{Z}_p/ p^m \mathbb{Z}_p$ for every $x\in X(w_{G_4} \widetilde{c})$.
  Indeed, by Lemma \ref{lemma: X(w_{G_n} c)}, it is enough to verify this for $x=x_{a,b,c,x',y',z'}^{d,f,u',v',w'}(w_{G_4} \widetilde{c})$. Appiying the Property (1)---(10). Since
\begin{equation*}
\begin{aligned}
  \mu &=p^{-s}  p^tv_1^{-1}( p^{s-d-c}u'z' - w') + p^{-a}  p^{t}v_1^{-1}( p^{-f}v' - p^{-b-c}y'z') \\
  &=p^{-s}\lambda v_3 u_2 p^{r}+p^{-a}v_1^{-1} \tilde{n} \in 1+p^m \mathbb{Z}_p,
\end{aligned}
\end{equation*}
and
$$ u_1=\mu^{-1}p^{-t}v_2^{-1}\widetilde{t}, \;\; u_5=-\lambda^{-1}(v_1v_4)^{-1}p^{-s} \widetilde{n}$$
for $\widetilde{n}, \widetilde{t} \in p^m \BZ_p$,
we have
$$u_1 \in p^{-t+m} \BZ_p \subseteq p^{-\ell+m} \BZ_p,\;\; u_2\in p^{-r+m}\mathbb{Z}_p \subseteq p^{-\ell+m} \BZ_p$$
and
$$u_5 \in p^{-s+m} \BZ_p \subseteq p^{-\ell+m} \BZ_p.$$
Moreover, $u_3 \in p^{-s+m} \BZ_p \subseteq p^{-\ell+m} \BZ_p$, $u_4 \in p^{-r+m} \BZ_p \subseteq p^{-\ell+m} \BZ_p$ and $u_6 \in p^{-\ell} \BZ_p^{\times}$ are directly from the properties. The claim is now easily verified.

  Now let $\mathcal{S}$ be a finite subset of $\mathbb{Z}_{\geq -m}^2\times (\mathbb{Z}_p^\times)^3$
  such that $X_{a,b,c,x',y',z'}(w_{G_4} c)$ is the disjoint union of the
  $X_{a,b,c,x',y',z'}^{d,f,u',v',w'}(w_{G_4} c)$ with $(d,f,u',v',w')\in\mathcal{S}$.
  Then as in Lemma \ref{lemma: stevens} \cite[Theorem 4.10]{Ste87}, 
  we have
  \begin{equation}\label{eqn: S decomp}
\begin{aligned}
    S_{a,b,c,x',y',z'}(\psi_p ;\tilde{c},w_{G_4})
    & <  p^{-3\ell}(1-p^{-1})^{-3} \times \\
&  \sum_{(d,f,u',v',w')\in\mathcal{S}}
    \#(X_{a,b,c,x',y',z'}^{d,f,u',v',w'}(w_{G_4} \tilde{c}))
    S_{w_{G_4}}(\theta_{a,b,c,x',y',z'}^{d,f,u',v',w'};\ell),
\end{aligned}
  \end{equation}
  where $S_{w_{G_4}}$ is defined in Definition \ref{defn: stevens} \cite[Definition 4.9]{Ste87} 
  and $\theta_{a,b,c,x',y',z'}^{d,f,u',v',w'}: A_{w_{G_4}}(\ell)\rightarrow\mathbb{C}^\times$
  is also the character defined in Definition \ref{defn: stevens} \cite[Definition 4.9]{Ste87} by
 \[
    \begin{split}
      \theta_{a,b,c,x',y',z'}^{d,f,u',v',w'}(\underline{\lambda}\times\underline{\lambda}')
      : &  =  \xi \left( u_1\lambda_1+ u_2\lambda_2+ u_3\lambda_3+ p^{-a}x'\lambda_1'+ p^{-b}y'\lambda_2'+ p^{-c}z'\lambda_3'\right) \\
      & = \xi \left(\frac{( \mu^{-1}v_2^{-1}p^{\ell+r-t}(p^{-a-b}x'y'-p^{-d}u'))\lambda_1}{p^{\ell}} \right.\\
      & \hskip 50pt \left. +\frac{ \lambda^{-1}(v_1v_3)^{-1}p^{\ell+t-r} (p^{s-c-d}u'z'-w'))\lambda_2}{p^{\ell}} \right.\\
      & \hskip 50pt \left. +\frac{( v_3v'(v_4w')^{-1} p^{\ell+r-s-f})\lambda_3}{p^{\ell}} \right.\\
      & \hskip 50pt \left. +\frac{ p^{\ell-a}x'\lambda_1' + p^{\ell-b}y'\lambda_2' + p^{\ell-c}z'\lambda_3'}{p^{\ell}}\right).
    \end{split}
  \]

  By Remark \ref{rmk:kloosterman} and Example 4.12 in \cite{Ste87}, we have
  \begin{equation}\label{eqn: S decomp to S_2}
    \begin{split}
       S_{w_{G_4}}(\theta_{a,b,c,x',y',z'}^{d,f,u',v',w'};\ell) & = S_2( \mu^{-1}v_2^{-1}p^{\ell+r-t}(p^{-a-b}x'y'-p^{-d}u'), z'p^{\ell-c};p^{\ell}) \\
      & \hskip 30pt  \cdot S_2( \lambda^{-1}(v_1v_3)^{-1}p^{\ell+t-r} (p^{s-c-d}u'z'-w'), y'p^{\ell-b};p^{\ell})\\
      & \hskip 30 pt \cdot S_2( v_3 v'(v_4w')^{-1} p^{\ell+r-s-f}, x'p^{\ell-a};p^{\ell}),
    \end{split}
  \end{equation}
  where $S_2$ is the restricted $\GL(2)$-Kloosterman sum defined in Remark \ref{rmk:kloosterman}.

  By the refined Weil's bound in Remark \ref{rmk:weil}, we have the inequality
  \begin{equation}\label{eqn: S_2}
    |S_2(\nu,\nu';p^{\ell})| \leq (\ell+m+1) \cdot  B_m \cdot (\gcd(| p^m \nu |_p^{-1},| p^m \nu'|_p^{-1},p^{\ell+m}))^{1/2} p^{(\ell+m)/2},
  \end{equation}
  for $\nu, \nu' \in p^{-m} \mathbb{Z}_p- \{0\}$ (Here we may let $B_m=p^{m/2}$. See Section 9 in \cite{KL13}).

 In order to apply the refined Weil's bound, we note that
  \[
    \begin{split}
       \gcd(| p^{\ell+m+r-s-f}|_p^{-1},| p^{\ell+m-a}|_p^{-1},p^{\ell+m}) & \leq \gcd( p^{2m} ,p^{\ell+m})p^{\ell-a}, \\
       \gcd(|  p^{\ell+m+t-r} (p^{s-c-d}u'z'-w')|_p^{-1},| p^{\ell+m-b}|_p^{-1},p^{\ell+m}) & \leq \gcd( p^{2m} ,p^{\ell+m})p^{\ell-b}, \\
       \gcd(|  p^{\ell+m+r-t}(p^{-a-b}x'y'-p^{-d}u')|_p^{-1},| p^{\ell+m-c}|_p^{-1},p^{\ell+m}) & \leq \gcd( p^{2m} ,p^{\ell+m})p^{\ell-c},
    \end{split}
  \]
since $\gcd(a,b) \leq \min(a,b)$.   Hence we have
  \begin{equation*}\label{eqn: S_w_8 bound}
  \begin{aligned}
    |S_{w_{G_4}}(\theta_{a,b,c,x',y',z'}^{d,f,u',v',w'};\ell)|
   & \leq (\ell+m+1)^3 \times  ( p^{2m} ,p^{\ell+m})^{1/2} ( p^{2m} ,p^{\ell+m})^{1/2} \\
         & \cdot  ( p^{2m} ,p^{\ell+m})^{1/2} \cdot  p^{3\ell+6m-\frac{a+b+c}{2}} \\
   & \leq (\ell+m+1)^3 \times p^{3\ell+9m-\frac{a+b+c}{2}}.
  \end{aligned}
  \end{equation*}
  This inequality, together with \eqref{eqn: S decomp}, gives
  \begin{equation}\label{eqn: S_a,b,c}
    \begin{split}
      |S_{a,b,c,x',y',z'}(\psi_p ; \tilde{c},w_{G_4})|& \leq  (p^{2m} ,p^{\ell+m})^{1/2}
      ( p^{2m} ,p^{\ell+m})^{1/2} ( p^{2m} ,p^{\ell+m})^{1/2} \\
      & \cdot (\ell+m+1)^3 \cdot (1-p^{-1})^{-3} \cdot p^{6m-\frac{a+b+c}{2}}  \sum_{(d,f,u',v',w')\in\mathcal{S}}  \#(X_{a,b,c,x',y',z'}^{d,f,u',v',w'}(w_{G_4} \widetilde{c})) \\
       & \leq (\ell+m+1)^3 \cdot (1-p^{-1})^{-3} \cdot p^{9m-\frac{a+b+c}{2}}  \sum_{(d,f,u',v',w')\in\mathcal{S}}  \#(X_{a,b,c,x',y',z'}^{d,f,u',v',w'}(w_{G_4} \widetilde{c})).
    \end{split}
  \end{equation}
  The sum appearing on the right hand side is equal to $\#(X_{a,b,c,x',y',z'}(w_{G_4} \tilde{c}))$.
  Since $p\geq 2$, we have $(1-p^{-1})^{-3}\leq 8$, by \eqref{eqn: S_a,b,c}. This completes the proof of the lemma.
\end{proof}

\begin{proof}[Proof of Theorem \ref{thm: w_8}]

  By the involution map $\iota$, we can assume that $t\geq s$ without loss of generality.
  Let
 \begin{equation}
 \begin{aligned}
    C &:=8 p^{6m} \cdot ( p^{2m} ,p^{\ell+m})^{1/2} \cdot ( p^{2m} ,p^{\ell+m})^{1/2} \cdot ( p^{2m} ,p^{\ell+m})^{1/2} \\
   & \times (\ell+m+1)^3(r+m+1)(s+m+1) \\
&= 8 p^{9m} \times (\ell+m+1)^3(r+m+1)(s+m+1).
  \end{aligned}
  \end{equation}

  At first, we deal with the case $t\geq r$.
  \begin{itemize}
  \item If $a+b+c\leq t$ and $d+f\leq r$, then we have
  \begin{align*}
  \#(d,f)&\leq (s+m+1)(r+m+1), \\
  \#(u',v',w')&\leq p^{d+s+f+3m}.
  \end{align*}
  So we have
  \begin{equation*}
\begin{aligned}
    \#(X_{a,b,c,x',y',z'}(w_{G_4} \widetilde{c})) & \leq (r+m+1)(s+m+1) \cdot p^{a+b+c+3m+d+f+s+3m}\\
& \leq (r+m+1)(s+m+1) \cdot p^{r+s+a+b+c+6m}.
\end{aligned}
  \end{equation*}
  Hence by Lemma \ref{lemma:S<<w8}, we have
  $$
     |S_{a,b,c,x',y',z'}(\psi_p ;\widetilde{c},w_{G_4})|   \leq  C  p^{r+s+t/2+3m}.
  $$
Applying the above Lemma \ref{lemma: X(w_{G_n} c)}, we have
$$ \vert Kl_p(\psi_p ;\tilde{c},w_{G_4}) \rvert \leq C(r+m+1)(s+m+1)(t+m+1)p^{r+s+t/2+3m}= C_8p^{r+s+t/2+3m}.$$
  \item If $a+b+c\leq t$ and $d+f>r$, then we assume that $d+f=r+k$, where $k\geq1$.
      Note that $d+m \leq s, f+m \leq r$, which implies that $k\leq s-2m$.
      By Property (1)---(10), we have $b+s=d+f=r+k$.
      Since $\lambda \in 1+p^m \mathbb{Z}_p$, we have $\#\{(u',v',w')\}\leq p^{d+f+(s-k)+3m}=p^{r+s+3m}$.
      Hence
      \[
        |S_{a,b,c,x',y',z'}(\psi_p ;\tilde{c},w_{G_4})|
        \leq  C  p^{r+s+\frac{a+b+c}{2}+3m}
        \leq C   p^{r+s+t/2+3m}.
      \]
  \item If $a+b+c>t$ and $d+f\leq r$, then by Property (1)---(10) and a similar argument as above, we can assume that $a+b+c=t+h$, where $h \geq 1$. We also see that $h \leq d$ or $h \leq f$. Now since $\mu \in 1+p^m \BZ_p$, we have $\#\{(u',v',w')\}\leq p^{(d-h)+f+s+3m}$.
  Hence we get that
  \[
    \begin{split}
      |S_{a,b,c,x',y',z'}(\psi_p ;\tilde{c},w_{G_4})|  & \leq  C  p^{d-h+f+s+\frac{a+b+c}{2}+3m} \\
     & =C p^{d+f+s+(t-h)/2+3m}
      \leq C   p^{r+s+t/2+3m}.
    \end{split}
  \]
  \item If $a+b+c>t$ and $d+f>r$, then using the similar argument as above we have $\#\{(u',v',w')\}\leq p^{(d-h)+f+(s-k)+3m}$.
      Hence
      \[
        \begin{split}
       |S_{a,b,c,x',y',z'}(\psi_p ;\tilde{c},w_{G_4})| & \leq C  p^{d-h+f+s-k+\frac{a+b+c}{2}+3m} \\
    & =C p^{d+f+s+(t-h)/2+3m}
        \leq C  p^{r+s+t/2+3m}.
        \end{split}
      \]
  \end{itemize}
  Note that in this case, we always have $r+s+t/2 \leq t+3s/2+r/2$.  Theorem \ref{thm: w_8} now follows from the equality
  $$Kl_p(\psi_p ;\tilde{c},w_{G_4})=\sum\limits_{a,b,c,x',y',z'}S_{a,b,c,x',y',z'}(\psi_p ;\widetilde{c},w_{G_4}).$$

  Now we handle the case $r>t$.
  By a similar argument as above, we obtain
  \[
    |S_{a,b,c,x',y',z'}(\psi_p ;\tilde{c},w_{G_4})| \leq C  p^{r+s+t/2+3m}.
  \]
  Note that if $t$ is small, this bound is not good enough to get a nontrivial upper bound for Kloosterman sums. So we have to bound this in another way.
  \begin{itemize}
   \item Assume that $f>t$, then by previous Property (1)---(10), we have $b+c=f$, and $a+f\leq r$. By Property (1)---(10),
       we have $\#(u',v')\leq p^{d+f-(a+f-t)+2m}$.
       If $d+f\leq r$, we see that

        \begin{equation*}
\begin{aligned}
         |S_{a,b,c,x',y',z'}(\psi_p ;\tilde{c},w_{G_4})|
         &\leq C p^{d+f-(a+f-t)+s+\frac{a+b+c}{2}+3m}
          \leq C p^{t+s+d+\frac{b+c}{2}+3m} \\
          &\leq C p^{t+s+\frac{d}{2}+\frac{d+f}{2}+3m}
         \leq C p^{t+3s/2+r/2+3m}.
\end{aligned}
        \end{equation*}

    \item Assume that $f>t$ and $d+f>r$, by writing $d+f=r+k$, $1\leq k\leq s+2m$, we have

            \begin{equation*}
\begin{aligned}
            |S_{a,b,c,x',y',z'}(\psi_p ;\tilde{c},w_{G_4})|  & \leq C  p^{d+f-(a+f-t)+s-k+\frac{a+b+c}{2}+3m} \\
& \leq C  p^{t+s+\frac{d}{2}+\frac{d+f}{2}-k+3m}
         \leq C  p^{t+3s/2+r/2+3m}.
\end{aligned}
            \end{equation*}
         since $\mu,\lambda \in 1+p^m \BZ_p$.

  \item  Assume that $f\leq t$ and $a+b+c>r>t$. Since $\mu\in 1+p^m \mathbb{Z}_p$, we have $\#(u',v')\leq p^{d+f-(a+b+c-t)+2m}$.   Hence by the same argument on the size of $d+f$, we have
     \[
         |S_{a,b,c,x',y',z'}(\psi_p ;\tilde{c},w_{G_4})|
         \leq C   p^{r+t+s-\frac{a+b+c}{2}+3m}
         \leq C   p^{t+s+r/2+3m}.
     \]
  \item Assume that $f\leq t$, and $a+b+c\leq r$. If $a+b+c \leq t$, then we have
    \[
    \begin{split}
      |S_{a,b,c,x',y',z'}(\psi_p;\tilde{c},w_{G_4})| & \leq C  p^{d+f+s+\frac{a+b+c}{2}+3m}  \leq C  p^{t+2s+t/2+3m}.
    \end{split}
    \]
If $t<a+b+c \leq r$, we write $a+b+c=t+h$, where $1 \leq h \leq r-t$. Since $\mu\in 1+p^m \mathbb{Z}_p$ and $h \leq d$ or $h \leq f$, we have $\#(u',v',w')\leq p^{(d+f-h)+s+3m}$. Therefore we sill have
\begin{equation*}
\begin{aligned}
 |S_{a,b,c,x',y',z'}(\psi_p ;\tilde{c},w_{G_4})|  &\leq C  p^{d+f-h+s+\frac{a+b+c}{2}+3m}  \leq C p^{d+f+s-\frac{a+b+c}{2}+t+3m} \\ &\leq C p^{d+f+s+t/2+3m} \leq C  p^{t+2s+t/2+3m}.
\end{aligned}
\end{equation*}
So in this case, we have
$$ |S_{a,b,c,x',y',z'}(\psi_p ;\tilde{c},w_{G_4})|  \leq C  p^{d+f+s+\frac{a+b+c}{2}+3m}  \leq C  p^{t+2s+t/2+3m}.$$
Note that $t+2s+t/2 \leq t+3s/2+r/2$ if and only if $s+t \leq r$, $t+2s+t/2 \leq r+s+t/2$ if and only if $s+t \leq r$, and $t+3s/2+r/2 \leq r+s+t/2$ if and only if $s+t \leq r$. Hence we have the following inequality:
$$ |S_{a,b,c,x',y',z'}(\psi_p ;\tilde{c},w_{G_4})| \leq  C \min(p^{r+s+t/2+3m},p^{t+3s/2+r/2+3m}).$$

  \end{itemize}
  This proves \eqref{eqn:Kl-w8}.

  We now give a proof of the second claim.
  If $r+\sigma+\varrho/2\leq \varrho+3\sigma/2+r/2$, i.e., $r\leq \varrho+\sigma$, then $\sigma+r\leq 3\varrho$, so $r+\sigma+\varrho/2 \leq 7(\varrho+r+\sigma)/8$. If $r+\sigma+\varrho/2 > \varrho+3\sigma/2+r/2$, i.e., $r>\varrho+\sigma$, then $4\sigma< 2r$ and $\varrho+5\sigma < 3r$, so we still have $\varrho+3\sigma/2+r/2 < 7(\varrho+r+\sigma)/8$. This proves that $\min(p^{r+\sigma+\varrho/2},p^{\varrho+3\sigma/2+r/2})\leq p^{7(t+r+s)/8}$, as claimed, and hence Theorem \ref{thm: w_8}.
\end{proof}

\begin{rmk}
Note that the trivial bound for the local Kloosterman sum in \cite{DR98} is the following:
$$ \vert Kl_p(\psi_p ;\tilde{c},w_{G_4}) \rvert = O_{\epsilon}(p^{(1+\epsilon)(t+r+s)}).$$
Since $\frac{7}{8}<1$, we get a nontrivial bound for the local Kloosterman sum (integral) by applying Stevens' method.
\end{rmk}

\begin{rmk}
The bound in the appendix is better than the bound in Section 5. The bound in Section 5 is $1-\frac{1}{4 \times 4^2-18 \times 4+22}=1-\frac{1}{14}=\frac{13}{14}$, which is larger than $1-\frac{1}{8}=\frac{7}{8}$. Applying a similar method (Stevens' approach \cite{Ste87}), such kind of non-trivial bound for $\GL(4)$ Kloosterman sums attached to the longest Weyl element $w_{G_4}$ is also achieved by Bingrong Huang with the exponent $\frac{9}{10}<1$ in the appendix of \cite{GSW21}. After a more careful estimation for the case $r>t$, we can slightly improve the bound in \cite{GSW21} by moving the exponent from $1-\frac{1}{10}=\frac{9}{10}$ to $1-\frac{1}{8}=\frac{7}{8}$.
\end{rmk}

\begin{rmk}
  The result is not optimal.
  To improve the bound in some cases, one may use the stationary phase formulas as Dabrowski and Fisher did for $\GL(3)$ (See \cite{DF97}).
\end{rmk}

\begin{rmk}
In Theorem \ref{thm: w_8} and Section 4, we only consider the special non-degenerated additive character $\psi(\sum_{i=1}^{n-1} u_{i,i+1})$, since all non-degenerated additive characters are in the same orbit under the action of diagonal matrices $T$. If we consider the general additive characters, we write the non-degenerated additive characters $\psi_p$ and $\psi_p'$ of $N(\mathbb{Q}_p)$ which are trivial on $N(p^m \mathbb{Z}_p)$ as follows:
\begin{equation} 
       \psi_p\left(\begin{pmatrix} 1&u_1&*&*\\ &1&u_2&*\\ &&1&u_3\\ &&&1\end{pmatrix}\right) = \xi( \nu_1 u_1+ \nu_2 u_2+ \nu_3 u_3),
\end{equation}
and
\begin{equation} 
      \psi_p' \left(\begin{pmatrix} 1&u_1&*&*\\ &1&u_2&*\\ &&1&u_3\\ &&&1\end{pmatrix}\right) = \xi( \nu_1' u_1+ \nu_2' u_2+ \nu_3' u_3),
\end{equation}
where $\nu_1,\nu_2,\nu_3$, $\nu_1', \nu_2',\mu_3'$ $ \in p^{-m}\mathbb{Z}_p- \{0 \}$. Moreover, we further assume that $p^{-m} \leq \vert \nu_i \rvert \leq p^{m}$ and $p^{-m} \leq \vert \nu_i' \rvert \leq p^{m}$ for all $1 \leq i, i'  \leq 3$. Here $m$ is same as the $m$ that we defined in previous Section 3 and 4. We have the following non-trivial upper bound for $\GL(4)$ generalized Kloosterman sums on non-degenerated additive characters: Let $\ell=
  \max(r,s,t) \geq m$, $\varrho=\max(t,s)$, $\sigma=\min(t,s)$, and
  \begin{equation}
\begin{aligned}
 D_8 &= 8  p^{9m} (|\nu_1\nu_3'p^{2m}|_p^{-1},p^{\ell+m})^{1/2} (|\nu_2\nu_2'p^{2m}|_p^{-1},p^{\ell+m})^{1/2} (|\nu_3\nu_1' p^{2m}|_p^{-1},p^{\ell+m})^{1/2} \\
& \cdot (\ell+m+1)^3 (\varrho+m+1)(r+m+1)^2(\sigma+m+1)^2.
\end{aligned}
  \end{equation}
Then
  \begin{equation}
    \begin{split}
      |Kl_p(\psi_p,\psi_p';\tilde{c},w_{G_4})|
      & \leq
        D_8 \cdot \min(p^{r+\sigma+\varrho/2+3m},p^{\varrho+3\sigma/2+r/2+3m}).
    \end{split}
  \end{equation}
  In particular, we have $|Kl_p(\psi_p,\psi_p';\tilde{c},w_{G_4})|\leq D_8 \cdot p^{7(t+r+s)/8+3m}.$
\end{rmk}

\end{document}